\documentclass[leqno,11pt]{amsart}
\usepackage{amsmath}
\usepackage{amsfonts}
\usepackage{amssymb}
\usepackage{amsthm}

\usepackage{mathrsfs}
\usepackage{dsfont}
\usepackage{txfonts}
\usepackage{mathabx}

\usepackage{graphicx}
\usepackage[all]{xy}
\usepackage{pgf,tikz}
\usepackage{subfigure}
\usepackage{tikz-3dplot}
\usetikzlibrary{patterns,arrows}
\tikzset{
    partial ellipse/.style args={#1:#2:#3}{
        insert path={+ (#1:#3) arc (#1:#2:#3)}
    }
}

\usepackage{hyperref}
\usepackage{color}
\usepackage[shortlabels]{enumitem}

\newtheorem{theorem}{Theorem}[section]
\newtheorem{lemma}[theorem]{Lemma}
\newtheorem{corollary}[theorem]{Corollary}

\newtheorem{proposition}[theorem]{Proposition}

\newtheorem{claim}[theorem]{Claim}

\theoremstyle{definition}
\newtheorem{definition}[theorem]{Definition}

\theoremstyle{remark}

\newcommand{\cC}{\mathcal{C}}
\newcommand{\cD}{\mathcal{D}}
\newcommand{\cP}{\mathcal{P}}
\newcommand{\cT}{\mathcal{T}}
\newcommand{\cH}{\mathcal{H}}
\newcommand{\cN}{\mathcal{N}}
\newcommand{\fH}{\mathfrak{H}}
\newcommand{\fp}{\mathfrak{p}}
\newcommand{\fL}{\mathfrak{L}}

\newcommand{\eps}{\varepsilon}

\newcommand{\abs}[1]{\left|{#1}\right|}

\title[{Noncollapsed translators in $\mathbb{R}^4$}]{Classification of noncollapsed  translators in $\mathbb{R}^4$}

\author{Kyeongsu Choi, Robert Haslhofer, Or Hershkovits}

\begin{document}

\begin{abstract}
In this paper, we  classify all noncollapsed singularity models for the mean curvature flow of 3-dimensional hypersurfaces in $\mathbb{R}^4$ or more generally in $4$-manifolds. Specifically, we prove that every noncollapsed translating hypersurface in $\mathbb{R}^4$ is either $\mathbb{R}\times$2d-bowl, or a 3d round bowl, or belongs to the one-parameter family of 3d oval bowls constructed by Hoffman-Ilmanen-Martin-White.
\end{abstract}

\maketitle

\tableofcontents

\section{Introduction}

A hypersurface $M^n\subset \mathbb{R}^{n+1}$ is called a \emph{translator} if its mean curvature vector satisfies
\begin{equation}\label{eq_translator}
\mathbf{H}=v^\perp
\end{equation}
for some $0\neq v\in\mathbb{R}^{n+1}$. Solutions of \eqref{eq_translator} correspond to selfsimilarly translating solutions $\{M_t=M+tv\}_{t\in\mathbb{R}}$ of the mean curvature flow,
\begin{equation}
(\partial_t x)^\perp = \mathbf{H}(x).
\end{equation}
Translators model the formation of type II singularities under mean curvature flow, see e.g. \cite{Hamilton_Harnack,HuiskenSinestrari_convexity,White_nature}. We recall that Huisken and Hamilton grouped singularities of the mean curvature flow at some time $T$ into type I and II, depending on whether $(T-t)|A|^2$ stays bounded or not \cite{Huisken_monotonicity,Hamilton_Harnack}. Type I singularities are modelled on \emph{shrinkers}, and are easier to analyze than type II singularities. For example it is known in any dimension that the round cylinders $\mathbb{R}^{k}\times S^{n-k}$ are the only mean-convex shrinkers \cite{Huisken_shrinker,White_nature}, and also the only stable shrinkers \cite{CM_generic}. In an attempt to get a grasp on type II singularities, translators have received a lot of attention over the last 25 years, but despite these efforts no general classification result has been obtained for $n\geq 3$, not even for convex graphs.\\

For $n=2$, there is by now a very precise understanding of translators. Altschuler-Wu \cite{AltschulerWu} constructed a translator that is the graph of an entire rotationally invariant function, called the \emph{bowl}.
In \cite{Wang_convex}, Wang proved that the bowl is the unique (up to rigid motions and scaling) convex translator in $\mathbb{R}^3$ that is an entire graph. More recently, a complete classification of graphical translators in $\mathbb{R}^3$ has been obtained by Hoffman-Ilmanen-Martin-White \cite{HIMW}, building on important prior work of Spruck-Xiao \cite{SpruckXiao}. Namely, they proved that any such translator is either a bowl, or a grim reaper surface, or belongs to the one-parameter family of $\Delta$-wings discovered by Ilmanen. See \cite{Nguyen_trident}, \cite{HMW_tridents} and \cite{HMW} for other examples of translators, and \cite{HIMW_translators_survey} for a survey article about translators in $\mathbb{R}^3$.  See also \cite{CCK} for a recent classification of translators of the $\alpha$-Gauss curvature flow in $\mathbb{R}^3$.\\

For $n\geq 3$, in his pioneering work \cite{Wang_convex}, Wang constructed graphical convex translators that are not rotationally symmetric, addressing a conjecture of White \cite{White_nature}. The only instances for $n\geq 3$ where some classification has been obtained are the uniformly $2$-convex case \cite{Haslhofer_bowl,BourniLangford,SpruckSun} and the case of solutions contained in strip regions \cite{BLT}, which both  very much behave like the $2$-dimensional case.\\

\subsection{Main results}

In the present paper, we address the classification problem for translators in $\mathbb{R}^4$. We focus on the situation most relevant for singularity analysis, namely the noncollapsed case. We recall that a hypersurface $M$ is called \emph{noncollapsed} if it has positive mean curvature and there is some $\alpha>0$ such that at every point $p\in M$ the inscribed radius and exterior radius is at least $\alpha/H(p)$, see \cite{ShengWang,Andrews_noncollapsing,HaslhoferKleiner_meanconvex}. It is known since the work of White \cite{White_size,White_nature} that all blowup limits of any mean-convex mean curvature flow are noncollapsed. In fact, one can take $\alpha=1$, see \cite{Brendle_inscribed,HK_inscribed}. More generally, by Ilmanen's mean-convex neighborhood conjecture \cite{Ilmanen_problems}, which has been proved recently in the case of neck-singularities in \cite{CHH,CHHW}, it is expected even without mean-convexity assumption that all  blowup limits near any cylindrical singularity are ancient noncollapsed flows.\\

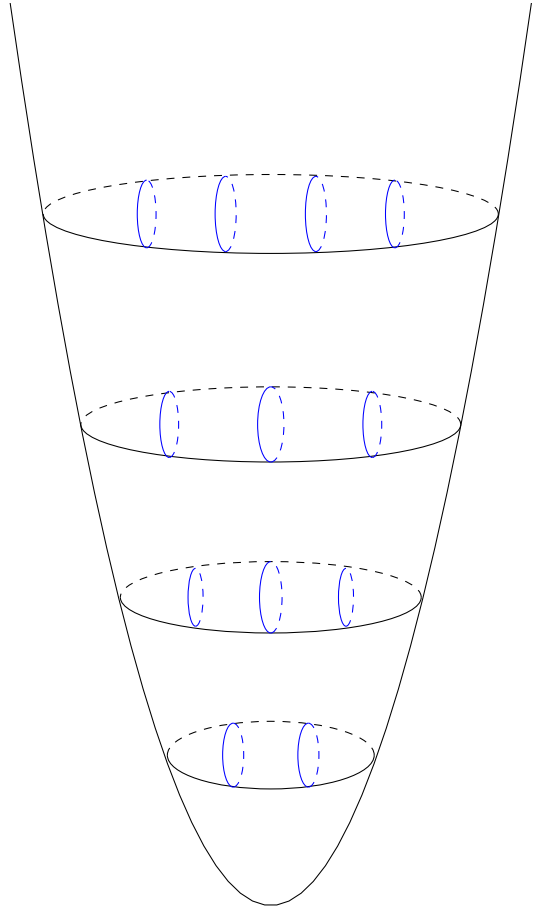
\begin{figure}
\begin{tikzpicture}[x=1cm,y=1cm] \clip(-4,12) rectangle (4,-1);
\draw [samples=100,rotate around={0:(0,0)},xshift=0cm,yshift=0cm,domain=-8:8)] plot (\x,{(\x)^2});
\draw [dashed] (0,2) [partial ellipse=0:180:2.75*0.5cm and 0.9*0.5cm];
\draw (0,2) [partial ellipse=180:360:2.75*0.5cm and 0.9*0.5cm];
\draw [dashed] (0,4.1) [partial ellipse=0:180:4*0.5cm and 0.95*0.5cm];
\draw (0,4.1) [partial ellipse=180:360:4*0.5cm and 0.95*0.5cm];
\draw [dashed] (0,6.4) [partial ellipse=0:180:5.05*0.5cm and 1*0.5cm];
\draw (0,6.4) [partial ellipse=180:360:5.05*0.5cm and 1*0.5cm];
\draw [dashed] (0,9.2) [partial ellipse=0:180:6.05*0.5cm and 1.05*0.5cm];
\draw (0,9.2) [partial ellipse=180:360:6.05*0.5cm and 1.05*0.5cm];
\draw [color=blue] (0,4.1) [partial ellipse=90:270:0.3*0.5cm and 0.95*0.5cm];
\draw [color=blue,dashed] (0,4.1) [partial ellipse=-90:90:0.3*0.5cm and 0.95*0.5cm];
\draw [color=blue] (-2*0.5,4.1) [partial ellipse=90:270:0.2*0.5cm and 0.77*0.5cm];
\draw [color=blue,dashed] (-2*0.5,4.1) [partial ellipse=-90:90:0.2*0.5cm and 0.77*0.5cm];
\draw [color=blue] (2*0.5,4.1) [partial ellipse=90:270:0.2*0.5cm and 0.77*0.5cm];
\draw [color=blue,dashed] (2*0.5,4.1) [partial ellipse=-90:90:0.2*0.5cm and 0.77*0.5cm];
\draw [color=blue] (0,6.4) [partial ellipse=90:270:0.35*0.5cm and 1*0.5cm];
\draw [color=blue,dashed] (0,6.4) [partial ellipse=-90:90:0.35*0.5cm and 1*0.5cm];
\draw [color=blue] (-2.7*0.5,6.4) [partial ellipse=90:270:0.25*0.5cm and 0.88*0.5cm];
\draw [color=blue,dashed] (-2.7*0.5,6.4) [partial ellipse=-90:90:0.25*0.5cm and 0.88*0.5cm];
\draw [color=blue] (2.7*0.5,6.4) [partial ellipse=90:270:0.25*0.5cm and 0.88*0.5cm];
\draw [color=blue,dashed] (2.7*0.5,6.4) [partial ellipse=-90:90:0.25*0.5cm and 0.88*0.5cm];
\draw [color=blue] (-1*0.5,2) [partial ellipse=90:270:0.28*0.5cm and 0.85*0.5cm];
\draw [color=blue,dashed] (-1*0.5,2) [partial ellipse=-90:90:0.28*0.5cm and 0.85*0.5cm];
\draw [color=blue] (1*0.5,2) [partial ellipse=90:270:0.28*0.5cm and 0.85*0.5cm];
\draw [color=blue,dashed] (1*0.5,2) [partial ellipse=-90:90:0.28*0.5cm and 0.85*0.5cm];
\draw [color=blue] (-1.2*0.5,9.2) [partial ellipse=90:270:0.28*0.5cm and 1*0.5cm];
\draw [color=blue,dashed] (-1.2*0.5,9.2) [partial ellipse=-90:90:0.28*0.5cm and 1*0.5cm];
\draw [color=blue] (1.2*0.5,9.2) [partial ellipse=90:270:0.28*0.5cm and 1*0.5cm];
\draw [color=blue,dashed] (1.2*0.5,9.2) [partial ellipse=-90:90:0.28*0.5cm and 1*0.5cm];
\draw [color=blue] (-3.3*0.5,9.2) [partial ellipse=90:270:0.25*0.5cm and 0.9*0.5cm];
\draw [color=blue,dashed] (-3.3*0.5,9.2) [partial ellipse=-90:90:0.25*0.5cm and 0.9*0.5cm];
\draw [color=blue] (3.3*0.5,9.2) [partial ellipse=90:270:0.25*0.5cm and 0.9*0.5cm];
\draw [color=blue,dashed] (3.3*0.5,9.2) [partial ellipse=-90:90:0.25*0.5cm and 0.9*0.5cm];
\end{tikzpicture}
\caption{The oval bowls $M_k$ are $3$-dimensional translating hypersurfaces in $\mathbb{R}^4$, whose level sets look like $2$-dimensional ovals in $\mathbb{R}^3$. It is a one-parameter family of translators, whose principal curvatures at the tip are $(k,\tfrac{1-k}{2},\tfrac{1-k}{2})$. For $k=1/3$ it is the round bowl (with round spherical level sets), while for $k\to 0$ one has convergence to $\mathbb{R}\times$2d-bowl.}\label{figure_oval_bowls}
\end{figure}

Let us first review the known examples of noncollapsed translators in $\mathbb{R}^4$:
Two examples that have been known for quite a while are $\mathbb{R}\times \mathrm{Bowl}_2$ -  the product of the line with the 2-dimensional bowl from  from Altschuler-Wu \cite{AltschulerWu}, and $\mathrm{Bowl}_3$ - the 3d \emph{round bowl} constructed by Clutterbuck-Schn\"urer-Schulze \cite{CSS}. More recently, Hoffman-Ilmanen-Martin-White \cite{HIMW} constructed examples that are not rotationally symmetric. Specifically, for every triple $(k_1,k_2,k_3)$ of nonnegative numbers with $k_1+k_2+k_3 =1$ they proved that there exists at least one unit-speed graphical translator with tip principal curvatures $(k_1,k_2,k_3)$. Moreover, they showed that when one takes $k_1\leq k_2=k_3$ then one always gets a translator that is an entire graph and has circular symmetry in the last two variables. It is not hard to show that these entire graphical translators are in fact noncollapsed (see Theorem \ref{thm_non_collapsed}). Hence, for every $k\in (0,\tfrac{1}{3})$ there exists at least one noncollapsed translators $M_k\subset\mathbb{R}^4$ that is noncollapsed and circular symmetric and whose principal curvatures at the tip are $(k,\tfrac{1-k}{2},\tfrac{1-k}{2})$.  The HIMW-translators $\{M_k\}_{k\in (0,1/3)}$ interpolate between $M_0=\mathbb{R}\times \mathrm{Bowl}_2$ and $M_{1/3}=\mathrm{Bowl}_3$. Furthermore, as we will see later, the HIMW-translators have oval level sets, as illustrated in Figure \ref{figure_oval_bowls}, and we thus refer to them as the \emph{oval bowls}. \\

Our main classification theorem shows that any noncollapsed translators in $\mathbb{R}^4$ in fact must be equal, up to rigid motion and scaling, to one of the examples from the literature that we reviewed above:

\begin{theorem}[classification of noncollapsed translators]\label{thm_main}
Every noncollapsed translator in $\mathbb{R}^4$ is, up to rigid motion and scaling,
\begin{itemize}
\item either $\mathbb{R}\times \mathrm{Bowl}_2$,
\item or the 3d round bowl $\mathrm{Bowl}_3$,
\item or belongs to the one-parameter family of 3d oval bowls $\{M_k\}_{k\in (0,1/3)}$ constructed by Hoffman-Ilmanen-Martin-White.
\end{itemize}
\end{theorem}

In particular, our main theorem provides a complete classification of singularity models for the mean curvature flow of embedded mean-convex hypersurfaces in $\mathbb{R}^4$ or more generally also in 4-manifolds (observe that even for general ambient 4-manifolds the blowup limits always live in Euclidean space). To discuss this, recall that for mean-convex flows all blowup limits are noncollapsed and convex \cite{White_size,White_nature,HuiskenSinestrari_convexity,HaslhoferKleiner_meanconvex}. In particular, for type I singularities one can always pass to a type I blowup limit that is a shrinker by Huisken's monotonicity formula \cite{Huisken_monotonicity}, while for type II singularities one can always pass to a type II  blowup limit that is a translator by Hamilton's Harnack inequality \cite{Hamilton_Harnack}.

\begin{corollary}[{classification of singularity models}]\label{cor_sing}
For the mean curvature flow of closed embedded mean-convex hypersurfaces in $\mathbb{R}^4$ (or more generally in a 4-manifold), every type I blowup limit (ala Huisken) is
\begin{itemize}
\item either a round shrinking $S^3$,
\item or a round shrinking $\mathbb{R}\times S^2$,
\item or a round shrinking $\mathbb{R}^2\times S^1$,
\end{itemize}
and every type II blowup limit (ala Hamilton) is
\begin{itemize}
\item either $\mathbb{R}\times \mathrm{Bowl}_2$,
\item or the 3d round bowl $\mathrm{Bowl}_3$,
\item or belongs to the one-parameter family of 3d oval bowls $\{M_k\}_{k\in (0,1/3)}$ constructed by Hoffman-Ilmanen-Martin-White.
\end{itemize}
\end{corollary}

In particular, our corollary seems to be the first general classification result of singularity models in higher dimensions.  Recall that while singularities for mean curvature flow in $\mathbb{R}^3$ and for three-dimensional Ricci flow are by now well understood, the classification of singularities in higher dimensions, without special assumptions such as two-convexity or positive isotropic curvature, is widely open.\\

Our main classification result is related to a recent breakthrough by Angenent-Daskalopoulos-Sesum \cite{ADS1,ADS2}, who proved that every compact ancient noncollapsed flow in $\mathbb{R}^3$ (or more generally in $\mathbb{R}^{n+1}$ assuming uniform $2$-convexity) is either a round shrinking sphere or an ancient oval. The ancient ovals, whose existence has been proved in \cite{White_nature,HaslhoferHershkovits_ancient}, are compact ancient solutions that for $t\to 0$ converge to a round point, but for $t\to -\infty$ look very oval, namely like a cylinder with two bowl-like caps.\\

Let us now discuss some major challenges that arise in establishing our main classification result:\\

First, the round bowl has a neck-tangent flow at $-\infty$, i.e.
\begin{equation}\label{neck_tangent}
\lim_{\lambda \rightarrow 0} \lambda M_{\lambda^{-2}t}=\mathbb{R}\times S^2(\sqrt{-4t}),
\end{equation}
but oval bowls have a bubble-sheet tangent flow at $-\infty$, i.e.
\begin{equation}\label{bubble_tangent_intro}
\lim_{\lambda \rightarrow 0} \lambda M_{\lambda^{-2}t}=\mathbb{R}^2\times S^1(\sqrt{-2t}).
\end{equation}
While the case of neck-singularities has been analyzed extensively over the last 20 years culminating in the recent classification from  \cite{ADS1,ADS2,BC,BC2,CHH,CHHW} (see also \cite{Brendle_Bryant,ABDS,BDS,LiZhang,BN,BDNS} for corresponding classification results for the Ricci flow), the classification of bubble-sheet singularities up to now seemed to be a problem out of reach.\\

Second, the classification of ancient ovals from the recent breakthrough by Angenent-Daskalopoulos-Sesum \cite{ADS1,ADS2} crucially relies on the property that eventually all such ovals agree up to rigid motion and scaling. In contrast, the examples from \cite{HIMW} for different values of $k$ are genuinely distinct, and furthermore it is not known a-priori whether or not the HIMW-family is unique and depends continuously on $k$. Even though this may sound like a more technical point, this actually causes the following fundamental issue: In the spectral analysis one cannot kill the neutral and unstable modes in any straightforward way.\\

Third, the classification of round bowl and ancient ovals crucially relies on the fact that they are rotationally symmetric. In contrast, the oval bowls from \cite{HIMW} are only $\mathrm{SO}(2)$-symmetric but not $\mathrm{SO}(3)$-symmetric. In particular, this increases the number of independent variables, and thus precludes the direct use of ODE techniques or techniques from 1+1 dimensional parabolic equations.\\

\bigskip

\subsection{Key results and outline of the proofs} Let us now outline the main steps of our argument. 
Roughly speaking, our main classification result will follow by combining the following five key results:
\begin{itemize}
\item Theorem \ref{thm_blowdown} (blowdown and circular symmetry),
\item Theorem \ref{thm_unique_asympt_intro} (uniform sharp asymptotics),
\item Theorem \ref{thm:uniqueness_eccentricity_intro} (spectral uniqueness),
\item Theorem \ref{all_achived_thm_intro} (existence with prescribed eccentricity),
\item Theorem \ref{thm_mon_analyt} (monotonicity and analyticity).
\end{itemize}
We will now discuss these five key results in turn. For the rest of this outline, we denote by $M=\partial K$ any noncollapsed translator in $\mathbb{R}^4$, where we normalize without loss of generality such that it translates with unit speed in positive $x_1$-direction. Assuming that $M$ is neither $\mathbb{R}\times$2d-bowl nor a 3d round bowl, the ultimate goal is to show that it is an oval bowl $M_k$, and is uniquely determined by the tip curvature $k$.\\

In Section \ref{sec_coarse_asymptotics}, we discuss coarse asymptotics and circular symmetry. The key to get started is:

\begin{theorem}[{blowdown and circular symmetry, c.f. \cite{CHH_blowdown,Zhu}}]\label{thm_blowdown}
The blowdown of $M=\partial K$ is always a halfline, more precisely
\begin{equation}
\lim_{\lambda\to 0}\lambda K = \{ \lambda e_1\, | \, \lambda\geq 0\}.
\end{equation}
In particular, $M$ has a unique tip point and is $\mathrm{SO}$(2)-symmetric.
\end{theorem}

The result about the blowdown has already been established in our previous paper \cite{CHH_blowdown}. To prove this we had to rule out the potential scenario of noncollapsed wing-like translators, which we did via fine-bubble sheet analysis. In particular, the blowdown directly yields the existence of a unique tip point where $x_1$ is minimized. It then follows from a recent result by Zhu \cite{Zhu} that $M$ is $\mathrm{SO}$(2)-symmetric. Zhu's proof was based on a bubble-sheet version of the Brendle-Choi neck improvement theorem \cite{BC,BC2}. Exploiting the fact that the blowdown is a halfline, we found a shorter proof of Zhu's result, which is based instead on methods from \cite{Brendle_steady,Haslhofer_bowl} and which we include for convenience of the reader.\\

Theorem \ref{thm_blowdown} also yields further important information about the coarse asymptotics of the level sets
\begin{equation}
\Sigma^h:= M\cap \{x_1=h\}.
\end{equation}
Exploiting the more quantitate information from the proof we show that for every $\delta>0$ we have
\begin{equation}\label{eq_diam_intro}
 \lim_{h\to \infty} \frac{\textrm{diam}(\Sigma^h)}{h^{1/2+\delta}}=0.
\end{equation}
Moreover, using the vanishing asymptotic slope property, which follows again from Theorem \ref{thm_blowdown}, we show that the level sets move almost like a mean curvature flow of surfaces in $\mathbb{R}^3$. Namely, we show that 
\begin{equation}\label{mean_curv_exp}
|H-H^h| \leq CH^3,
\end{equation}
where $H$ is the mean curvature of $M$, and $H^h$ is the mean curvature of $\Sigma^h\subset\{x_1=h\}$.\\

In Section \ref{sec_sharp}, we establish uniform sharp asymptotics. Loosely speaking, our result shows that the level sets $M\cap \{x_1=-t\}$ have the same sharp asymptotics as the ones from Angenent-Daskalopoulos-Sesum \cite{ADS1} for the 2-dimensional ancient ovals in $\mathbb{R}^3$, and moreover these sharp asymptotics hold uniformly for certain families of translators. In more detail, we establish uniform sharp asymptotics for the profile function of the level sets. Specifically, assuming without loss of generality that the $\textrm{SO}(2)$-symmetry from above is in the $x_3x_4$-plane centered at the origin, we can express the level sets as
\begin{equation}
\Sigma^{-t}=\left\{ (-t,x,x_3,x_4)\in\mathbb{R}^4 : -d^-(t)\leq x \leq d^+(t),\, (x_3^2+x_4^2)^{1/2} = V(x,t) \right\}.
\end{equation}
The profile function $V(x,t)$ is defined for all $t\ll 0$ and all $x$ on a maximal interval $[-d^-(t),d^+(t)]$. We also consider the renormalized profile function $v$ defined by
\begin{equation}
v(y,\tau)=e^{\tau/2}V(e^{-\tau/2}y,-e^{-\tau}).
\end{equation}
Moreover, in the tip regions we define $Y(\cdot,\tau)$ as the inverse function of $v(\cdot,\tau)$, and let
\begin{equation}
Z(\rho,\tau)= |\tau|^{1/2}\left(Y(|\tau|^{-1/2}\rho,\tau)-Y(0,\tau)\right).
\end{equation}
As we will see, in the central region there is an inwards quadratic bending of the form
\begin{equation}\label{inwards_bending_intro}
v(y, \tau)=\sqrt{2}-\frac{y^2-2}{\sqrt{8} |\tau|} + o(|\tau|^{-1}).
\end{equation}
It will be crucial that our uniform sharp asymptotics in all regions hold for all times where the function $v$ behaves approximately like \eqref{inwards_bending_intro} in an Gaussian $L^2$-sense. To describe this, let us discuss some background and notation. The evolution of $v$ is governed by the one-dimensional Ornstein-Uhlenbeck operator
\begin{equation}
\fL=\partial_y^2-\tfrac{y}{2}\partial_y+1.
\end{equation}
Recall that $\fL$ is a self-adjoint operator on the Hilbert space $\fH:=L^2(\mathbb{R},e^{-y^2/4} dy)$, 
and that
\begin{equation}
\fH = \fH_+\oplus \fH_0 \oplus \fH_-,
\end{equation}
where $\fH_+$ is spanned by the unstable eigenfunctions $\psi_1=1$ and $\psi_2=y$, and $\fH_0$ is spanned by the neutral eigenfunction $\psi_0=y^2-2$. We write $\fp_\pm$ and $\fp_0$ for the orthogonal projections on $\fH_\pm$ and $\fH_0$. Moreover, we fix a small constant $\theta>0$, and consider the cylindrical profile function
\begin{equation}\label{localization}
v_{\cC}=\varphi_{\cC}(v)v,
\end{equation}
where $\varphi_{\cC}$ is a suitable cutoff function that localizes in the cylindrical region $\mathcal{C}=\{ v \geq \tfrac58 \theta\}$. Finally, given any $\tau_0\ll 0$ after a suitable shift in the $x_1x_2$-plane we can assume that
\begin{equation}\label{eq_centering_intro}
\fp_{+}(v_{\cC}(\tau_0)-\sqrt{2})=0.
\end{equation}

\begin{definition}[$\kappa$-quadratic]\label{def_mu_qudratic_intro}
We say that $M$ (normalized as above and centered as in \eqref{eq_centering_intro}) is \emph{$\kappa$-quadratic at time $\tau_{0}$} if its cylindrical profile function $v_{\cC}$ satisfies
\begin{equation}\label{mu_quad_back}
\left\|v_{\cC}(y,\tau_{0})-\sqrt{2}+\frac{y^2-2}{\sqrt{8}|\tau_{0}|}\right\|_{\fH}\leq \frac{\kappa}{|\tau_{0}|},
\end{equation}
and for every $\tau\in[2\tau_0,\tau_0]$ the renormalized hypersurface $\bar{M}_{\tau}=e^{-\tau/2}M_{-e^{-\tau}}$ can be expressed locally as a graph of a function $u(y_1,y_2,\tau)$ over the cylinder $\mathbb{R}^2\times S^1(\sqrt{2})$ with the estimate
\begin{equation}\label{techn_graph_rad}
\sup_{\tau\in [2\tau_0,\tau_0]}\|u(\cdot,\tau)\|_{C^4(B(0,2|\tau_0|^{1/100}))} \leq |\tau_0|^{-1/50}.
\end{equation}
\end{definition}

Here, the small parameter $\kappa>0$ measures the deviation from \eqref{inwards_bending_intro} in the Gaussian $L^2$-norm. The condition involving the bubble-sheet function $u$ is more technical and can be ignored at first reading.\\

Using these notions, we can now precisely state our uniform sharp asymptotics:

\begin{theorem}[uniform sharp asymptotics]\label{thm_unique_asympt_intro}
For every $\eps>0$ there exists $\kappa>0$ and $\tau_{\ast}>-\infty$, such that if $M$ is $\kappa$-quadratic at time $\tau_{0}$ for some $\tau_0 \leq \tau_{\ast}$, then for every $\tau \leq \tau_{0}$ the following holds:
\begin{enumerate}
\item Parabolic region: The renormalized profile function satisfies
\begin{equation}
\left| v(y, \tau)-\sqrt{2}\left(1-\frac{y^2-2}{4 |\tau|}\right) \right| \leq\frac{\eps}{|\tau|} \qquad (|y |\leq \eps^{-1}).
\end{equation}
\item Intermediate region: The function $\bar{v}(z,\tau):=v(|\tau|^{1/2}z,\tau)$ satisfies
\begin{equation}
|\bar{v}(z,\tau)-\sqrt{2-z^2}|\leq \eps,
\end{equation}
on $[-\sqrt{2}+\eps,\sqrt{2}-\eps]$.
\item Tip regions: We have the estimate
\begin{equation}
\| Z(\cdot,\tau)-Z_0(\cdot)\|_{C^{100}(B(0,\eps^{-1}))}\leq \eps,
\end{equation}
where $Z_0(\rho)$ is the profile function of the $2d$-bowl with speed $1/\sqrt{2}$.
\end{enumerate}
Moreover, for every $\tau\leq \tau_0$ the renormalized hypersurface $\bar{M}_{\tau}=e^{-\tau/2}M_{-e^{-\tau}}$ can be expressed locally as a graph of a function $u(y_1,y_2,\tau)$ over the cylinder $\mathbb{R}^2\times S^1(\sqrt{2})$ with the estimate
\begin{equation}
\| u \|_{{C^4}(B(0,2|\tau|^{1/10})}\leq |\tau|^{-1/5}.
\end{equation}
Finally, given any $\kappa>0$, after suitable recentering every $M$ is $\kappa$-quadratic at time $\tau_{0}$, provided $\tau_0=\tau_0(M,\kappa)>-\infty$ is sufficiently negative.
\end{theorem}

In particular, the uniform sharp asymptotics imply that the level sets $\Sigma^h$ satisfy the estimate
\begin{equation}
\left| \frac{d^\pm(h)}{\sqrt{2h\log h}}-1 \right| \leq \eps.
\end{equation}
Also, while we initially only assumed that we have the graphical radius $|\tau|^{1/100}$ for $2\tau_0\leq\tau\leq \tau_0$, the theorem shows that we actually get the improved graphical radius $|\tau|^{1/10}$ for all $\tau\leq \tau_0$.\\

To prove the uniform sharp asymptotics, we carry out a fine bubble-sheet analysis, which generalizes the fine neck analysis from \cite{ADS1}. Roughly speaking, this can be done by   carefully analyzing the evolution of  $u(y_1,y_2,\tau)$, which is governed by the two-dimensional Ornstein-Uhlenbeck operator 
\begin{equation}
\mathcal L=\partial_{y_1}^2+\partial_{y_2}^2-\tfrac{y_1}{2}\partial_{y_1}-\tfrac{y_2}{2}\partial_{y_2} +1.
\end{equation}
The most challenging part is to establish that the estimates are in fact uniform for all $M$ that are $\kappa$-quadratic at time $\tau_0$. To this end, remembering Definition \ref{def_mu_qudratic_intro} ($\kappa$-quadratic) we have to (i) upgrade information at the single time $\tau_0$ to information for all $\tau\leq \tau_0$, and (ii) upgrade information about profile function $v(y,\tau)$ in the Hilbert space $\fH$ to information about the bubble-sheet graph function $u(y_1,y_2,\tau)$ in the larger Hilbert space $\mathcal{H}\cong \fH\otimes \fH$. To accomplish (i) we use Merle-Zaag type arguments. To accomplish (ii) we exploit the fact that $|\partial_{y_1}u |$ is exponentially small on our bubble-sheet thanks to the translator equation.

\bigskip

Our next key result says that noncollapsed translators in $\mathbb{R}^4$ are uniquely characterized by the spectral projection of their cylindrical profile function to the unstable and neutral space:

\begin{theorem}[spectral uniqueness]\label{thm:uniqueness_eccentricity_intro} There exist $\kappa>0$ and $\tau_{\ast}>-\infty$ with the following significance:
If $M^1$ and $M^2$ are noncollapsed translators in $\mathbb{R}^4$ (normalized and centered as before) that are $\kappa$-quadratic at time $\tau_0$, where $\tau_0 \leq \tau_{\ast}$, and if their cylindrical profile functions $v^1_{\cC}$ and $v^2_{\cC}$ satisfy
\begin{equation}\label{spec_cent_intro}
\fp_{+}(v^1_{\cC}(\tau_0)-v^2_{\cC}(\tau_0))=0\;\;\;\;\textrm{(equal spectral center)},
\end{equation}
and
\begin{equation}\label{spec_ecc_intro}
\fp_{0}(v^1_{\cC}(\tau_0)-v^2_{\cC}(\tau_0))=0\;\;\;\; \textrm{(equal spectral eccentricity)},
\end{equation}
then
\begin{equation}
M^1=M^2.
\end{equation} 
\end{theorem}

The statement of Theorem \ref{thm:uniqueness_eccentricity_intro} (spectral uniqueness) is similar to the main technical result of \cite{ADS2}. Some important technical differences are that Theorem \ref{thm:uniqueness_eccentricity_intro} is uniform across all $\kappa$-quadratic solutions and that instead of simply truncating the difference of profile functions, we use the intrinsic localization \eqref{localization}. This is crucial to ensure that having equal spectrum is manifestly an equivalence relation.\\

The biggest difference, however, is how these theorems can be applied. For the ancient ovals it was shown in \cite[Section 4]{ADS2} that by a suitable rigid motion and scaling one can always arrange that the truncated difference of the profile functions satisfies the conditions $\fp_{+}( w_{\cC}(\tau_0))=0$  and $\fp_{0} (w_{\cC}(\tau_0))=0$. This of course was only possible since the ancient ovals are -- at the end of the day -- unique up to rigid motion and scaling. In contrast, the HIMW translators are a genuinely distinct one-parameter family of solutions. While an easy shift in the $x_1x_2$-plane still allows us to impose our usual centering condition \eqref{eq_centering_intro}, which in particular implies \eqref{spec_cent_intro}, dealing with the spectral eccentricity is far more subtle. In particular, since it is not known a priori whether or not the HIMW family is unique and continuous, while all tip curvatures $k$ are realized, it is highly nonobvious whether or not all spectral eccentricities are realized.\\

In Section \ref{sec_HIMW}, we overcome the above difficulties and complete the proof of the main classification theorem, modulo the proof of the spectral uniqueness theorem, which will be proven in the last section. A key point is to show that the Hoffman-Ilmanen-Martin-White construction in fact realizes all eccentricities. To describe this, recall from \cite{HIMW}  that for every ellipsoidal parameter $a\in [0,\frac{1}{3}]$ and every height $h<\infty$, there exists an $\mathrm{SO}(2)$-symmetric translator-with-boundary $M^{a,h}$, with tip at the origin  and whose boundary lies at height $x_1=h$ and is an ellipse of the form
$a^2 x_2^2 + (\tfrac{1-a}{2})^2 x_3^2 + (\tfrac{1-a}{2})^2 x_4^2= R^2$,
where $R=R(a,h)$. We then define the \emph{HIMW class} $\mathcal{A}$ as the collection of all possible limits, namely\footnote{A priori this slightly generalizes the construction from Hoffman-Ilmanen-Martin-White, but a posteriori it will be equivalent.}
\begin{equation}
\mathcal{A}:= \left\{ \lim_{i\to \infty} M^{a_i,h_i} \, |\,  a_i \in [0,1/3] \textrm{ and } h_i\to \infty \right\}.
\end{equation}
We first establish some basic properties of this class and show that every member of $\mathcal{A}$ is noncollpased. Hence, the above results apply to the class $\mathcal{A}$. Also, given any $\tau_0$, it is easy to see that there is a unique shift in $x_1$-direction such that our centering condition \eqref{eq_centering_intro} holds. We denote this shifted class by $\mathcal{A}'$.\\
We then consider the eccentricity map
\begin{equation}
\mathcal{E}:\mathcal{A}' \rightarrow \mathbb{R}, \quad M \mapsto  \langle v^M_{\mathcal{C}}(\tau_0),2-y^2 \rangle_{\fH} .
\end{equation}
Observe that the expected value of $\mathcal{E}$ for translators satisfying the sharp asymptotics at time $\tau_0$ is
\begin{equation}
e_0 =\frac{4\sqrt{2\pi}}{|\tau_0|}.
\end{equation}
Our next theorem shows that in fact all values in a neighborhood of definite size are realized:

\begin{theorem}[existence with prescribed eccentricity]\label{all_achived_thm_intro} 
There exist a constants $\kappa>0$ and $\tau_\ast>-\infty$ with the following significance. For every $\tau_0\leq \tau_\ast$ and every $x\in \mathbb{R}$ with $|x-e_0|\leq \tfrac{\kappa}{10|\tau_0|}$ there exists a shifted HIMW translator $M\in  \mathcal{A}'$ that is $\kappa$-quadratic at time $\tau_0$ and satisfies
\begin{equation}
\mathcal{E}(M)=x.
\end{equation}
\end{theorem}
The theorem, applied in combination with the other key results from above, has the following two fundamental consequences:
\begin{enumerate}[(A)] 
\item Every noncollapsed translator in $\mathbb{R}^4$ is, up to rigid motion and scaling, a member of the HIMW class $\mathcal{A}$.\label{fund_a}
\item The space $\mathcal{A}$ is homeomorphic to an interval.\label{fund_b}
\end{enumerate}
Let us sketch how these two fundamental facts follow: Given any noncollapsed translator $M\subset \mathbb{R}^4$ that is neither a 3d bowl nor splits off a line, by Theorem \ref{thm_unique_asympt_intro} (uniform sharp asymptotics), choosing $\tau_0\ll 0$, normalizing and shifting, we can arrange that the centering condition \eqref{eq_centering_intro} holds and that $M$ is $\tfrac{\kappa}{100}$-quadratic at time $\tau_0$. Then, by Theorem \ref{all_achived_thm_intro} (existence with prescribed eccentricity) we can find a $\kappa$-quadratic  shifted HIMW translator $M'\in  \mathcal{A}'$ with $\mathcal{E}(M')=\mathcal{E}(M)$. Finally, Theorem \ref{thm:uniqueness_eccentricity_intro} (spectral uniqueness) implies that $M=M'$, which yields \ref{fund_a}. Moreover, a similar argument, now also using the fact that our sharp asymptotics are uniform, in fact shows that every point in $\mathcal{A}$ has a neighborhood that is homeomorphic to an interval, which yields \ref{fund_b}.\\

Let us now explain our strategy to prove Theorem \ref{all_achived_thm_intro}.  We fix $\tau_0\ll 0$ and denote by $\mathcal{B}_\kappa$ the set of all translators $M\in\mathcal{A}'$ that are $\kappa$-quadratic at time $\tau_0$. By Theorem \ref{thm:uniqueness_eccentricity_intro} (spectral uniqueness) the restricted eccentricity map $\mathcal{E}|_{\mathcal{B}_\kappa}:\mathcal{B}_\kappa\to \mathbb{R}$ is injective. Our goal is to show that the image of $\mathcal{E}|_{\mathcal{B}_\kappa}$ contains the interval
\begin{equation}
I:=\left[e_0-\frac{\kappa}{10|\tau_0|},e_0+\frac{\kappa}{10|\tau_0|}\right]\, .
\end{equation}
We choose a reference translator $M_0$ that is $\tfrac{\kappa}{100}$-quadratic at time $\tau_0$. 
Observe that  $\mathcal{E}(M_0)$ is contained in the interior of $I$. Also recall that we can express $M_0$ as a limit of a sequence $M_i$ of shifted HIMW translators-with-boundary with ellipsoidal parameters $c_i$ and height $h_i$.\\
We then run a continuity argument as follows:
For each $i$, we choose the maximal  interval $[a_i,b_i]$ containing $c_i$ such that for every $a\in [a_i,b_i]$ the  shifted  HIMW translators-with-boundary $M^a_i$ with ellipsoidal parameters $a$ and height $h_i$ satisfies, roughly speaking, the following two properties:
\begin{enumerate}[(i)]
\item $M^a_i$ is $\kappa$-quadratic at time $\tau_0$, and\label{kappa_quadr_boarderline}
\item $\mathcal{E}(M^a_i)\in I$.
\end{enumerate}
Since the HIMW construction at any finite height $h_i$ depends continuously on the ellipsoidal parameter, it is not hard to see that $0<a_i<c_i<b_i<\tfrac{1}{3}$. We then argue that for all large $i$ the endpoint elements are mapped to the endpoints of the interval, i.e.
\begin{equation}\label{cont_claim_eq_intro}
\mathcal{E}(M^{a_i}_i)\in  \partial I\quad\textrm{ and }\quad  \mathcal{E}(M^{b_i}_i)\in  \partial I\, .
\end{equation}
To show this, we have to exclude the possibility that \ref{kappa_quadr_boarderline} gets saturated at the endpoint elements, which we do using Theorem \ref{thm_unique_asympt_intro} (uniform sharp asymptotics) together with the fact that $\mathcal{E}(M^a_i)\in I$ . For this step, it is crucial that our notion of $\kappa$-quadraticity only depends on the behaviour of the cylindrical profile function at the single time $\tau_0$, and that our sharp asymptotics are uniform among such $\kappa$-quadratic solutions. Furthermore, invoking in addition a Rado-type argument that will be discussed further below, we show that
\begin{equation}\label{bound_diff}
\mathcal{E}(M^{a_i}_i)\neq \mathcal{E}(M^{b_i}_i).
\end{equation}
Hence, by the intermediate value theorem for each $x\in  I$  there exists some $d_i\in [a_i,b_i]$ with $\mathcal{E}(M^{d_i}_i)=x$. Finally, passing to a subsequential limit, we get the desired translator $M\in \mathcal{B}_\kappa$ satisfying $\mathcal{E}(M)=x$.\\

Having established the two fundamental facts \ref{fund_a} and \ref{fund_b}, our final key step is:

\begin{theorem}[{monotonicity and analyticity, c.f. \cite{CHH_tip}}]\label{thm_mon_analyt}
The tip curvature map $k:\mathcal{A}\to [0,1/3]$ is monotone and analytic.
\end{theorem}

Since every monotone analytic function is strictly monotone, this is indeed sufficient to conclude our main classification theorem (Theorem \ref{thm_main}) and its corollary (Corollary \ref{cor_sing}).\\

To establish monotonicity we use a Rado-type argument. This method, going back to \cite{Rado}, is traditionally used in the study of 2-dimensional surfaces see e.g. \cite{Gulliver,Cheng,MeeksYau,Ros,Brendle_sphere,HIMW}. Here, we observe that the method can be adapted to our setting of 3-dimensional hypersurfaces with circular symmetry. Finally, analyticity follows from Lyapunov-Schmidt reduction and a linearized version of the estimates from Section \ref{sec_profile}. This proof of analyticity is rather standard but also rather lengthy, and will thus be given in a separate technical paper \cite{CHH_tip}.\footnote{Analyticity is only needed to relate the spectral eccentricity and the tip curvature. Readers who are happy with a classification of noncollapsed translators in terms of their spectral eccentricity can of course simply skip the paper about analyticity.}
\\

Finally, in Section \ref{sec_profile}, we prove Theorem \ref{thm:uniqueness_eccentricity_intro} (spectral uniqueness), by adapting the argument from \cite{ADS2}  -- with some important differences and additional steps -- to our setting. To explain the underlying mechanism, recall that by equation \eqref{mean_curv_exp} the level sets almost evolve by mean curvature flow. More precisely, the profile function $V$ of the level sets of our translator satisfies the equation
\begin{equation}\label{evVintro}
V_t=\frac{(1+V_t^2)V_{xx}+(1+V_x^2)V_{tt}-2V_xV_tV_{xt}}{1+V_x^2+V_t^2}-\frac{1}{V}.
\end{equation}
For comparison, the profile function $U$ of the ancient ovals in $\mathbb{R}^3$ would satisfy
\begin{equation}\label{evUintro}
U_t = \frac{U_{xx}}{1+U_x^2}-\frac{1}{U}.
\end{equation}
Heuristically, thanks to the  vanishing asymptotic slope one hopes that the functions $U$ and $V$ behave quite similarly. However, while \eqref{evUintro} is an uniformly parabolic PDE, equation \eqref{evVintro} is an elliptic PDE with degenerating coefficients, so some careful arguments are needed to make these heuristics precise.\\

In terms of the renormalized profile function our evolution equation takes the form
\begin{align}
v_\tau=&\frac{v_{yy}}{1+v_y^2}-\frac{y}{2}v_y+\frac{v}{2}-\frac{1}{v}+ e^{\tau}\cN[v],
\end{align}
where $\cN$ is a certain nonlinear error term, involving second derivatives with respect to both $y$ and $\tau$. Our inverse profile function satisfies
\begin{equation}
Y_\tau=\frac{Y_{vv}}{1+Y_v^2}+\frac{1}{v}Y_v +\frac{1}{2}(Y-vY_v)+e^{\tau}\mathcal{M}[Y],
\end{equation}
for another nonlinear term $\mathcal{M}$, which we also view as error term.\\

We first prove that our profile function is almost quadratically concave, namely
\begin{equation}
(v^2)_{yy}\leq \frac{e^\tau}{v^2}.
\end{equation}
This is based on the maximum principle, and thus some care is needed to handle the error term as opposed to the analysis of the ovals in \cite{ADS2}, where the profile function was exactly quadratically concave. The almost quadratic concavity estimate has the important corollary that $Y\sim Ce^{-v^2/4}$ near the tips.\\

We then consider the difference of the profile functions $w:=v_1-v_2$, as well as its truncated version
\begin{equation}
w_{\cC}:=v_1\varphi_{\cC}(v_1)-v_2\varphi_{\cC}(v_2),
\end{equation}
where, as before, $\varphi_{\cC}$ localizes in the cylindrical regions $\mathcal{C}_i=\{ v_i\geq \tfrac{5}{8}\theta\}$.
The difference function $w$ satisfies an evolution equation of the schematic form
\begin{equation}
w_\tau=\mathfrak{L}w+\mathcal{E}[w]+e^{\tau}\mathcal{F}[w],
\end{equation}
where $\mathfrak{L}$ is the one-dimensional Ornstein-Uhlenbeck operator. The function $w_{\cC}$ satisfies a related equation with additional terms coming from the cutoff function. We also work with the difference of inverse profile functions $W:=Y_1-Y_2$, as well as its truncated version
\begin{equation}
W_{\cT}:=\varphi_{\cT}\cdot(Y_1-Y_2),
\end{equation}
where $\varphi_{\cT}$ is a suitable cutoff function that localizes in the tip region $\mathcal{T} = \{ v\leq 2 \theta\}$. The function $W_{\cT}$ also satisfies a related degenerate elliptic PDE, which we again view as parabolic PDE with error terms.\\

Our energy estimates require certain weighted integral norms, similarly as in \cite{ADS2}. In addition to the Gaussian $L^2$-norm $\| \,\, \|_{\fH}$, one also needs the Gaussian $H^1$-norm
\begin{equation}
\|f\|_{\mathcal{D}} := \left( \int (f^2 +f_y^2) e^{-y^2/4} dy \right)^{1/2},
\end{equation}
and its dual norm $\| \,\, \|_{\mathcal{D}^\ast}$. Moreover, for time-dependent functions this induces the parabolic norms 
\begin{equation}
\|f \|_{\mathcal{X},\infty}:=\sup_{\tau\leq \tau_0 }\left( \int_{\tau-1}^\tau \| f(\cdot,\sigma)\|^2_{\mathcal{X}} \, d\sigma \right)^{1/2},
\end{equation}
where $\mathcal{X}=\fH,\mathcal{D}$ or $\mathcal{D}^\ast$. Furthermore, in the tip region one works with the norm
\begin{equation}
\| F\|_{2,\infty}:= \sup_{\tau\leq \tau_0} \frac{1}{|\tau|^{1/4}} \left( \int_{\tau-1}^\tau \int_0^{2\theta} F(v,\sigma)^2 e^{\mu(v,\sigma)}\, dv\, d\sigma \right)^{1/2},
\end{equation}
where $\mu$ is a carefully chosen weight satisfying $\mu(v,\tau)=-\tfrac{1}{4}Y_1(v,\tau)^2$ for $v\geq \theta/2$.\\

In contrast to \cite{ADS2}, we also need exponentially weighted $C^2$-norms to control the higher derivative terms coming from the nonlinearities $e^\tau\cN$ and $e^\tau\mathcal{M}$. Specifically, in the cylindrical region $\mathcal{C}=\mathcal{C}_1\cup \mathcal{C}_2$ we work with\footnote{For technical reasons, in the cylindrical region we use the weight $|\tau|e^{\tau}$. Alternatively, one could use $e^{\frac{99}{100}\tau}$.} 
\begin{align}
\| f\|_{C^2_{\exp}(\mathcal{C})}:=\sup_{\tau\leq \tau_0} \left( |\tau|e^{\tau}\sup_{y:(y,\tau)\in \mathcal{C} } \big(|f|+|f_y|+|f_\tau|+ |f_{yy}|+|f_{y\tau}|+|f_{\tau\tau}|\big)\right),
\end{align}
and in the tip region we work with
\begin{align}
\| F\|_{C^2_{\exp}(\mathcal{T})}:=\sup_{\tau\leq \tau_0} \left( e^{\tau}\sup_{v\leq 2\theta} \big(|F|+|F_v|+|F_\tau|+ |F_{vv}|+|F_{v\tau}|+|F_{\tau\tau}|\big)\right).
\end{align}
\bigskip

In the cylindrical region we prove the energy estimate
\begin{equation}
\| w_\cC-\fp_0w_\cC \|_{\mathcal{D},\infty}\leq  \varepsilon\left( \|w_{\cC}\|_{\mathcal{D},\infty}+\|w\, 1_{\{\theta/2\leq v_1\leq\theta\}}\|_{\mathfrak{H},\infty}\right)+C\| w\|_{C^2_{\exp}(\mathcal{C})}.
\end{equation}
In the tip region we prove the energy estimate
\begin{equation}
\|W_{\mathcal{T}}\|_{2,\infty}\leq \frac{C}{|\tau_0|}\left(\|W 1_{[\theta,2\theta]}\|_{2,\infty}+\| W\|_{C^2_{\exp}(\mathcal{T})}\right).
\end{equation}
The proofs of these energy estimates are along the lines of \cite{ADS2}, but with various additional steps and technical tweaks necessitated by our intrinsic localization and the nonlinear terms.

We then combine our two energy estimates, taking also into account the equivalence of norms in the transition region similarly as in \cite{ADS2}, to derive the decay estimate
\begin{equation}\label{est_decay_intro}
\| w_{\cC}\|_{\mathcal{D},\infty} + \|W_{\mathcal{T}}\|_{2,\infty} \leq C\left(\| w\|_{C^2_{\exp}(\mathcal{C})} + \| W\|_{C^2_{\exp}(\mathcal{T})}\right)\, .
\end{equation}
For comparison, in the corresponding estimate in \cite[Section 8]{ADS2} the right hand side would simply vanish and one could conclude directly that $w$ and $W$ vanish identically. In our case, however, the estimate \eqref{est_decay_intro} is only half of the story, since the right hand side contains the exponentially weighted error terms coming from our nonlinearities. While \eqref{est_decay_intro} gives control backwards in $\tau$, we also need an estimate that gives control forwards in $\tau$. To this end, we consider the Hausdorff distance of the level sets, namely
\begin{equation}
D(h):=d_{\textrm{Hausdorff}}\left(M^1 \cap \{x_1 = h\},M^2 \cap \{x_1 = h\}\right).
\end{equation}
Note that $D(h)$ is essentially equivalent to the sum of the $L^\infty$ norms of $w$ and $W$ at time $\tau=-\log h$. We then consider the level $h'=e^{-\tau'+1}$, where 
$\tau'\in (-\infty,\tau_0]$ is such that
\begin{equation}\label{eq_C2_maximum_time_intro}
\|w\|_{C^2_{\exp}(\cC)}+\|W\|_{C^2_{\exp}(\cT)}\leq 2e^{\tau'}\left(|\tau'|\|w\|_{C^2|C_{\tau'}} +\|W\|_{C^2|T_{\tau'}} \right)\, .
\end{equation}
Using the comparison principle for translators we show that we have the weighted $L^\infty$-estimate
\begin{equation}
\sup_{h \in [h'/e^2,h']}D(h)\leq 10(\log h')^{1/2} D(h')\, .
\end{equation}
Using this weighted $L^\infty$-estimate control, we can then estimate the weighted $C^2$-norms in terms of the weighted $L^2$-norms via interior estimates. Specifically,  taking also into account that thanks to our sharp asymptotics the ellipticity of \eqref{evVintro} only degenerates polynomially in $\log h$, we derive the estimate
\begin{equation}\label{eq_forward_intro}
\| w\|_{C^2_{\exp}(\mathcal{C})} + \| W\|_{C^2_{\exp}(\mathcal{T})}\leq \eps \big( \| w_{\cC}\|_{\mathcal{D},\infty} + \|W_{\mathcal{T}}\|_{2,\infty}\big).
\end{equation}
Finally, combining \eqref{est_decay_intro} and \eqref{eq_forward_intro}  we infer that $w$ and $W$ vanish identically, i.e. that $M^1=M^2$. This concludes the outline of the proof.\\

\bigskip

\noindent\textbf{Acknowledgments.}
KC has been supported by KIAS Individual Grant MG078901 and a POSCO Science Fellowship. RH has been supported by an NSERC Discovery Grant and a Sloan Research Fellowship. OH has been supported by the Koret Foundation early career award and ISF grant 437/20. We thank Beomjun Choi, Wenkui Du and the anonymous referee for very helpful comments.\\

\bigskip

\section{Coarse asymptotics and circular symmetry}\label{sec_coarse_asymptotics}

Let $M\subset\mathbb{R}^4$ be a noncollapsed translator. Without loss of generality we can assume that it translates with unit speed in positive $x_1$-direction, namely
\begin{equation}
{\bf{H}}=e_1^\perp\, .
\end{equation}
By the convexity estimate \cite[Theorem 1.10]{HaslhoferKleiner_meanconvex}, our translator is convex. If $M$ splits off a line, then it must be $\mathbb{R}\times$2d-bowl by \cite{Haslhofer_bowl}. We can thus assume from now on that $M$ is strictly convex.\\

\subsection{Coarse asymptotics}

Let $K$ be the closed domain bounded by $M$. Consider the blowdown
\begin{equation}
\check{K}:=\lim_{\lambda\rightarrow 0}\lambda K.
\end{equation}
By the main theorem of our prior paper \cite{CHH_blowdown} the blowdown is a halfline, namely
\begin{equation}\label{eq_blowdown}
\check{K} =\{ \lambda e_1 \, | \, \lambda\geq 0\}.
\end{equation}
In the following, we write $\nu$ for the outwards unit normal.

\begin{proposition}[asymptotic slope and tip point]\label{lemma_asympt_slope}
We have $\langle e_1, \nu \rangle\rightarrow 0$ as $x_1\rightarrow \infty$. Moreover, there exists a unique tip point  $p_0\in M$  such that $x_1(p_0)=\inf_{p\in M}x_1(p)$.   
\end{proposition}

\begin{proof}
 By \eqref{eq_blowdown} and convexity, the fact that $\langle e_1, \nu \rangle\rightarrow 0$ as $x_1\to\infty$ is clear.

To find a tip point, assume without loss of generality that $0\in M$, and consider the infimum
\begin{equation}
m:=\inf_{p\in M}x_1(p).
\end{equation}
Let $p_i\in M$ be a minimizing sequence. Suppose towards a contradiction that $|p_i|\rightarrow \infty$.  Then, up to a  subsequence $p_i/||p_i||\rightarrow w\in S^3$, and the ray $\ell_w:=\{\lambda w\, | \, \lambda\geq 0\}$ is contained in $K$, and thus in $\check{K}$. By \eqref{eq_blowdown} this implies $w=e_1$, which contradicts the assumption that $p_i$ is a minimizing sequence. 
Thus, there exists a point $p_0$ with 
\begin{equation}
x_1(p_0)=\inf_{p\in M}x_1(p).
\end{equation}
By strict convexity this tip point is unique.
This completes the proof of the proposition.
 \end{proof}
 \bigskip

Next, by \cite[Theorem 1.14]{HaslhoferKleiner_meanconvex} and \cite{CM_uniqueness} the tangent flow to $M_t=M+te_1$ at time $-\infty$ is either a neck or a bubble-sheet, namely either
\begin{equation}\label{neck_tangent}
\lim_{\lambda\to 0} \lambda M_{\lambda^{-2}t} = \mathbb{R}\times S^2(\sqrt{-4t}),
\end{equation}
or
\begin{equation}\label{bubble_sheet_tangent}
\lim_{\lambda\to 0} \lambda M_{\lambda^{-2}t} = \mathbb{R}^2\times S^1(\sqrt{-2t}).
\end{equation}
If \eqref{neck_tangent} holds, then $M$ is the round bowl by \cite{Haslhofer_bowl}. We can thus assume from now on that \eqref{bubble_sheet_tangent} holds.\\

Let us now consider the level sets
\begin{equation}
\Sigma^h:= M\cap \{x_1= h\}.
\end{equation}
By strict convexity,  the level sets $\Sigma^h$ are compact and diffeomorphic to the two-sphere.

\begin{proposition}[diameter growth]\label{prop_diameter_lev}
The level sets satisfy
\begin{equation}
\lim_{h\to \infty} \frac{\textrm{diam}(\Sigma^h)}{h^{1/2}}=\infty\qquad\textrm{and} \qquad \lim_{h\to \infty} \frac{\textrm{diam}(\Sigma^h)}{h^{1/2+\delta}}=0
\end{equation}
for every $\delta>0$.
\end{proposition}

\begin{proof} 
The first estimate follows from the assumption that we are in case  \eqref{bubble_sheet_tangent}. To prove the second estimate, note that since $M$ is strictly convex, \cite[Theorem 1.10]{CHH_blowdown} implies that in the fine bubble-sheet expansion of the renormalized flow $\bar{M}_\tau=e^{\tau/2}M_{-e^{-\tau}}$ the neutral mode is dominant. Hence, we can apply \cite[Corollary 1.8]{CHH_blowdown}, which says that given any $\delta>0$ for $\tau\ll 0$ we have the estimate
\begin{equation}
\bar{M}_\tau \cap \{x_1=0\}\subseteq B(0,{e^{\delta|\tau|}}).
\end{equation}
On the other hand, using the translator equation and remembering the renormalization we see that
\begin{equation}\label{ren_and_unren}
\Sigma^h = e^{-\tau/2}\left(\bar{M}_{\tau} \cap \{x_1=0 \}\right),
\end{equation}
where $\tau = - \log h$. Combining these facts yields the assertion.
\end{proof}

As a corollary of the proof we also obtain:

\begin{corollary}[inscribed radius]\label{prop_inscribed_lev}
The maximal inscribed radius of the level sets satisfies
\begin{equation}
\lim_{h\to \infty} \frac{r_{\textrm{in}}(\Sigma^h)}{(2h)^{1/2}}=1.
\end{equation}
\end{corollary}

\begin{proof} 
By \eqref{bubble_sheet_tangent} the renormalized flow $\bar{M}_\tau$ for $\tau\to -\infty$ converges to $\Gamma=\mathbb{R}^2\times S^{1}(\sqrt{2})$. Hence,
using the inwards quadratic bending from \cite[Theorem 1.7]{CHH_blowdown} we see that the maximal inscribed radius of $\bar{M}_\tau$ for $\tau\to -\infty$ converges to $\sqrt{2}$. Together with \eqref{ren_and_unren}, where $\tau = - \log h$, this implies the assertion.
\end{proof}

The following estimate shows that the mean curvature of the level set is, up to a cubic error term, the same as the mean curvature $H$ of the translator, when $x_1$ is high:

\begin{proposition}[mean curvature of level sets]\label{prop_mean_curv_levelsets}
There exists a uniform constant $C<\infty$ such that
\begin{equation}
|H-H^h| \leq CH^3,
\end{equation}
where $H^h$ is the mean curvature of $\Sigma^h=M\cap \{ x_1=h \}$ in $P^h=\{x_1=h\}$. 
\end{proposition}

\begin{proof}
On a translator we have $-\nabla H=A(e_1^{\top},-)$ and $-\langle \nabla H,e_1 \rangle=\Delta H+|A|^2H$.
Thus, 
\begin{equation}
|A(e_1^{\top},e_1^{\top})|\leq |\Delta H|+|A|^2H.
\end{equation}
From the local curvature estimate \cite[Theorem 1.8]{HaslhoferKleiner_meanconvex}, we know that $|\Delta H|\leq CH^3$,  and so
\begin{equation}
|A(e_1^{\top},e_1^{\top})|\leq CH^3.
\end{equation}
Now, given $p\in \Sigma ^h$, let $\{U,V\}$ be and orthonormal basis to $T_p\Sigma^h$ and let $W:=e_1^{\top}/||e_1^{\top}||$. Then  $\{U,V,W\}$ is an orthonormal basis to $T_pM$ and 
\begin{equation}
H=A(U,U)+A(V,V)+A(W,W) = A(U,U)+A(V,V)+O(H^3).
\end{equation}
Now, let $\gamma_U$ and $\gamma_V$ be unit speed curves in $\Sigma^h$ such that $\gamma_{U}(0)=p$ and $\gamma'_U(0)=U$ respectively $\gamma_{V}(0)=p$ and $\gamma'_V(0)=V$. Then
\begin{equation}
H=\langle \gamma_U''+\gamma_V'',\nu \rangle+O(H^3). 
\end{equation}
On the other hand, the normal to $\Sigma^{h}$ in $P^h$ is 
\begin{equation}
\nu^h=\frac{\nu+He_1}{\sqrt{1-H^2}}.
\end{equation}
As $\langle U,e_1 \rangle=0$ and $\langle V,e_1 \rangle=0$, we conclude that
\begin{equation}
H^h=\langle \gamma_U''+\gamma_V'',\nu^h \rangle= \langle \gamma_U''+\gamma_V'',\nu \rangle(1+O(H^2))=(H+O(H^3))(1+O(H^2)).  
\end{equation}
This proves the proposition.
\end{proof}

\bigskip

\subsection{Circular symmetry}\label{sec_symmetry}
Let $M\subset \mathbb{R}^4$ be a strictly convex noncollapsed translator, normalized such that it translates with unit speed in positive $x_1$-direction, and that the tip is at the origin. Further, suppose that $M$ is not the round bowl. By Colding-Minicozzi \cite{CM_uniqueness} the asymptotic cylinder $\mathbb{R}^2\times S^1$ is unique. We can assume without loss of generality that the $\mathbb{R}^2$-factor is in the $x_1x_2$-plane.  Let $R$ be the rotation vector field corresponding to the circular symmetry of the asymptotic cylinder, namely
\begin{equation}
R:=x_3\partial_{x_4}-x_4\partial_{x_3}\, .
\end{equation}
The goal of this subsection is to give a short proof of Zhu's theorem:

\begin{theorem}[{circular symmetry}]\label{thm_symmetry}
$M$ is $\mathrm{SO}(2)$-symmetric. More precisely, there exists some $a\in \{0\}\times\mathbb{R}^2$, such that the recentered translator $M-a$ is invariant under rotations generated by the vector field $R$.
\end{theorem}

Note that rotations with center $a\in\{0\}\times\mathbb{R}^2$ are generated by the vector field
\begin{equation}
R_a:=(x_3-a_3)\partial_{x_4}-(x_4-a_4)\partial_{x_3}\, .
\end{equation}

Consider the rotation function $f_{a}:=\langle R_a,\nu\rangle$, where $\nu$ is the outwards unit normal of $M$. Our goal is to find some $a\in \{0\}\times\mathbb{R}^2$ such that $f_{a}$ vanishes identically on $M$.

\begin{proposition}[{weighted estimate, c.f. \cite[Proposition 3.1]{Haslhofer_bowl}}]\label{prop_ell_est}
For all $h>0$ we have
\begin{equation}
\sup_{\{x_1\leq h\}}\left|{\frac{f_a}{H}}\right|\leq \sup_{\{x_1= h\}}\left|{\frac{f_a}{H}}\right|\, .
\end{equation}
\end{proposition}

\begin{proof} On our translator, the rotation functions and the mean curvature satisfy
\begin{align}
\left(\Delta + e_1^\top\cdot\nabla +|A|^2\right) f_a &=0,\label{ell_eq_frot}\\
\left(\Delta + e_1^\top\cdot\nabla +|A|^2\right) H &=0.
\end{align}
Hence, the assertion follows from the maximum principle.
\end{proof}

\begin{proof}[{Proof of Theorem \ref{thm_symmetry}}]
Consider the function
\begin{equation}
B(h) : = \min_{a\in\{0\}\times\mathbb{R}^2} \max_{\{x_1=h\}}|f_{a}| 
\end{equation}

\noindent\underline{Case 1:} Suppose there is a sequence $h_{i}\to \infty$ with $B(h_{i}) = 0$. For each $i$, choose $a_i$ such that
\begin{equation}
 \max_{\{x_1=h_i\}}|f_{a_i}|=0.
\end{equation}
Proposition \ref{prop_ell_est} (weighted estimate) implies that $f_{a_i}=0$ in the region $\{x_1\leq h_i\}$. Observe that
\begin{equation}
f_{a_i} = f_0 - \langle T_i ,\nu\rangle,
\end{equation}
where $T_i = (0,0,-a^i_4,a^i_3)$.
Thus, $ \langle T_i ,\nu\rangle=f_0$ in the region $\{x_1\leq h_i\}$. Hence, $a_i$ is constant and $f_{{a_i}}=0$ everywhere, and we have proven rotational symmetry.\\

\noindent\underline{Case 2:} Suppose now that $B(h) > 0$ for $h$ large. Fix $\tau \in (0,1/4)$ such that
$\tau^{-\tfrac{1}{2}+\delta}>2D$, where  $D<\infty$ is the constant from  \cite[Proposition 4.1]{Haslhofer_bowl}.
By Proposition \ref{prop_diameter_lev} (diameter asymptotics) we have
\begin{equation}
B(h) \leq O(h^{1/2+\delta}).
\end{equation}
Hence, we can then find $h_{i}\to\infty$ such that
\begin{equation}\label{seq_decay}
\inf_{h\in [\tau h_i, h_i]}B(h) \geq \frac{1}{2}\tau^{1/2+\delta}  B( h_{i}).
\end{equation}
Choose $a_{i}$ such that 
\begin{equation}\label{sup_bm}
\max_{\{x_1=h_i\}}|f_{{a_i}}| = B(h_{i}).
\end{equation}
Let $p_i\in M\cap\{x_1=h_i\}$ be a point where the maximum in \eqref{sup_bm} is attained, and consider the renormalized function
\begin{equation}
\widetilde f_{i} : ={B(h_{i})}^{-1}f_{{a_i}}.
\end{equation}
Recall that the family
$\{M_t=M+te_1\}_{t\in\mathbb{R}}$ moves by mean curvature flow.
If we view $\tilde{f}_i$ as a one parameter family of functions on $M_t$, then equation \eqref{ell_eq_frot} takes the form
\begin{equation}
 \partial_t \widetilde{f}_i=(\Delta+\abs{A}^2)\widetilde{f}_i.
\end{equation}
Set $\lambda_i:=H(p_i)$, and consider the parabolic rescalings
\begin{equation}
\widehat{M}_t^i:=\lambda_i (M_{\lambda_i^{-2}t}-p_i),
\end{equation}
and
\begin{equation}
\widehat{f}_i(x,t):=\widetilde{f}_i(\lambda_i^{-1}x+p_i,\lambda_i^{-2}t),
\end{equation}
where $x\in \widehat{M}_t^i$.
Note that $\widehat{M}_t^i$ moves by mean curvature flow and that $\widehat{f}_i$ satisfies the parabolic equation
\begin{equation}
\partial_t \widehat{f}_i= (\Delta+\abs{A}^2)\widehat{f}_i.
\end{equation}
Observe that $\lambda_i\to 0$ by Proposition \ref{lemma_asympt_slope} (asymptotic slope) and the translator equation. On the other hand, using Corollary \ref{prop_inscribed_lev} (inscribed radius) and the sharp noncollapsing estimate from \cite{HK_inscribed} we get
\begin{equation}\label{eq_from_inscr}
\liminf_{i\to \infty} (2h_i)^{1/2}\lambda_i \geq 1.
\end{equation}
Thus, by the global convergence theorem \cite[Thm. 1.12]{HaslhoferKleiner_meanconvex}, for $i\to \infty$ the mean curvature flows $\widehat{M}_t^i$ converge (subsequentially) to an ancient noncollapsed mean curvature flow $\widehat{M}^\infty_t$ that splits off a line in $x_1$-direction. Write $\widehat{M}^\infty_t=N_t\times \mathbb{R}$. Observe that $N_t$ is noncompact by \eqref{bubble_sheet_tangent}. Hence, by the classification from Brendle-Choi \cite{BC} the 2d-flow $N_t$ must be either (a) a round shrinking cylinder $\{C_t\}_{t< 1/2}$, or (b) a translating bowl soliton $B$.\\

Using equation \eqref{seq_decay}, Proposition \ref{prop_ell_est} (weighted estimate), and the knowledge of the mean curvature of the limiting flow, we see that $\hat{f}_m$ converges (subsequentially) to a limit $f=\{f(t)\}$, which after splitting of the $\mathbb{R}$-factor in $x_1$-direction can be viewed as a function on $N_t$, solving
\begin{equation}
 \partial_t f=(\Delta_{N_t}+\abs{A_{N_t}}^2) f,
\end{equation}
that in case (a) satisfies $\abs{f(z,\theta,t)}\leq 4$ for $t\in(0,\tfrac14)$, and in case (b) satisfies $|f(z,\theta,t)|\leq C(1+z)^{-1/2}$, where $z$ and $\theta$ denote the height and angle on $N_t$. Moreover, since $\textrm{div}_{\mathbb{R}^{4}} R=0$ and $\langle R,\partial_{x_1}\rangle=\langle R,\partial_{x_2}\rangle=0$, the divergence theorem yields, after splitting off an $\mathbb{R}$-factor in $x_1$-direction, that for every $z$ we have
\begin{equation}
\int f(z,\theta,t)\, d\theta = 0.
\end{equation}

Let us first consider case (a). Note that $f$ is independent of $z$. Hence,  \cite[Proposition 4.1]{Haslhofer_bowl} gives
\begin{equation}\label{eq_contr1}
 \inf_{T\in \mathbb{R}^2} \sup_{C_t}|f(t)-f_{T}|\leq D\left(\tfrac{1}{2}-t\right)
\end{equation}
for all $t\in [1/4,1/2)$, where $f_T=\langle T,\nu \rangle$ for translations $T\in \mathbb{R}^2$.
On the other hand, we have
\begin{align}
 \inf_{T} \sup_{\{x_1=0\}} | \hat{f}_i(\tfrac12-\tau)-f_T|
& =\inf_{T} \sup_{\{x_1 = h_{i}\}} | \tilde{f}_i(\lambda_i^{-2}(\tfrac12-\tau))-f_T| 
 =\inf_{T} \sup_{\{x_1 = h_{i}-\lambda_i^{-2}(\frac12-\tau)\}} | \tilde{f}_i(0)-f_T| \nonumber\\
& = \frac{1}{B(h_{i})} \inf_{T} \sup_{\{x_1 =   h_{i}- \lambda_i^{-2}(\frac12-\tau)\}} |f_{R_{a_i}}-f_T| 
 = \frac{B(h_{i}-\lambda_i^{-2}(\tfrac12-\tau))}{B(h_{i})}  \geq \frac 12 \tau^{1/2+\delta}
\end{align}
for $i$ large enough, where we used \eqref{seq_decay} and \eqref{eq_from_inscr} in the last step.
Taking the limit as $i\to\infty$ gives
\begin{equation}\label{eq_contr2}
 \inf_{T} \sup_{C_{\frac12-\tau}}|f(\tfrac12-\tau)-f_{T}|\geq \tfrac12 \tau^{1/2+\delta}.
\end{equation}
Since $\tau^{-\frac{1}{2}+\delta}>2D$, this contradicts \eqref{eq_contr1}. This completes the analysis in case (a).\\

Finally, in case (b) Proposition \ref{prop_liouville_bowl} (Liouville property) from below gives a contradiction. This finishes the proof of the theorem.
\end{proof}

In the above proof we have used the following proposition:

\begin{proposition}[Liouville property]\label{prop_liouville_bowl}
Suppose $f$ is a solution on the 2d-bowl $B$ of
\begin{equation}\label{eq_ell_trans}
(\Delta+ e_z^\top\cdot\nabla + |A|^2)f=0,
\end{equation}
such that for every $z$ we have
\begin{equation}\label{eq_trans_center}
\int f(z,\theta)\, d\theta = 0.
\end{equation}
If
\begin{equation}\label{trans_growth}
|f|\leq C(1+z)^{-1/2}
\end{equation}
for some $C<\infty$, then $f=0$.
\end{proposition}
\begin{proof}
The argument from \cite{Haslhofer_bowl}, which has been written for $f=f_R$ but also applies for other solutions $f$ of \eqref{eq_ell_trans} satisfying \eqref{trans_growth}, shows that $f=\langle T,\nu\rangle$ for some $T\in\mathbb{R}^2$. By \eqref{eq_trans_center} we must have $T=0$.
\end{proof}

\bigskip

\section{Uniform sharp asymptotics}\label{sec_sharp}

Throughout this section, $M\subset \mathbb{R}^4$ denotes any noncollapsed translator that is neither the 3d round bowl nor $\mathbb{R}\times$2d-bowl. As before, we normalize such that the translation is in $x_1$-direction with unit speed.\\

To establish the sharp asymptotics we need suitable inner barriers for the renormalized mean curvature flow near the cylinder $\Gamma=\mathbb{R}^2\times S^1(\sqrt{2})$. To begin with, recall from Angenent-Daskalopoulos-Sesum \cite[Figure 1 and Section 8]{ADS1} that there is some $L_0>1$ such that  for every $a\geq L_0$ there are shrinkers 
\begin{align}\label{fol_ads}
{\Sigma}_a &= \{ \textrm{surface of revolution with profile } r=u_a(y_1), 0\leq y_1 \leq a\}\subset\mathbb{R}^3.
\end{align}
The parameter $a$ captures where the  concave functions $u_a$ meet the $y_1$-axis.  In our previous paper \cite[Section 3]{CHH_blowdown} we constructed a bubble-sheet foliation $\Gamma_a\subset \mathbb{R}^4$  by shifting and rotating the ADS-shrinker foliation $\Sigma_a\subset\mathbb{R}^3$. For the present paper, we need the somewhat more general inner barriers
\begin{equation}\label{fol_bubble}
\Gamma_a^\eta:=\big\{(r\cos\theta ,r\sin\theta,y_3,y_4) \,: \, \theta\in [0,2\pi), (r-\eta,y_3,y_4) \in {\Sigma}_a \big\}\subset\mathbb{R}^4,
\end{equation}
where we now shift by $\eta>0$ instead of by $1$.

\begin{proposition}[barriers]\label{prop_inner.barrier} The hypersurfaces $\Gamma_a^\eta$ act as an inner barriers for the renormalized mean curvature flow in the region $|(y_1,y_2)|\geq 3\eta^{-1}$.
\end{proposition}

\begin{proof}
Being an inner barrier for the renormalized mean curvature flow is equivalent to the condition
\begin{equation}\label{to_show_negative_div}
H_{{\Gamma}_a^{\eta}} \leq \tfrac{1}{2}\langle \vec{y},\nu \rangle.
\end{equation}
To show this, note that by symmetry of the hypersurfaces $\Gamma_a^\eta$, it suffices to compute $H_{{\Gamma}_{a}^{\eta}}$ in the region $\{y_2=0,\;y_1>0\}$, where we can identify points and unit normals in ${\Gamma}_{a}^{\eta}$ with the corresponding ones in ${\Sigma}_a$, by disregarding the $y_2$-component. 
The relation between the mean curvature  of a surface ${\Sigma}\subset\mathbb{R}^3$ and its (unshifted) rotation ${\Gamma} \subset\mathbb{R}^4$ on points with $y_2=0$ and $y_1>0$ is given by 
\begin{equation}\label{rotate_H}
H_{{\Gamma}}=H_{{\Sigma}}+\frac{1}{y_1}\langle e_1,\nu \rangle.
\end{equation}  
In our case, the convexity of $\Sigma_a$ gives $\langle e_1,\nu  \rangle\geq 0$, so using \eqref{rotate_H} and the shrinker equation we infer that
\begin{equation}\label{rot_cal}
H_{{{\Gamma}}_a^{\eta}}
= \frac{1}{2}\langle \vec{y}-\eta e_{1},\nu\rangle+\frac{1}{y_1} \langle e_1,\nu  \rangle \leq \frac{1}{2}\langle \vec{y}, \nu \rangle,  
\end{equation}
where in the last inequality, we have used that $y_1 \geq 2\eta^{-1}$. This proves the proposition.
\end{proof}

\bigskip

\subsection{Sharp asymptotics in bubble-sheet region} 

We consider the renormalized mean curvature flow
\begin{equation}
\bar{M}_\tau=e^{\frac{\tau}{2}}M_{-e^{-\tau}},
\end{equation}
where $\tau=-\log(-t)$. Then,  $\bar{M}_\tau$ converges to
\begin{equation}
\Gamma=\mathbb{R}^2\times S^1(\sqrt{2})
\end{equation}
as $\tau \to -\infty$. Recall that we have circular symmetry (see Theorem \ref{thm_symmetry}). In particular, this symmetry must preserve $\Gamma$. This symmetry must also preserve the positive $e_1$-axis. Hence, after shifting $M$ in the $x_3x_4$-plane, the hypersurfaces $\bar{M}_\tau$ are left invariant by the rotation vector field
\begin{equation}
V= x_3 e_4 - x_4 e_3.
\end{equation}
Denote by  $\Omega_\tau$ the set of points ${\bf{y}}=(y_1,y_2)\in\mathbb{R}^2$ such that $({\bf{y}},r \cos\theta,r\sin\theta )\in \bar{M}_\tau$ for some $r> 0$. There exists a unique function $u:\Omega_\tau \times \mathbb{R}\to (-\sqrt{2},\infty)$ such that  
\begin{equation}
\big({\bf{y}},(\sqrt{2}+u({\bf{y}},\tau))\cos\theta,(\sqrt{2}+u({\bf{y}},\tau))\sin\theta \big) \in \bar{M}_\tau,
\end{equation}
and
\begin{equation}
\lim_{{\bf{y}}\to \partial \Omega_\tau}u({\bf{y}},\tau)=-\sqrt{2}.
\end{equation}
Moreover, there exists an \emph{admissible graphical radius function $\rho_0(\tau)$ for $\tau\leq \tau_\ast$}, namely a positive smooth function $\rho_0:(-\infty,\tau_\ast]\to\mathbb{R}_+$ with  $\lim_{\tau\to -\infty}\rho_0(\tau)=\infty$ such that 
\begin{align}
-\rho_0(\tau)\leq \rho_0'(\tau)\leq 0
\end{align}
and
\begin{equation}\label{eq_graphC4}
\|u\|_{C^4(B(0,2\rho_0(\tau)))}\leq \rho_0(\tau)^{-2}
\end{equation}
hold for $\tau\leq \tau_\ast$.

Since $\bar{M}_\tau$ moves by renormalized mean curvature flow, the graph function $u$ satisfies the equation
\begin{align}
u_\tau= \sum_{i,j=1}^2 \left(\delta_{ij}-\frac{u_{y_i}u_{y_j}}{1+|\nabla u|^2}\right) u_{y_iy_j}-\frac{1}{\sqrt{2}+u}+\frac{1}{2}\Big(\sqrt{2}+u-{\bf{y}}\cdot \nabla u\Big).
\end{align}
This yields
\begin{align}
u_\tau=\mathcal{L}u+E,\label{u.equation}
\end{align}
where
\begin{equation}
\mathcal L=\frac{\partial^2}{\partial y_1^2}+\frac{\partial^2}{\partial y_2^2}  - \frac{y_1}{2} \frac{\partial}{\partial y_1} - \frac{y_2}{2} \frac{\partial}{\partial y_2} + 1,
\end{equation}
and where the error term thanks to \eqref{eq_graphC4} satisfies the pointwise estimate
\begin{equation}\label{pointwiseE}
|E|\leq C\rho_0^{-2}(|u|+|\nabla u|).
\end{equation}
The two-dimensional Ornstein-Uhlenbeck operator $\mathcal{L}$ has $3$ unstable eigenfunctions, namely
\begin{equation}\label{eigenfunctions_unstable}
1,y_1,y_2,
\end{equation}
and $3$ neutral eigenfunctions, namely
\begin{equation}\label{expansion_neutral}
y_1^2-2,y_2^2-2,y_1y_2.
\end{equation}

Next, we fix a smooth cut-off function $\chi:\mathbb{R}^+\to [0,1]$ such that $\chi(s)=1$ for $s\leq 1$ and $\chi(s)=0$ for $s\geq 2$. Then, we define
\begin{equation}\label{def_alpha}
\alpha(\tau)=\left(\int_{|{\bf{y}}|\leq 2\rho_0(\tau)}u^2({\bf{y}},\tau)\chi\left(\tfrac{|{\bf{y}}|}{\rho_0(\tau)}\right)\tfrac{1}{2\sqrt{2\pi}} e^{-\frac{|{\bf{y}}|^2+2}{4}}d{\bf{y}}\right)^{1/2},
\end{equation}
and
\begin{equation}\label{beta def}
\beta(\tau)=\sup_{\sigma \leq \tau}\alpha(\sigma).
\end{equation}
Arguing similarly as in \cite[Proof of Lemma 4.17]{CHH} we see that the inverse Poincare inequality from \cite[Proposition 4.4]{CHH_blowdown} yields that $\lim_{\tau\to -\infty}\beta(\tau)=0$.

\begin{proposition}[barrier estimate]\label{C0 estimate beta3}
There are constants $c>0$ and $C<\infty$ such that
\begin{equation}
 |u ({\bf{y}},\tau)| \leq C\beta(\tau)^{\frac{1}{2}}
\end{equation}
holds for $|{\bf{y}}| \leq c\beta(\tau)^{-\frac{1}{4}}$ and $\tau\ll 0$.
\end{proposition}

\begin{proof}
By parabolic estimates (see \cite[Appendix A]{CHH_blowdown}), there is a constant $K<\infty$ such that
\begin{equation}\label{C_0.bound.y=L_0}
|u({\bf{y}},\tau)|\leq K \beta(\tau)
\end{equation}
holds for $|{\bf{y}}|\leq 2L_0$ and $\tau\ll 0$, where $L_0$ is the constant from the ADS-foliation \eqref{fol_ads}.
Given $\hat \tau\ll 0$, consider  the barrier hypersurface $\Gamma_a=\Gamma_a^1$ from \eqref{fol_bubble} with parameters $\eta=1$ and
\begin{equation}\label{a=root.beta}
a=\frac{c_0}{\sqrt{K\beta(\hat \tau)}}.
\end{equation}
If we choose $c_0$ small enough, then by \cite[Lemma 4.4]{ADS1} the profile function $u_a$ of the ADS-shrinker $\Sigma_a$ satisfies
\begin{equation}
u_a(L_0-1) \leq \sqrt{2}-K\beta(\hat \tau).
\end{equation}
Combining this with \eqref{C_0.bound.y=L_0}, the inner barrier principle from Proposition \ref{prop_inner.barrier} implies that $\Gamma_a$ is enclosed by $\bar M_\tau$  for $|{\bf{y}}| \geq L_0$ and $\tau \leq \hat\tau$. Since  $u_a(\sqrt{a})^2\geq 2-2/a$ (see e.g. \cite[Equation (195)]{CHH}), this yields
\begin{equation}
\left(\sqrt{2}+u({\bf{y}},\hat\tau)\right)^2 \geq 2-2/a
\end{equation}
for $|{\bf{y}}|\in [L_0,\sqrt{a}-1]$. Hence, remembering \eqref{a=root.beta} we conclude that 
\begin{equation}
 u ({\bf{y}},\tau) \geq -C\beta(\tau)^{\frac{1}{2}}
\end{equation}
holds for $|{\bf{y}}| \leq c\beta(\tau)^{-\frac{1}{4}}$ and $\tau\ll 0$.
Finally, by convexity, using also \eqref{C_0.bound.y=L_0}, this lower bound implies a corresponding upper bound. This concludes the proof of the proposition.
\end{proof}

We now define
\begin{equation}\label{eq_better_rho}
\rho(\tau):=\beta(\tau)^{-\frac{1}{5}}
\end{equation}
Then, Proposition \ref{C0 estimate beta3} (barrier estimate) and standard interior Schauder estimates give
\begin{equation}
\|u(\cdot,\tau)\|_{C^4( B_{2\rho(\tau)}(0))}\leq \rho(\tau)^{-2}
\end{equation}
for $\tau\ll 0$. Moreover, thanks to \cite[Theorem 1.10]{CHH_blowdown} the neutral eigenfunctions dominate\footnote{Thanks to the $\mathrm{SO}(2)$-symmetry the fine tuning rotation $S(\tau)$ from \cite[Proposition 4.1]{CHH_blowdown} is simply the identity matrix.} and we thus have
\begin{equation}
\left|\frac{d}{d\tau} \alpha^2\right| = o(\alpha^2).
\end{equation}
This implies
\begin{equation}
-\rho(\tau)\leq \rho'(\tau)\leq 0
\end{equation}
for $\tau \ll 0$, i.e. $\rho$ is an admissible graphical radius function.\\

We now work with the truncated graph function
\begin{equation}
\hat{u}({\bf{y}},\tau)=u({\bf{y}},\tau)\chi \left( \tfrac{|{\bf{y}}|}{\rho(\tau)}\right),
\end{equation}
where $\rho$ denotes the improved graphical radius from equation \eqref{eq_better_rho}.

\begin{proposition}[evolution equation]\label{lemma_taylor}
The function $\hat{u}$ satisfies
\begin{equation}
\partial_\tau \hat u = \mathcal L \hat u  -\tfrac{1}{2\sqrt 2} \hat u ^2 + E,
\end{equation}
where the error term can be estimated by
\begin{align}\label{est_eq}
|  E|\leq &C\chi |  u|^3+C\chi |\nabla u|^2 |\nabla^2u|+ C |\chi'|\rho^{-1} \big(|\nabla u|+|{\bf{y}}||u|\big)\nonumber\\
&+ C |\chi''|\rho^{-2}|u|+ C \chi(1-\chi)\big(|u|^2+|\nabla^2u|^2\big).
\end{align}
\end{proposition}

\begin{proof}
We compute
\begin{align}
\left|\partial_\tau \hat u-\chi \partial_\tau  u\right|= |{\bf{y}}||\rho'\rho^{-2}|\chi'u|\leq C |\chi'|\rho^{-1}|{\bf{y}}||u|,
\end{align}
and
\begin{equation}
\left|\mathcal{L} \hat u-\chi \mathcal{L}  u\right|= \left|u\Delta \chi +2\nabla u\cdot\nabla \chi-\tfrac{1}{2}u {\bf{y}}\cdot \nabla \chi\right|
\leq C |\chi'|\rho^{-1} \big(|\nabla u|+|{\bf{y}}||u|\big)+ C |\chi''|\rho^{-2}|u|.
\end{equation}
Moreover, we have
\begin{align}
&\left|-\frac{\hat u^2}{2\sqrt{2}}+\frac{\chi u^2}{2\sqrt{2}}\right|\leq \chi(1-\chi)u^2, &&
 \left|\frac{\chi u^3}{4+2\sqrt{2} u}\right|\leq \chi |u|^3.
\end{align}
Together with \eqref{u.equation} and \eqref{pointwiseE} (with $\rho_0$ replaced by $\rho$) this yields the desired result.
\end{proof}

We now consider the neutral eigenfunction
\begin{align}
\psi_0  =  2^{-\frac{3}{2}} (\tfrac{e}{2\pi})^{\frac{1}{4}}(y_2^2-2),
\end{align}
which is normalized with respect to the Gaussian inner product $\langle \cdot,\cdot\rangle_{\mathcal{H}}$. Here, for $\theta$-independent functions the Gaussian inner product is given by
\begin{equation}
\langle f,g\rangle_{\mathcal{H}}=\int f({\bf{y}}) g({\bf{y}}) (8\pi)^{-\frac{1}{2}}e^{\frac{-|{\bf{y}}|^2+2}{4}}d{\bf{y}}.
\end{equation}
We now define 
\begin{equation}\label{alpha0.def}
\alpha_0=\langle \hat u, \psi_0\rangle_{\mathcal{H}}.
\end{equation}
Then, by \cite[Theorem 1.7]{CHH_blowdown} we have
\begin{equation}\label{hatu_expansion}
\hat u = \alpha_0\psi_0 + o(|\alpha_0|)
\end{equation}
in $\mathcal{H}$-norm.

\begin{lemma}[error estimate]\label{lemma_error_est} 
The error term $E$ from Proposition \ref{lemma_taylor} (evolution equation) satisfies the estimate
\begin{equation}
|\langle E,\psi_0 \rangle_{\cH} |\leq C \beta(\tau)^{2+\frac{1}{5}}
\end{equation}
for $\tau\ll 0$.
\end{lemma}

\begin{proof}
Using the inverse Poincare inequality from \cite[Proposition 4.4]{CHH_blowdown}, the argument from \cite[Proof of Proposition 4.21]{CHH} applies.
\end{proof}

\begin{proposition}[evolution of expansion coefficient]\label{prop_ODEs}
The coefficient $\alpha_0$ from the expansion \eqref{hatu_expansion} satisfies
\begin{align}
\tfrac{d}{d\tau} \alpha_0 &= -(\tfrac{e}{2\pi})^{\frac{1}{4}}\alpha_0^2  + o(\beta^2).\label{ODE1}
\end{align}
\end{proposition}

\begin{proof}
Using Proposition \ref{lemma_taylor} (evolution equation) and $\mathcal{L}\psi_0=0$ we see that
\begin{align}
\tfrac{d}{d\tau} \alpha_0 &= \langle \partial_\tau \hat u , \psi_0 \rangle_{\cH}=\langle \mathcal L \hat u  -\tfrac{1}{2\sqrt 2} \hat u ^2+  E , \psi_0  \rangle_{\cH}=\langle -\tfrac{1}{2\sqrt{2}} \hat u ^2 +  E , \psi_0 \rangle_{\cH}.
\end{align}
Together with \eqref{hatu_expansion} and Lemma \ref{lemma_error_est} (error estimate) this implies
\begin{equation}
\tfrac{d}{d\tau} \alpha_0=-\tfrac{1}{2\sqrt{2}} c\alpha_0^2+o(\beta^2),
\end{equation}
where
\begin{equation}
c= \int \psi_0^3 (8\pi)^{-\frac{1}{2}} e^{-\frac{|{\bf{y}}|^2+2}{4}}d{\bf{y}}.
\end{equation}
Computing $c$ yields the desired result.
\end{proof}

\begin{theorem}[inwards quadratic bending]\label{thm_L2.convergence}
The function $\hat{u}$ satisfies
\begin{equation}
\lim_{\tau\to -\infty}|\tau|\hat u({\bf{y}},\tau)=-\frac{y_2^2-2}{2\sqrt{2}}
\end{equation}
in $\mathcal{H}$-norm. In particular, for $\tau\ll 0$ we have
\begin{equation}
\|\hat u(\cdot,\tau)\|_{\mathcal{H}}=(2\pi /e)^{\frac{1}{4}}|\tau|^{-1}+o(| \tau|^{-1}).
\end{equation}
\end{theorem}

\begin{proof}
Let
\begin{equation}
\beta_0(\tau):= \sup_{\sigma\leq \tau} |\alpha_0(\sigma)|,
\end{equation}
where $\alpha_0$ is defined in \eqref{alpha0.def}. By Proposition \ref{prop_ODEs} (evolution of the expansion coefficient) there is some $\tau_\ast>-\infty$ so that for $\tau\leq \tau_\ast$ we have 
\begin{equation}
\left |\tfrac{d}{d\tau} \alpha_0 +(\tfrac{e}{2\pi})^{\frac{1}{4}}\alpha_0^2 \right| \leq \tfrac{1}{10}\beta_0^2.
\end{equation}
Suppose that at some $\tau_0\leq \tau_\ast$ we have $\beta_0(\tau_0)=|\alpha_0(\tau_0)|$. Then,
\begin{equation}
-\tfrac{d}{d\tau} \alpha_0(\tau_0)\geq (\tfrac{e}{2\pi})^{\frac{1}{4}}\alpha_0^2(\tau_0)-\tfrac{1}{10}\beta_0^2(\tau_0) \geq \tfrac12 \alpha_0^2(\tau_0)>0,
\end{equation}
implies that there exists some small $\delta>0$ such that $\alpha_0(\tau)>\alpha_0(\tau_0)$ holds for $\tau \in (\tau_0-\delta,\tau_0)$. Since $\beta_0(\tau_0)=|\alpha_0(\tau_0)|\geq |\alpha_0(\tau)|$ for $\tau\leq \tau_0$, we thus have $\alpha_0(\tau_0)<0$.

\bigskip

Next, we choose any time $\tau_1\leq \tau_\ast$ satisfying $\beta_0(\tau_1)=|\alpha_0(\tau_1)|$, and an interval $I=[\tau_1,\tau']\subset [\tau_1,\tau_\ast]$ such that $\frac{d}{d\tau}\alpha_0(\tau)\leq 0$ for $\tau \in I$. Since $\alpha_0(\tau_1)<0$, we have $-\alpha_0(\tau)=\beta_0(\tau)$ for all $\tau \in I$. Moreover,
\begin{equation}
-\tfrac{d}{d\tau} \alpha_0 \geq \tfrac12 \alpha_0^2(\tau)\geq \tfrac12 \alpha_0^2(\tau_1)>0
\end{equation}
holds for all $\tau\in I$. Therefore, if $\tau'<\tau_\ast$ we can keep extending $\tau'$ until $\tau'=\tau_\ast$. Namely, $-\alpha_0(\tau)=\beta_0(\tau)$ holds for all $\tau \in [\tau_1,\tau_\ast]$. Since $\tau_1$ was arbitrarily, we infer that $-\alpha_0(\tau)=\beta_0(\tau)$ for all $\tau \leq \tau_\ast$. Namely, we have $\alpha_0<0$ and 
\begin{align}
\tfrac{d}{d\tau} \alpha_0 &= -(\tfrac{e}{2\pi})^{\frac{1}{4}}\alpha_0^2  + o(\alpha_0^2),
\end{align}
for all $\tau\leq \tau_\ast$. Integrating this ODE yields
\begin{align}
\alpha_0(\tau)=\frac{-(2\pi /e)^{\frac{1}{4}}+o(1)}{| \tau|}
\end{align}
for $\tau\ll \tau_\ast$. Together with \eqref{hatu_expansion} this implies the assertion.
\end{proof}

\bigskip

Recall that in contrast to \cite{ADS1,ADS2}, where only a single solution was considered, we need estimates for families that are uniform depending only on the quadratic bending in the central region. As opposed to the introduction, we will first work with the following stronger notion of $\kappa$-quadraticity:

\begin{definition}[strongly $\kappa$-quadratic]\label{def_mu_qudratic_str}
We say that a noncollapsed translator $M$ in $\mathbb{R}^4$, normalized as above, that is neither a 3d round bowl nor $\mathbb{R}\times2d$-bowl, is \emph{strongly $\kappa$-quadratic from time $\tau_0$} if
\begin{enumerate}[(i)]
\item $\rho(\tau)=|\tau|^{1/10}$ is an admissible graphical radius function for $\tau \leq \tau_0$, and 
\item the truncated graph function $\hat{u}({\bf{y}},\tau)=u({\bf{y}},\tau)\chi \left( \tfrac{|{\bf{y}}|}{\rho(\tau)}\right),
$ satisfies the estimate  \begin{equation}\label{mu_quad_back}
\left\|\hat{u}({\bf{y}},\tau)+\frac{y_2^2-2}{2\sqrt{2}|\tau|}\right\|_{\mathcal{H}}\leq \frac{\kappa}{|\tau|}\;\;\;\;\;\textrm{for}\;\tau\leq \tau_0.
\end{equation}
\end{enumerate}
\end{definition}

\begin{corollary}[strong $\kappa$-quadraticity]\label{cor_strong_kappa}
For every $\kappa>0$ and every noncollapsed translator $M$ in $\mathbb{R}^4$, normalized as above, that is neither a 3d round bowl nor $2d$-bowl $\times\mathbb{R}$, there exists $\tau_\ast=\tau_\ast(\kappa,M)>-\infty$ such that $M$ is strongly $\kappa$-quadratic from any time $\tau_0\leq \tau_{\ast}$.
\end{corollary}

\begin{proof}
By Theorem \ref{thm_L2.convergence} (inwards quadratic bending) and the inverse Poincare inequality from \cite[Proposition 4.4]{CHH_blowdown} we have $\beta(\tau)\sim |\tau|^{-1}$. Together with the above, this implies the assertion.
\end{proof}

Finally, in the parabolic region the $L^2$-estimate from Theorem \ref{thm_L2.convergence} can be upgraded to an $L^\infty$-estimate. Moreover, this estimate kicks in at time $\tau_0$ and is uniform depending only on $\kappa$:

\begin{proposition}[uniform asymptotics in parabolic region]\label{cor_asymptotic.middle}
For every $\eps>0$ there exist constants $\kappa>0$ and ${\tau}_{\ast}>-\infty$, such that if $M$ is strongly $\kappa$-quadratic from time $\tau_0\leq \tau_{\ast}$, then we have the estimate
\begin{align}
\left|u({\bf{y}},\tau)+\frac{y^2_2-2}{2\sqrt{2} \,|\tau|}\right|\leq \frac{\eps}{|\tau|}
\end{align}
for $\tau \leq \tau_0$ and $|{\bf{y}}| \leq \eps^{-1}$.
\end{proposition}

\begin{proof}
Consider the difference
\begin{equation}
D({\bf{y}},\tau):=\hat u({\bf{y}},\tau)-\tfrac{y^2_2-2}{2\sqrt{2}\tau}.
\end{equation}
If $M$ is $\kappa$-quadratic from time $\tau_0$, then by definition we have
\begin{equation}\label{difference.H-norm}
\|D\|_{\mathcal{H}}\leq \frac{\kappa}{|\tau|}
\end{equation}
for every $\tau \leq \tau_0$. On the other hand, by Theorem \ref{thm_L2.convergence} (inwards quadratic bending) and the parabolic estimates from \cite[Theorem A.1]{CHH_blowdown} there exist a constant $C=C(\eps)<\infty$, such that
\begin{equation}
\sup_{|{\bf{y}}|\leq 2\eps^{-1}}|u({\bf{y}},\tau)|\leq  C |\tau|^{-1},
\end{equation}
for $\tau\leq\tau_0$, provided $\tau_0\leq\tau_\ast(\eps)$. Therefore, standard interior estimates give
\begin{equation}\label{difference.H^3-norm}
\|D(\cdot,\tau)\|_{W^{3,2}(B(0,{\eps^{-1}}))}\leq C|\tau|^{-1}
\end{equation}
for such  $\tau$. Applying Agmon's inequality with \eqref{difference.H-norm} and \eqref{difference.H^3-norm} we conclude that
\begin{equation}
\|D(\cdot,\tau)\|_{L^{\infty}(B(0,{\eps^{-1}}))}\leq \frac{\eps}{|\tau|},
\end{equation}
provided $\kappa$ is sufficiently small. This proves the proposition.
\end{proof}

\bigskip

\subsection{Sharp asymptotics in intermediate region} To capture the intermediate region we consider the function
\begin{equation}
\bar v(z,\tau)=\sqrt{2}+u(0,|\tau|^{1/2}z,\tau).
\end{equation}
We will show that $\bar v(z,\tau)$ converges to $\sqrt{2-z^2}$ uniformly on each compact interval in $(-\sqrt{2},\sqrt{2})$. More precisely, we make this convergence explicit in the parameter $\kappa$ of strong $\kappa$-quadraticity:

\begin{proposition}[{intermediate region}]\label{prop_intermediate} For every $\eps>0$ there exist $\kappa>0$ and $\tau_{\ast}>-\infty$, such that if $M$ is strongly $\kappa$-quadratic from time $\tau_0 \leq \tau_{\ast}$, then on $I=[-\sqrt{2}+\eps,\sqrt{2}-\eps]$ we have
\begin{equation}
\sup_{z\in I,\;\tau \leq \tau_0}\left|\bar v(z,\tau) - \sqrt{2-z^2}\,\right|\leq \eps.
\end{equation}
\end{proposition}

\begin{proof}
We will adapt the proof from \cite[Section 6]{ADS1} to our setting.\\

\noindent\underline{Lower bound:} 
By \cite[Lemma 4.4]{ADS1}, there exist some $a_0\geq 1$ and an increasing function $M:(a_0,\infty)\to (100,\infty)$ with $\displaystyle \lim_{a\to \infty}M(a)=\infty$ such that the profile function $u_a$ of $\Sigma_a$ satisfies
\begin{equation}\label{barrier.asymptotic}
 u_a(y)\leq \sqrt{2}-\frac{y^2-3}{\sqrt{2}\, a^2}
\end{equation}
for $0\leq y\leq M(a)$. Let  $\delta>0$ be such that $\sqrt{\tfrac{2}{1+\delta}}\geq \sqrt{2}-\eps$, and define 
\begin{equation}\label{hata.def}
\hat a(\tau) =\sqrt{\frac{2 |\tau|}{1+\delta}}.
\end{equation}
Choose $\tau_{\ast}$ such that
\begin{equation}
\delta^{-1} \leq  \min\{M(|{\tau}_{\ast}|^{\frac{1}{2}}),|{\tau}_{\ast}|^{\frac{1}{2}-\frac{1}{100}}\},
\end{equation}
and such that 
\begin{equation}\label{eq_ADS4.3}
\left|\sqrt{2\left(1-\tfrac{y^2}{a(\tau)^2}\right)}-u_{a(\tau)}(y)\right| \leq \eps
\end{equation}
hold for every $\tau \leq {\tau}_{\ast}$ and $|y| \leq a(\tau)$, which is possible in light of \cite[Lemma 4.3]{ADS1}.

By Proposition \ref{prop_inner.barrier} (barriers) for each fixed $\hat \tau\leq \tau_0\leq\tau_{\ast}$ the static hypersurface $\Gamma_{\hat a(\hat \tau)}^{3\delta}\subset \mathbb{R}^4$ plays the role of an inner barrier in the region $|{\bf{y}}| \geq \delta^{-1}$. Since $\Gamma_{\hat a(\hat \tau)}^{3\delta}\subset \mathbb{R}^4$ is compact and enclosed by the cylinder $\displaystyle  \mathbb{R}^2\times S^1(\sqrt{2})=\lim_{\tau\to -\infty} \bar M_\tau$, this yields
\begin{equation}
\sqrt{2}+u({\bf{y}},\tau)\geq u_{\hat a(\hat \tau)}(|{\bf{y}}|-3\delta)
\end{equation}
for ${\bf{y}}\in \Omega_\tau\setminus B_{\delta^{-1}}(0)$ and $\tau\leq \hat \tau$, provided the boundary condition
\begin{equation}\label{Dirichlet}
\sqrt{2}+u({\bf{y}},\tau)\geq u_{\hat a(\hat \tau)}(\delta^{-1}-3\delta)
\end{equation}
holds for $|{\bf{y}}|=\delta^{-1}$ and $\tau\leq \hat \tau$. To check this boundary condition, note that \eqref{barrier.asymptotic} by our choice of constants implies
\begin{align}
u_{\hat a(\hat \tau)}(\delta^{-1}-3\delta)-\sqrt{2} \leq  -\frac{[\delta^{-1}-3\delta]^2-3}{\sqrt{2}\,\hat a^2(\hat \tau)}\leq   -\frac{\delta^{-2}}{2\sqrt{2}\,|\hat \tau|}.
\end{align}
Moreover, using also Corollary \ref{cor_asymptotic.middle} (uniform asymptotics in parabolic region) we see that if our solution is $\kappa$-quadratic from time $\tau_0\leq\tau_\ast$, for $\kappa$ sufficiently small, then (after reducing $\tau_{\ast}$ to be the minimum of its current value and the value from Proposition \ref{cor_asymptotic.middle}) we have
\begin{equation}\label{asymptotic.for.barrier}
 u\left( {\bf{y}},\tau\right)\geq -\frac{y_2^2-2+\delta}{2\sqrt{2}|\tau|}\geq -\frac{|{\bf{y}}|^2}{2\sqrt{2}|\tau|}\geq -\frac{\delta^{-2}}{2\sqrt{2}|\hat \tau|}
\end{equation}
for $|{\bf{y}}|=\delta^{-1}$ and $\tau\leq \hat \tau$. Thus, the boundary condition \eqref{Dirichlet}  indeed holds for $|{\bf{y}}|=\delta^{-1}$ and $\tau\leq \hat \tau$. Consequently,
\begin{equation}
\bar v( z,\hat \tau)=\sqrt{2}+u(0,|\hat \tau|^{\frac{1}{2}}z,\tau) \geq u_{\hat a(\hat \tau)}(|\hat \tau|^{\frac{1}{2}}|z|)
\end{equation}
holds for $\hat \tau \leq \tau_{0}$ and $z$ satisfying $|\hat \tau|^{-\frac{1}{2}}\delta^{-1}\leq  |z|\leq  |\hat \tau|^{-\frac{1}{2}}\hat a(\hat \tau) $.  As $ |\hat \tau|^{-\frac{1}{2}}\hat a(\hat \tau)=\sqrt{\tfrac{2}{1+\delta}}\geq \sqrt{2}-\eps$, while by the choice ot ${\tau}_{\ast}$ one has $|\hat \tau|^{-\frac{1}{2}}\delta^{-1} \geq |\hat \tau|^{-1/100}$, we obtain
\begin{equation}
\bar v( z, \tau)\geq u_{\hat a( \tau)}(|\tau|^{\frac{1}{2}}|z|)
\end{equation}
for $|z|\in [|\tau|^{-\frac{1}{100}},2-\varepsilon]$ and $\tau \leq \tau_{0}$. Thus, by \eqref{eq_ADS4.3}, for every $\tau\leq \bar{\tau}_{0}$ and  $|z|\in [| \tau|^{-\frac{1}{100}},  \sqrt{2}-\varepsilon]$ we get
\begin{equation}
\bar v( z, \tau) +\eps \geq \sqrt{ 2-\frac{2|  \tau|z^2}{\hat a^2(  \tau)}}=\sqrt{2-(1+\delta)z^2}\, .
\end{equation}
Finally, since $\bar v$ is concave in $z$, we have $\bar v(z,\tau) \geq \min\{\bar v(| \tau|^{-\frac{1}{100}},\tau),\bar v(-| \tau|^{\frac{1}{100}},\tau)\}$ for $|z|\leq |\tau|^{-\frac{1}{100}}$. Putting things together, we conclude that 
\begin{equation}\label{barv.lower}
\inf_{|z|\leq \sqrt{2}-\eps,\;\tau \leq \tau_0}\left(\,\bar v(z,\tau)- \sqrt{2-z^2}\,\right)\geq -2\eps.
\end{equation}

\bigskip

\noindent\underline{Upper bound:} 
Since $u$ is concave, we have
\begin{align}
u_\tau\leq -\frac{1}{\sqrt{2}+u}+\frac{1}{2}\Big(\sqrt{2}+u-{\bf{y}}\cdot \nabla u\Big).
\end{align}
Thus, $v(y,\tau):=(\sqrt{2}+u(0,y,\tau))^2-2$ satisfies
\begin{align} 
v_\tau\leq v- \frac{1}{2}y v_y.
\end{align}
Hence, for each $\alpha\in \mathbb{R}$ we have
\begin{align} 
\frac{d}{d \tau}\left(e^{-\tau}v(\alpha e^{\frac{\tau}{2}},\tau)\right) \leq 0,
\end{align}
provided $\tau$ is negative enough so that $(0,\alpha e^{\frac{\tau}{2}})\in \Omega_\tau$. Thus, for $\bar \tau \leq \tau$ we get
\begin{equation}
 v(\alpha e^{\frac{\tau}{2}},\tau) \leq  e^{\tau -\bar \tau} v\left(\alpha e^{\frac{\tau}{2}} e^{-\frac{\tau - \bar \tau}{2}},\tau-(\tau-\bar \tau) \right).
\end{equation}
Therefore, for $\sigma\in (0,1]$ we obtain
\begin{equation}\label{vbddab}
 v(y,\tau) \leq  \sigma^{-2} v( \sigma y ,\tau+2\log \sigma).
\end{equation}
On the other hand, by Proposition \ref{cor_asymptotic.middle} (uniform asymptotics in parabolic region), given any $A<\infty$, there exists $\kappa>0$ such that if $M$  is $\kappa$-quadratic from time $\tau_0 \leq {\tau}_{\ast}$, then  
\begin{equation}\label{vequalsth}
v(y,\tau)\leq  |\tau|^{-1}(2-y^2)+A^{-1}|\tau|^{-1}
\end{equation} 
for $|y|\leq A$.
Thus, for $|y|\geq  A$ we obtain
\begin{align}
v(y,\tau)&\leq (y/A)^2 v(\pm A,\tau-2\log (|y|/A))\\
&= - \frac{(1-2A^{-2})y^2}{|\tau|+2\log (|y|/ A)}+A^{-1}(|\tau-2\log (|y|/A)|^{-1}).
\end{align}
This implies
\begin{align}
\bar v(z,\tau)=\sqrt{2+v(|\tau|^{1/2}z,\tau)}\leq \sqrt{2-z^2(1-2A^{-2})+|\tau|^{-1/2}}
\end{align}
uniformly for $|z| \geq A|\tau|^{-\frac{1}{2}}$. In addition, the concavity of $\bar v$ and  \eqref{barv.lower} yield
\begin{align}
\bar v(z,\tau)&\leq 2\bar v(A|\tau|^{-\frac{1}{2}},\tau)-\bar v(2A|\tau|^{-\frac{1}{2}}-z,\tau) \\
&\leq 2\sqrt{2-z^2(1-2A^{-2})+|\tau|^{-1/2}}-{\sqrt{2-z^2}+2\eps\leq  \sqrt{2-z^2}+4\eps} .
\end{align}
for $|z|\leq A|\tau|^{-\frac{1}{2}}$, provided $|\tau|$ is sufficiently large and $A$ is large enough (which happens for $\kappa$ small enough). Hence,
\begin{equation}\label{barv.upper}
\sup_{|z|\leq \sqrt{2}-\eps,\;\tau \leq \tau_{0}}\left(\,\bar v(z,\tau)- \sqrt{2-z^2}\,\right)\leq 4\eps
\end{equation}
This finishes the proof of the proposition.
\end{proof}

\bigskip

\subsection{Sharp asymptotics in terms of level sets}
Let us now reformulate the results from the previous subsections in terms of the level sets.
Recall that, after re-centering our translator $M$ in the $x_3x_4$-plane, the level sets $\Sigma^h=M\cap \{x_1=h\}$ are left invariant by the field $x_3e_4-x_4e_3$. Hence, we can represent the level sets as
\begin{equation}
\Sigma^h=\left\{ (h,x_2,x_3,x_4)\in\mathbb{R}^4 : -d_1(h)\leq x_2 \leq d_2(h),\, (x_3^2+x_4^2)^{1/2} = V(x_2,-h) \right\}.
\end{equation}
The function $V(x,t)$, where $t=-h$, is called the profile function, and is defined for $x\equiv x_2\in [-d^-(h),d^+(h)]$. It vanishes at the endpoints of this interval. We also consider the rescaled profile function $v$ defined by
\begin{equation}
V(x,t)=\sqrt{-t} v(y,\tau)
\end{equation}
where
\begin{equation}
y= \frac{x}{\sqrt{-t}},\qquad \tau = -\log(-t).
\end{equation}
In the tip regions, since $\partial_ y v\neq 0$, we can define $Y(v,\tau)$ as the inverse function of $v(y,\tau)$. In addition, to capture the tips at scale $|\tau|^{-1/2}$, we consider the function
\begin{equation}\label{Z.definition}
Z(\rho,\tau)= |\tau|^{1/2}\left(Y(|\tau|^{-1/2}\rho,\tau)-Y(0,\tau)\right).
\end{equation}

The following theorem shows that the profile function of the level sets of our translator satisfies exactly the same sharp asymptotics as the profile function of the ancient ovals in \cite{ADS1}. An important difference with \cite{ADS1}, where only a single solution is considered, is that our estimates are uniform:

\begin{theorem}[uniform sharp asymptotics assuming strong $\kappa$-quadraticity]\label{thm_unique_asympt}
For every $\eps>0$ there exists $\kappa>0$ and ${\tau}_{\ast}>-\infty$, such that if $M$ is strongly $\kappa$-quadratic from time $\tau_0 \leq \tau_{\ast}$, then for every $\tau \leq \tau_0$ the following holds:
\begin{enumerate}
\item Parabolic region: The renormalized profile function satisfies
\begin{equation}
\left| v(y, \tau)-\sqrt{2}\left(1-\frac{y^2-2}{4 |\tau|}\right) \right| \leq\frac{\eps}{|\tau|} \qquad (|y |\leq \eps^{-1}).
\end{equation}
\item Intermediate region: The function $\bar{v}(z,\tau):=v(|\tau|^{1/2}z,\tau)$ satisfies
\begin{equation}
|\bar{v}(z,\tau)-\sqrt{2-z^2}|\leq \eps,
\end{equation}
on $[-\sqrt{2}+\eps,\sqrt{2}-\eps]$.
\item Tip regions: We have the estimate
\begin{equation}
\| Z(\cdot,\tau)-Z_0(\cdot)\|_{C^{100}(B(0,\eps^{-1}))}\leq \eps,
\end{equation}
where $Z_0(\rho)$ is the profile function of the $2d$-bowl with speed $1/\sqrt{2}$.
\end{enumerate}
In particular, $\Sigma^h$ satisfies the estimate
\begin{equation}
\left| \frac{d^\pm(h)}{\sqrt{2h\log h}}-1 \right| \leq \eps.
\end{equation}
\end{theorem}

\begin{proof} By definition of the level sets, we have
\begin{equation}
\Sigma^h = (M-he_1)\cap\{x_1=0\}.
\end{equation}
Hence, describing $\Sigma^h$ amounts to describing the $x_1=0$ section of the time $t=-h$ slice of the flow $M_t=M+te_1$, which has already been done in the previous subsections. Specifically, observing that $v(y,\tau)=\sqrt{2}+u(0,y,\tau)$ and applying Proposition \ref{cor_asymptotic.middle} (uniform asymptotics in parabolic region) we obtain
\begin{equation}
\left| v(y, \tau)-\sqrt{2}\left(1-\frac{y^2-2}{4 |\tau|}\right) \right| \leq\frac{\eps}{|\tau|} \qquad (|y |\leq \eps^{-1}),
\end{equation}
which proves the first assertion. Next, by Proposition \ref{prop_intermediate} (intermediate region) we have
\begin{equation}
\sup_{\tau\leq \tau_0}\sup_{|z|\leq \sqrt{2}-\eps}|\bar{v}(z,\tau)-\sqrt{2-z^2}|\leq \eps,
\end{equation}
which proves the second assertion. In particular, scaling back to the original surface $\Sigma^h$ this implies
\begin{equation}\label{diameter_sharp}
\left| \frac{d^\pm(h)}{\sqrt{2h\log h}}-1 \right| \leq \eps.
\end{equation}
Recall that by Proposition \ref{lemma_asympt_slope} (asymptotic slope and tip point) we have $H=-\langle e_1,\nu\rangle\to 0$ as $h\to \infty$. Denote by $H^{\pm}_{\textrm{tip}}(h)$ the mean curvature of $M$ at the point $p_h^\pm$ at level $h$ with maximal respectively minimal $x_2$-value.
Using the above and Hamilton's Harnack inequality \cite{Hamilton_Harnack}, similarly as in \cite[Section 7.2]{ADS1}, we get
\begin{equation}\label{tip.mean.curvature}
\left|\sqrt{\frac{h}{\log h}}H^{\pm}_{\textrm{tip}}(h)-\frac{1}{\sqrt{2}}\right| \leq 2\eps.
\end{equation}
Now, suppose towards a contradiction there is a sequence $M^i$ that is $\kappa_i$-quadratic from time $\tau_{0,i}$ with $\kappa_i\rightarrow 0$ and $\tau_{0,i}\rightarrow -\infty$, but such that at some time $\tau_i \leq \tau_{0,i}$  the function $Z_i(\rho,\tau)$ is not $\eps$-close in $C^{100}(B(0,\eps^{-1}))$ to $Z_0(\rho)$, the profile function of the $2d$-bowl with speed $1/\sqrt{2}$. Let $h_i=e^{-\tau_i}\rightarrow \infty$ be the height of the tips. Using the theory of noncollapsed flows from \cite{HaslhoferKleiner_meanconvex} we see that for $i\to \infty$ the sequence of flows that is obtained from $M_t^i$ by shifting $(p^\pm_{h_i},0)$ to the origin and parabolically rescaling by $\sqrt{\log( h_i)/h_i}$ converges to an ancient noncollapsed flow $M_t^\infty$ that splits isometrically as $M_t^\infty=\mathbb{R}\times N_t^\infty$. By construction, $N_t^\infty\subset\mathbb{R}^3$ is a noncompact ancient noncollapsed flow, whose time zero slice is contained in a halfspace and with mean curvature $1/\sqrt{2}$ at the base point. Hence, the classification by Brendle-Choi \cite{BC} implies that $N_t^\infty$ is the rotationally symmetric translating bowl soliton with speed $1/\sqrt{2}$. This yields that $Z_i(\rho,\tau)\to Z_0(\tau)$ smoothly and locally uniformly. For $i$ large enough this contradicts the assumption that $Z_i$ is not $\eps$-close to $Z_0$, and thus finishes the proof of the theorem.
\end{proof}

\bigskip

\subsection{Uniform sharp asymptotics from one time} 

In this subsection, we show that one can conclude the sharp asymptotics from information about the cylindrical profile function of the flow at the time $\tau_0$ itself. This will be used in the next section in the continuity method along the HIMW class.  

Recall that the renormalized profile function $u=u(y_1,y_2,\tau)$ from the bubble-sheet analysis and the renormalized profile function $v=v(y,\tau)$ of the level sets $M\cap \{x_1=e^{-\tau}\}$ are related by
\begin{equation}\label{eq_rel_uv}
v(y,\tau)=\sqrt{2}+u(0,y,\tau)\, .
\end{equation}
In the analysis of the function $v$ we work with the Hilbert space  $\fH:=L^2(\mathbb{R},e^{-y^2/4} dy)$, while on the other hand in the analysis of the function $u$ we worked with the Hilbert space $\mathcal{H}\cong \fH\otimes \fH$. 

\begin{definition}[$\kappa$-quadratic]\label{def_mu_qudratic}
We say that a noncollapsed translator $M\neq \textrm{Bowl}_3,\mathbb{R}\times\textrm{Bowl}_2$ in $\mathbb{R}^4$, normalized as above and centered such that $\fp_{+}(v_{\cC}(\tau_0)-\sqrt{2})=0$, is \emph{$\kappa$-quadratic at time $\tau_{0}$} if
\begin{enumerate}[(i)]
\item the cylindrical profile function $v_{\cC}=v\varphi_{\cC}(v)$ at time $\tau_0$ satisfies
\begin{equation}\label{mu_quad_back}
\left\|v_{\cC}(y,\tau_{0})-\sqrt{2}+\frac{y^2-2}{2\sqrt{2}|\tau_{0}|}\right\|_{\fH}\leq \frac{\kappa}{|\tau_{0}|},
\end{equation}
\item and the bubble-sheet graph function $u$  satisfies
\begin{equation}\label{mu_quad_rad}
\sup_{\tau\in [2\tau_0,\tau_0]}\|u(\cdot,\cdot,\tau)\|_{C^4(B(0,2|\tau_0|^{1/100})} \leq |\tau_0|^{-1/50}.
\end{equation}
\end{enumerate}
\end{definition}

In contrast to Definition \ref{def_mu_qudratic_str} (strongly $\kappa$-quadratic) here we work with the smaller Hilbert space $\fH$, and more importantly we only prescribe the behavior of the function $v_{\cC}$ at the time $\tau_0$ itself as opposed to prescribing the behavior at all times  $\tau\leq \tau_0$.
The main goal of this subsection is to prove:

\begin{theorem}[$\kappa$-quadraticity implies strong $\kappa$-quadraticity]\label{thm_precise_from_one_time}
For every $\kappa>0$ sufficiently small there exists $\tau_{\ast}>-\infty$ with the following significance. If $M$ is $\frac{\kappa}{5}$-quadratic at time $\tau_0$ for some $\tau_0\leq \tau_\ast$, then $M$ is strongly-$\kappa$ quadratic from time $\tau_0$.  
\end{theorem}

In particular, by definition of strong $\kappa$-quadraticity, we get that $\rho(\tau)=|\tau|^{1/10}$ is an admissible graphical radius for $\tau \leq \tau_0$, so the solution is graphical on a much larger scale than we initially assumed.\\

To prove  Theorem \ref{thm_precise_from_one_time}, we will start with a lemma that upgrades the information of $v$ in $\fH$ to information about $u$ in $\mathcal{H}$, essentially by exploiting the fact that $\partial_{y_1}u$ is very small on our bubble-sheet. Before stating the lemma, we recall that we can decompose
\begin{equation}\label{decomp_Hilbert}
\cH=\cH_{+}\oplus \cH_{0}\oplus\cH_{-},
\end{equation}
according to the positive, neutral and negative eigenspaces of $\mathcal{L}$, and that we denote the corresponding projections by $\cP_{+},\cP_{0}$ and $\cP_{-}$. Moreover, in the following we work with the truncated function
\begin{equation}\label{eq_trunc_bubble}
\hat{u}(y_1,y_2,\tau):= u(y_1,y_2,\tau) \chi \left(\tfrac{|(y_1,y_2)|}{\rho(\tau)}\right),
\end{equation}
where $\rho(\tau)$ is a suitable graphical radius function that will be fixed below.

\begin{lemma}[upgrade to bubble-sheet]\label{prop_cut_compatibility}
For every $\kappa>0$ sufficiently small there exists $\tau_{\ast}>-\infty$ with the following significance. If  $M$ is $\frac{\kappa}{5}$-quadratic at time $\tau_0$ for some $\tau_0\leq\tau_\ast$, then 
\begin{equation}\label{precise_up_to}
\left\|\hat{u}(\tau_0)+\frac{y_2^2-2}{2\sqrt{2}|\tau_0|}\right\|_{\cH} \leq \frac{\kappa}{4|\tau_0|},
\end{equation}
and
\begin{equation}\label{U_plus_small}
\left\| \cP_{+} \hat{u}(\tau_0)  \right\|_{\cH} \leq \frac{1}{|\tau_0|^{100}}.
\end{equation}
\end{lemma}

\begin{proof} First observe that the unit normal at 
\begin{equation}
p=(y_1,y_2,(\sqrt{2}+u(y_1,y_2,\tau_0))\cos\theta, (\sqrt{2}+u(y_1,y_2,\tau_0))\sin\theta)\in \bar{M}_{\tau_0}
\end{equation}
is given by 
\begin{equation}
\nu=\frac{1}{\sqrt{1+(\partial_{y_1}u)^2+(\partial_{y_2}u)^2}}\left(-\partial_{y_1}u,-\partial_{y_2}u,\cos \theta,\sin \theta\right).
\end{equation}
Now, if $|(y_1,y_2)|\leq 2\rho(\tau_0)$, then $(\partial_{y_1}u)^2+(\partial_{y_2}u)^2\ll 1$, hence in particular $|\partial_{y_1} u| \leq 2 |\langle \nu, e_1\rangle|$.
On the other hand, since $p$ lies on a bubble-sheet, at the point $P$ on the unrescaled translator corresponding to $p$, we have $H(P)\leq e^{\tau_0/2}$. Together with the translator equation $H=\langle e_1,\nu \rangle$ and \eqref{mu_quad_rad} this yields
\begin{equation}\label{y_1der}
\sup_{|(y_1,y_2)|\leq 2|\tau_0|^{1/100}} |\partial_{y_1}u| \leq 2e^{\tau_0/2}.
\end{equation}
Now, remembering \eqref{eq_rel_uv}  and integrating this gradient estimate we infer that
\begin{equation}\label{pointwise_diff}
\sup_{|(y_1,y_2)|\leq |\tau_0|^{1/100}}\left|\sqrt{2}+ \hat{u}(y_1,y_2,\tau_0)-v_{\cC}(y_2,\tau_0)\right| \leq e^{\tau_0/3},
\end{equation} 
provided $\tau_0 \leq \tau_{\ast}$, where we also used that $u=\hat{u}$ and $v=v_{\cC}$ in the region under consideration.
On the other hand, we have the Gaussian tail estimates
\begin{equation}\label{error_diff}
\int_{|y|\geq \tfrac{1}{2} |\tau_0|^{1/100}}\left(v_{\cC}(y,\tau_{0})-\sqrt{2}+\frac{y^2-2}{2\sqrt{2}|\tau_{0}|}\right)^2e^{-\tfrac{y^2}{4}}dy \leq \frac{\kappa}{|\tau_0|^{100}},
\end{equation}
and
\begin{equation}
\int_{\max\{|y_1|,|y_2|\}\geq \tfrac{1}{2} |\tau_0|^{1/100}}\left(\hat{u}(y_1,y_2,\tau_{0})+\frac{y_2^2-2}{2\sqrt{2}|\tau_{0}|}\right)^2e^{-\tfrac{y_1^2+y_2^2}{4}}dy_1dy_2 \leq  \frac{\kappa}{|\tau_0|^{100}}.
\end{equation}
Combining the above inequalities and choosing $\tau_\ast=\tau_\ast(\kappa)$ sufficiently negative,  we infer that 
\begin{equation}
\left\|\hat{u}(\tau_0)+\frac{y_2^2-2}{2\sqrt{2}|\tau_0|}\right\|_{\cH} \leq \frac{1}{(2e)^{1/4}}\left\|v_{\cC}(\tau_{0})-\sqrt{2}+\frac{y^2-2}{2\sqrt{2}|\tau_{0}|}\right\|_{\fH}+\frac{\kappa}{20|\tau_0|}\,,
\end{equation}
where the factor $\frac{1}{(2e)^{1/4}}$ comes from the different normalizations of the two Hilbert spaces. Taking also into account the assumption that $M$ is $\tfrac{\kappa}{5}$-quadratic at time $\tau_0$, this proves \eqref{precise_up_to}. \bigskip

To derive \eqref{U_plus_small}, we first recall that $\mathcal{H}_{+}$ is spanned by the eigenfunctions $\phi_0=1,\phi_1=y_1,\phi_2=y_2$. Now, thanks to the normalization $\fp_{+}(v_{\cC}(\tau_0)-\sqrt{2}))=0$, for every $y_1$ and $i=0,1,2$
we have
\begin{equation}
\int_{-\infty}^{\infty} \left(v_{\cC}(y_2,\tau_0)-\sqrt{2}\right)\phi_i \, e^{-\frac{y_2^2}{4}}dy_2=0.
\end{equation}
Moreover, using the pointwise estimate \eqref{pointwise_diff} we see that
\begin{multline}
\left| \int_{\max\{|y_1|,|y_2|\}\leq \tfrac{1}{2} |\tau_0|^{1/100}}  \hat{u}(y_1,y_2,\tau_0)\phi_i\, e^{-\tfrac{y_1^2+y_2^2}{4}}dy_1dy_2 \right. \\
- \left.  \int_{\max\{|y_1|,|y_2|\}\leq \tfrac{1}{2} |\tau_0|^{1/100}}   (v_{\cC}(y_2,\tau_0)-\sqrt{2})\phi_i \, e^{-\tfrac{y_1^2+y_2^2}{4}}dy_1dy_2 \right| \leq \frac{\kappa}{|\tau_0|^{100}}.
\end{multline}
Furthermore, similarly as before we have the Gaussian tail estimates
\begin{equation}
|\tau_0|^{1/100}\sup_{|y_1|\leq \tfrac{1}{2}|\tau_0|^{1/100}}\int_{|y_2|\geq \tfrac{1}{2} |\tau_0|^{1/100}}\left|\left(v_{\cC}(y_2,\tau_{0})-\sqrt{2}\right) \phi_i \right| e^{-\tfrac{y_2^2}{4}}dy_2 \leq \frac{\kappa}{|\tau_0|^{100}},
\end{equation}
and
\begin{equation}
\int_{\max\{|y_1|,|y_2|\}\geq \tfrac{1}{2} |\tau_0|^{1/100}}\left|\hat{u}(y_1,y_2,\tau_{0})\phi_i\right|e^{-\tfrac{y_1^2+y_2^2}{4}}dy_1dy_2 \leq  \frac{\kappa}{|\tau_0|^{100}}.
\end{equation}
Combining the above equations we infer that
\begin{align}
\left|\int_{-\infty}^{\infty}\int_{-\infty}^{\infty} \hat{u}(y_1,y_2,\tau_0)\phi_i \, e^{-\frac{y_1^2+y_2^2}{4}} \, dy_1dy_2\right|\leq  \frac{3\kappa}{|\tau_0|^{100}}.
\end{align}
In particular, this shows that
\begin{equation}
\left\| \cP_{+} \hat{u}(\tau_0)  \right\|_{\cH} \leq \frac{1}{|\tau_0|^{100}},
\end{equation}
and thus finishes the proof of the lemma.
\end{proof}

As another preparation, we need some suitable graphical radius to get the argument started:

\begin{lemma}[initial graphical radius]\label{poly_graph}
There exists some universal number $q>0$ with the following significance. For every $\kappa>0$ sufficiently small, there exists a constant $\tau_{\ast}>-\infty$,  such that if $M$ is $\kappa$-quadratic at time $\tau_0\leq \tau_{\ast}$, then  $\rho(\tau)=|\tau|^q$ is an admissible graphical radius function for  $\tau \leq \tau_0$.
\end{lemma}

\begin{proof}We will use the Lojasiewicz-Simon inequality from Colding-Minicozzi \cite{CM_uniqueness} in combination with \eqref{mu_quad_rad} and the discussion after  Proposition \ref{C0 estimate beta3}.
Recall that the Gaussian area of a hypersurfaces in $\mathbb{R}^4$, given by
\begin{equation}
F(M)=(4\pi)^{-3/2}\int_{M}e^{-\frac{|x|^2}{4}},
\end{equation}
is decreasing along the renormalized mean curvature flow. Letting $\Gamma$ be the bubble-sheet cylinder, we have
\begin{equation}
\lim_{\tau\rightarrow -\infty} F(\bar{M}_{\tau})=F(\Gamma).
\end{equation} 
Using \eqref{mu_quad_rad} and Taylor expansion, at time $\tau_0$ we can estimate
\begin{equation}
\left|\int_{ |{\bf{y}}| \leq |\tau_0|^{1/100}}\left( (\sqrt{2}+u)\sqrt{1+|\nabla u|^2} e^{-\tfrac{|{\bf{y}}|^2+(\sqrt{2}+u)^2}{4}}-  \sqrt{2}e^{-\tfrac{|{\bf{y}}|^2+2}{4}}\right)\right| \leq 40 |\tau_0|^{-1/50}.
\end{equation}
Together with Gaussian tale estimates and monotonicity this implies
\begin{equation}\label{F_gap}
0 \leq F(\Gamma)-F(\bar{M}_{\tau}) \leq 50 |\tau_0|^{-1/50}
\end{equation}
for all $\tau\leq\tau_0$. Hence, by quantitative differentiation \cite{CHN_stratification}, for any $\eps>0$ and $R<\infty$, there exists ${\tau}_{\ast}$ such that  if $ \tau_0 \leq \tau_{\ast}$ then for every $\tau \leq \tau_0$ one has that $\bar{M}_{\tau}$ is a $C^{2,\alpha}$ graph of with norm at most $\eps$ over the cylinder in  $B(0,R)$.  Thus, by \cite[Theorem 6.1]{CM_uniqueness}, there exist $K<\infty$ and $\eta\in (1/3,1)$ such that for every $\tau \leq \tau_0-1$ we have 
\begin{equation}
\left(F(\Gamma)-F(\bar{M}_{\tau})\right)^{1+\eta} \leq K\left(F(\bar{M}_{\tau-1})-F(\bar{M}_{\tau+1})\right).
\end{equation}
Using the discrete Lojasiewicz lemma  \cite[Lemma 6.9]{CM_uniqueness} this yields
\begin{equation}\label{CM_decay}
\left(F(\Gamma)-F(\bar{M}_{\tau})\right) \leq C(K,\eta)|\tau|^{-1/\eta},
\end{equation}
for every $\tau\leq 2\tau_0$, and
\begin{equation}\label{tail_decay_sum}
\sum_{j=J}^{\infty} \left(F(\bar{M}_{-j-1})-F(\bar{M}_{-j})\right)^{1/2} \leq C(K,\eta)J^{-p}.
\end{equation}
for $J\geq 2|\tau_0|$, where $p :=\frac{1}{4\eta}-\frac{1}{4}$. Since the renormalized mean curvature flow is the negative gradient flow of the $F$-functional this implies
\begin{equation}
\int_{-\infty}^{\tau}\int_{\bar{M}_{\tau'}} \left|\mathbf{H}(q)+\frac{q^{\perp}}{2}\right|e^{-\frac{|q|^2}{4}} d\mu_{\tau'}(q) d\tau' \leq C|\tau |^{-p}
\end{equation}
for $\tau\leq 2\tau_0$. Hence, applying \cite[Lemma A.48]{CM_uniqueness}, we obtain
\begin{equation}
\int_{ \{|(y_1,y_2)|\leq \rho_0(\tau)/2\}}|u(y_1,y_2,\theta,\tau)|e^{-\frac{y_1^2+y_2^2}{4}}  \leq C|\tau|^{-p}
\end{equation}
for $\tau\leq 2\tau_0$, where $\rho_0(\tau)$ is our initial choice of graphical radius.
Thus, by Proposition \ref{C0 estimate beta3} (barrier estimate) the quantity $\beta$ from  \eqref{beta def} satisfies $\beta(\tau)\leq C|\tau|^{-p/2}$, so by \eqref{eq_better_rho} the function $|\tau|^{-p/20}$ is an admissible graphical radius for $\tau\leq 2\tau_0$. Hence, setting $q=\min\{\frac{p}{20},\frac{1}{200}\}$, together with \eqref{mu_quad_rad} we conclude that $\rho(\tau):=|\tau|^{q}$ is an admissible graphical radius for $\tau\leq \tau_0$. This finishes the proof of the lemma.
\end{proof}

\bigskip

After these preparations, we can now prove the main result of this subsection:

\begin{proof}[Proof of Theorem \ref{thm_precise_from_one_time}]
Using the graphical radius $\rho(\tau)=|\tau|^q$ from Lemma \ref{poly_graph} (initial graphical radius), we define the  truncated function $\hat{u}$ as in \eqref{eq_trunc_bubble}. Remembering \eqref{decomp_Hilbert} we consider
\begin{align}
&U_+(\tau) := \|\cP_+ \hat{u}(\tau)\|_{\mathcal{H}}^2\, , \nonumber\\ 
&U_0(\tau) := \|\cP_0 \hat{u}(\tau)\|_{\mathcal{H}}^2\, ,\label{def_U_PNM} \\ 
&U_-(\tau) := \|\cP_- \hat{u}(\tau)\|_{\mathcal{H}}^2\, . \nonumber
\end{align}
Recall from \cite[Section 4]{CHH_blowdown} that for $\tau\leq \tau_0$ we have the differential inequalities
\begin{align}
&\frac{d}{d\tau} U_+ \geq U_+ - C_0\rho^{-1} \, (U_+ + U_0 + U_-), \nonumber\\ 
&\Big | \frac{d}{d\tau} U_0 \Big | \leq C_0\rho^{-1} \, (U_+ + U_0 + U_-), \label{U_PNM_system}\\ 
&\frac{d}{d\tau} U_- \leq -U_- + C_0\rho^{-1} \, (U_+ + U_0 + U_-), \nonumber
\end{align}
where $C_0<\infty$ is a constant. We will first show that $U_0$ dominates in the following quantitative sense:
\begin{claim}[dominant mode]\label{claim_dom_mode}
For every $\tau \leq \tau_0$ we have the inequality
\begin{equation}\label{zero_dom_mz}
U_{+}(\tau)+U_-(\tau) \leq \frac{4C_0}{\rho(\tau_0)}U_{0}(\tau).
\end{equation}
\end{claim}
\begin{proof}
We argue as in the proof of the Merle-Zaag ODE lemma \cite{MZ}. Set $\eps=C_0\rho(\tau_0)^{-1}$. Possibly after decreasing $\tau_\ast$, we can assume that $\eps<1/100$. Now, if
at some time $\tau \leq \tau_0$ we had the inequality $2\eps(U_{+}+U_{0}) \leq U_{-}$, then by \eqref{U_PNM_system} at this time $\tau$ we would get
\begin{align}\label{neg_too_large_der}
\frac{d}{d\tau}\left(U_{-} -2\eps(U_{+}+U_{0}) \right) &\leq -U_- +\eps(1+4\eps)(U_{+}+U_0+U_-)-2\eps U_{+}\nonumber\\
&\leq  -U_-+\eps(1+4\eps)(1+\frac{1}{2\eps})U_- < 0.
\end{align}
Hence, if the inequality $2\eps(U_{+}+U_{0}) \leq U_{-}$ held at some time $\tau_1\leq \tau_0$, then it would hold on $(-\infty, \tau_1]$, contradicting Theorem \ref{thm_L2.convergence} (inwards quadratic bending). Thus, we must have      
\begin{equation}\label{zero_dom_mz2}
U_{-} <  2\eps(U_0+U_{+})
\end{equation}  
for every $\tau \leq \tau_0$. To finish the proof, we will show that
\begin{equation}\label{neg_small}
U_{+} <  8\eps U_0.
\end{equation}
for all $\tau\leq \tau_0$. Note that by Lemma \ref{prop_cut_compatibility} (upgrade to bubble-sheet) this inequality indeed holds at time $\tau=\tau_0$.
Now, if the inequality \eqref{neg_small} failed for some at some time less than $\tau_0$, then at the largest time $\tau<\tau_0$ where it failed we would have $U_{+} = 8\eps U_0$. Together with \eqref{U_PNM_system} and \eqref{zero_dom_mz2} this would imply
\begin{align*}
\frac{d}{d\tau}\left(8\eps U_0-U_{+}\right) &\leq \eps(8\eps+1)(U_{+}+U_0+U_-)-U_{+}\\
&\leq \eps(8\eps+1)(8\eps+1+2\eps(1+8\eps))U_0 - 8\eps U_0 \leq -\eps U_0<0.
\end{align*}
This contradicts the definition of $\tau$, and thus establishes \eqref{neg_small}. This concludes the proof of the claim.
\end{proof}

Now that we know that $U_0$ dominates, it is important to determine which eigenfunction in
\begin{equation}
\mathcal{H}_0=\textrm{span}\{ y_1^2-2,y_2^2-2,y_1y_2 \}
\end{equation}
is the dominated one. The following claim shows that $y_2^2-2$ dominates in a quantitative sense:
\begin{claim}[dominant eigenfunction]\label{alpha_3_small_claim}
For $\tau\leq\tau_0$ with $\|\hat{u}(\tau)\|_{\cH} \geq e^{-\rho(\tau)}$ we have the estimate
\begin{equation}
\left|\langle \hat{u}(\tau), y_1^2-2 \rangle_{\cH}\right|+\left|\langle  \hat{u}(\tau), y_1y_2 \rangle_{\cH}\right| \leq  \frac{C}{\rho(\tau)}\|\hat{u}(\tau)\|_{\cH}\, .
\end{equation}
\end{claim}

\begin{proof}
Let $\psi_1:=y_1^2-2, \psi_2:=y_1y_2$ and set $\alpha_i:=\langle \hat{u},\psi_i\rangle_{\cH}$. Then, for $i=1,2$ we have
\begin{align}\label{alpha_3_est}
|\alpha_i| \leq \left| \int_{\max\{|y_1|,|y_2|\} \leq \tfrac{1}{2}\rho(\tau) } u \psi_i \, e^{-\tfrac{y_1^2+y_2^2}{4}}dy_1dy_2\right|+\int_{\max\{|y_1|,|y_2|\} \geq \tfrac{1}{2}\rho(\tau) } |\hat{u} \psi_i| e^{-\tfrac{y_1^2+y_2^2}{4}}dy_1dy_2.
\end{align}
Using the Cauchy-Schwarz inequality and the inverse Poincare inequality from \cite[Proposition 4.4]{CHH_blowdown} we can estimate the second integral by
\begin{equation}\label{second_int_alpha_3}
\int_{\max\{|y_1|,|y_2|\} \geq \tfrac{1}{2}\rho(\tau) } |\hat{u} \psi_i|\,  e^{-\tfrac{y_1^2+y_2^2}{4}}dy_1dy_2 \leq \frac{C}{\rho(\tau)}\|\hat{u}\|_{\cH}.
\end{equation}
To bound the first integral, note that as in the proof of Lemma \ref{prop_cut_compatibility} (upgrade to bubble-sheet) we have
\begin{equation}\label{pointwise_diff_again}
\sup_{\max\{|y_1|,|y_2|\} \leq \tfrac{1}{2}\rho(\tau) } \left| \sqrt{2}+u(y_1,y_2,\tau)-v(y_2,\tau)\right| \leq e^{\tau/3}.
\end{equation}
Hence,
\begin{multline}
\left| \int_{\max\{|y_1|,|y_2|\} \leq \tfrac{1}{2}\rho(\tau) } u(y_1,y_2,\tau) \psi_i \, e^{-\tfrac{y_1^2+y_2^2}{4}}dy_1dy_2\right.\\
\left. - \int_{\max\{|y_1|,|y_2|\} \leq \tfrac{1}{2}\rho(\tau) } \left(v(y_2,\tau)-\sqrt{2}\right) \psi_i \, e^{-\tfrac{y_1^2+y_2^2}{4}}dy_1dy_2\right| \leq e^{\tau/4}.
\end{multline}
Now, since $\psi_2$ is an odd function of $y_1$ we clearly have
\begin{equation}
\int_{\max\{|y_1|,|y_2|\} \leq \tfrac{1}{2}\rho(\tau) } \left(v(y_2,\tau)-\sqrt{2}\right) \psi_2 \, e^{-\tfrac{y_1^2+y_2^2}{4}}dy_1dy_2=0.
\end{equation}
To estimate the integral involving $\psi_1$ observe that using the identity
\begin{equation}
\int_{-\infty}^{\infty} \psi_1 e^{-\frac{y_1^2}{4}}dy_1=0
\end{equation}
for every $y_2$ we have
\begin{equation}
\int_{|y_1| \leq \tfrac{1}{2}\rho(\tau) }   \left(v(y_2,\tau)-\sqrt{2}\right) \psi_1 \, e^{-\tfrac{y_1^2+y_2^2}{4}}dy_1=-\int_{|y_1| \geq \tfrac{1}{2}\rho(\tau) }   \left(v(y_2,\tau)-\sqrt{2}\right) \psi_1 \, e^{-\tfrac{y_1^2+y_2^2}{4}}dy_1.
\end{equation}
This yields
\begin{equation}
\left| \int_{\max\{|y_1|,|y_2|\} \leq \tfrac{1}{2}\rho(\tau) } \left(v(y_2,\tau)-\sqrt{2}\right) \psi_1 \, e^{-\tfrac{y_1^2+y_2^2}{4}}dy_1dy_2\right| \leq e^{-2\rho(\tau)}.
\end{equation}
Combining the above equations establishes the claim.
\end{proof}

Continuing the proof of the theorem, we consider the evolution of the coefficient
\begin{equation}
\alpha_0(\tau) := \langle \hat{u}(\tau),\psi_0\rangle_{\cH},
\end{equation}
where we now work with the normalized eigenfunction
\begin{equation}
\psi_0=2^{-3/2}(\tfrac{e}{2\pi})^{1/4}(y_2^2-2).
\end{equation}
Note that by the above two claims, $\alpha_0$ is dominant in a quantitative sense. Specifically, if we write
\begin{equation}\label{decomp_uw}
\hat{u}(\tau)=\alpha_0(\tau)\psi_0+w(\tau),
\end{equation}
then for all $\tau\leq \tau_0$ with $|\alpha_0(\tau)|\geq e^{-\rho(\tau)}$ we have the estimate
\begin{equation}\label{proj_small}
\|w(\tau)\|_{\cH} \leq \frac{C}{\rho(\tau_0)}|\alpha_0(\tau)|.
\end{equation}
Now, using Proposition \ref{lemma_taylor} (evolution equation)  and equation \eqref{decomp_uw} we see that
\begin{align}\label{second_order_eq}
\frac{d}{d\tau}\alpha_0&=-\tfrac{1}{2\sqrt{2}}\langle \hat{u}^2,\psi_0\rangle_{\cH}+ \langle E,\psi_0 \rangle_{\cH} \nonumber \\
& = -(\tfrac{e}{2\pi})^{\frac{1}{4}}\alpha_0^2-  \tfrac{1}{\sqrt{2}} \alpha_0 \langle w,\psi_0^2 \rangle_{\cH}   -\tfrac{1}{2\sqrt{2}}  \langle w^2, \psi_0 \rangle_{\cH} + \langle E,\psi_0 \rangle_{\cH},  
\end{align}
where $E$ satisfies the pointwise estimate \eqref{est_eq}.

\begin{claim}[error estimate]\label{second_evlove_claim}
For all $\tau\leq \tau_0$ with $|\alpha_0(\tau)|\geq e^{-\rho(\tau)^{1/2}}$ we have the estimate
\begin{equation}\label{second_evolve_eq}
\left|\alpha_0 \langle w,\psi_0^2 \rangle_{\cH} \right| + \left| \langle w^2, \psi_0 \rangle_{\cH} \right| +\left| \langle E,\psi_0 \rangle_{\cH}\right| \leq \frac{C}{\rho(\tau_0)}\alpha_0^2(\tau).
\end{equation}
\end{claim}

\begin{proof}
Using equation \eqref{proj_small} the first term is easily controlled as  
\begin{equation}
|\alpha_0 \langle w,\psi_0^2 \rangle_{\cH}| \leq C|\alpha_0|\|w\|_{\cH} \leq  \frac{C}{\rho(\tau_0)}\alpha_0^2(\tau).
\end{equation} 
To bound the last term, first observe that $E$ from \eqref{est_eq} is supported in the ball $\{|{\bf{y}}|\leq 2\rho(\tau)\}$, so in particular by the definition of admissible graphical radius we have the estimate
\begin{equation}
|u|+|\nabla u| + |\nabla^2 u| \leq \frac{1}{\rho(\tau)^2}.
\end{equation}
Now, for $|{\bf{y}}|\leq \rho(\tau)^{1/2}$ we can estimate 
\begin{equation}
|E||\psi_0|\leq \left(C |u|^3 + C |\nabla u|^2 |\nabla^2 u|\right) \rho(\tau)\leq \frac{C}{\rho(\tau)} \left(|u|^2 + |\nabla u|^2\right).
\end{equation}
Together with the inverse Poincare inequality from \cite[Proposition 4.4]{CHH_blowdown} this yields
\begin{equation}
\int_{|{\bf{y}}|\leq \rho(\tau)^{\frac{1}{2}}}|E\psi_0|\, e^{-\frac{|{\bf{y}}|^2}{4}}  \leq \frac{C}{\rho(\tau)}\|\hat{u}\|_{\cH}^2  \leq \frac{C}{\rho(\tau)}\alpha_0^2(\tau),
\end{equation}
where in the last step we used the above two claims.
In the remaining domain we have the coarse estimate $|E\psi_0| \leq C \rho(\tau)^2$ so we can bound
\begin{align}
\int_{\rho(\tau)^{1/2} \leq |{\bf{y}}| \leq 2\rho(\tau)}|E\psi_0|\, e^{-\frac{|{\bf{y}}|^2}{4}}  \leq  e^{-\rho(\tau)^{2/3}} \, .
\end{align}
Hence, for all $\tau\leq \tau_0$ with $|\alpha_0(\tau)|\geq e^{-\rho(\tau)^{1/2}}$ we get
\begin{equation}
\left| \langle E,\psi_0 \rangle_{\cH}\right| \leq \frac{C}{\rho(\tau_0)}\alpha_0^2(\tau).
\end{equation}
Finally, the second term is controlled similarly as in  \cite[Proof of Lemma 5.14]{ADS1}, but since we need to check that everything works from time $\tau_0$ we include the details. By Ecker's weighted Sobolev inequality \cite{Ecker_logsob} we have
\begin{equation}\label{eck_sob}
|\langle \psi_0, w^2 \rangle_{\cH}| \leq C\int (1+|{\bf{y}}|^2)w^2e^{-|{\bf{y}}|^2/4} \leq C(\|w\|^2_{\cH}+ \|\nabla w\|^2_{\cH}).
\end{equation}
Since the Gaussian $L^2$-norm is already controlled by \eqref{proj_small}, it thus suffices to control $\|\nabla w\|^2_{\cH}$. To this end,  note that projecting the evolution equation for $\hat{u}$ from Proposition \ref{lemma_taylor} (evolution equation) to the orthonormal complement of $\mathrm{span}\{\psi_0\}$ gives
\begin{equation}\label{w_evolve_Eq}
\partial_{\tau}w=\mathcal{L}w+g,
\end{equation}
where $g$ at all $\tau\leq \tau_0$ with $|\alpha_0(\tau)|\geq e^{-\rho(\tau)^{1/2}}$ satisfies the estimate
\begin{equation}\label{g_est}
\|g\|_{\mathcal{H}} \leq \frac{1}{2\sqrt{2}}\|\hat{u}^2\|_{\cH}+\|E\|_{\cH}  \leq  \frac{C}{\rho(\tau)} \|\hat{u}\|_{\cH}\leq \frac{C}{\rho(\tau_0)}|\alpha_0(\tau)|\, .
\end{equation}
Now, given $\hat{\tau}\leq \tau_0$, using \eqref{w_evolve_Eq} and integration by parts we compute
\begin{align}
\frac{d}{d\tau}\int {e^{\hat{\tau}-\tau}} w^2e^{-|{\bf{y}}|^2/4} &= \int  {e^{\hat{\tau}-\tau}}(2wg-2|\nabla w|^2 -w^2)\, e^{-|{\bf{y}}|^2/4}\nonumber\\
& \leq \int  {e^{\hat{\tau}-\tau}}(g^2-2|\nabla w|^2)\, e^{-|{\bf{y}}|^2/4},
\end{align}
and
\begin{align}
 \frac{d}{d\tau}\int (\tau-\hat{\tau})|\nabla w|^2 e^{-|{\bf{y}}|^2/4}
 &=\int \left(|\nabla w|^2  -2(\tau-\hat{\tau}) (\mathcal{L}  w)(\mathcal{L}w+g)\right)\, e^{-|{\bf{y}}|^2/4}\nonumber\\
 &\leq\int \left(|\nabla w|^2  +\tfrac{1}{2}(\tau-\hat{\tau}) g^2)\right)\, e^{-|{\bf{y}}|^2/4}\, .
\end{align}
For $\tau\in [\hat{\tau}-1,\hat{\tau}]$ this yields
\begin{align}
 \frac{d}{d\tau}\int \left((\tau-\hat{\tau})|\nabla w|^2+\tfrac{e^{\hat{\tau}-\tau}}{2} w^2\right)\, e^{-|{\bf{y}}|^2/4}\leq \int g^2\, e^{-|{\bf{y}}|^2/4}\, .
 \end{align}
Hence, together with \eqref{proj_small} and \eqref{g_est} we infer that
\begin{equation}
\|\nabla w(\tau)\|_{\mathcal{H}}^2 \leq \frac{C}{\rho(\tau_0)^2} \sup_{\tau' \in [\tau,\tau+1]}\alpha_0^2(\tau').
\end{equation}
Finally, using the second Merle-Zaag ODE from \eqref{U_PNM_system} and remembering \eqref{proj_small} we see that
\begin{equation}
\sup_{\tau' \in [\tau,\tau+1]} \alpha_0^2(\tau') \leq 2 \alpha_0^2(\tau).
\end{equation}
In light of \eqref{eck_sob}, the last estimate is established and the claim follows.
\end{proof}

Now, setting $\gamma:=(\tfrac{e}{2\pi})^{\frac{1}{4}}$ and $\eps:=C\rho(\tau_0)^{-1}$ by the evolution equation \eqref{second_order_eq} and Claim \ref{second_evlove_claim} (error estimate) we have
\begin{equation}
\left| \frac{d}{d\tau}\alpha_0 + \gamma\alpha_0^2 \right|\leq \eps \alpha_0^2(\tau)
\end{equation}
for all $\tau\leq \tau_0$ with $|\alpha_0(\tau)|\geq e^{-\rho(\tau)^{1/2}}$. Integrating this differential inequality backwards in time gives
\begin{equation}\label{inte_boud_alpha_0}
(\gamma-\eps)(\tau-\tau_0) \leq \frac{1}{\alpha_0(\tau)} - \frac{1}{\alpha_0(\tau_0)} \leq (\gamma+\eps)(\tau-\tau_0),
\end{equation}
as long as $|\alpha_0| \geq e^{-\rho^{1/2}}$ on $[\tau,\tau_0]$. Regarding the initial condition, observe that by \eqref{precise_up_to} we have
\begin{equation}
\left| \frac{1}{\alpha_0(\tau_0)} -\gamma \tau_0\right| \leq \frac{\gamma^2\kappa}{2} |\tau_0|, 
\end{equation}
provided $\kappa$ is sufficiently small. Hence, if $\tau_0\leq \tau_\ast(\kappa)$ is so that $\eps\leq\frac{\gamma^2\kappa}{2}$, then we obtain
\begin{equation}\label{alpha_0_behave}
\left|\frac{1}{\alpha_0(\tau)}-\gamma \tau \right| \leq \frac{\gamma^2\kappa}{2}|\tau|,
\end{equation}
as long as $|\alpha_0| \geq e^{-\rho^{1/2}}$ on $[\tau,\tau_0]$. Finally, since $\frac{1}{|\tau|} \gg e^{-|\tau|^{\frac{q}{2}}}$ if follows from continuity that \eqref{alpha_0_behave} holds unconditionally. In other words, we have shown that for all $\tau\leq \tau_0$ we have
\begin{equation}\label{alpha_0_behaves}
\left| \alpha_0(\tau) + \frac{1}{\gamma\tau}\right| \leq \frac{3\kappa}{4|\tau|}.
\end{equation}
Together with the estimate  \eqref{proj_small} this shows that $\|\hat{u}(\tau)\|_{\cH} \sim |\tau|^{-1}$ for every $\tau \leq \tau_0$. Hence, similarly as in \eqref{eq_better_rho} we can now upgrade to the new graphical radius $\rho(\tau)=|\tau|^{1/10}$ for $\tau\leq \tau_0$. Furthermore, combining \eqref{proj_small} and \eqref{alpha_0_behaves} also shows that $\hat{u}$, now defined with respect to the new graphical radius, satisfies
\begin{equation}\label{mu_quad_back_rest}
\left\|\hat{u}(\tau)+\frac{y_2^2-2}{2\sqrt{2}|\tau|}\right\|_{\mathcal{H}}\leq \frac{\kappa}{|\tau|}\;\;\;\;\;\textrm{for}\;\tau\leq \tau_0.
\end{equation}
Thus, we conclude that $M$ is strongly $\kappa$-quadratic from time $\tau_0$. This finishes the proof of the theorem.
\end{proof}

As a corollary of the proof, we also obtain the following projection estimate:

\begin{corollary}[projection estimate]\label{projection_estimate}
If $M$ is $\kappa$-quadratic at time $\tau_{0} \leq \tau_{\ast}$, then
\begin{equation}
\|\fp_{-}(v_{\cC}(\tau_0))\|_{\fH} \leq \frac{\kappa}{100|\tau_0|}.
\end{equation}
\end{corollary}

\begin{proof}Setting $w(y_1,y_2,\tau):=v_{\cC}(y_2,\tau)$  we compute
\begin{align}
(2e)^{-\frac{1}{4}}\|\fp_{-}(v_{\cC}(\tau_0))\|_{\fH} =\|\cP_{-}(w(\tau_0))\|_{\cH}  &\leq  \|\cP_-(w(\tau_0)-\hat{u}(\tau_0))\|_{\cH}+\|{\cP}_{-}(\hat{u}(\tau_0))\|_{\cH}\nonumber \\
 &\leq  \|w(\tau_0)-\sqrt{2}-\hat{u}(\tau_0)\|_{\cH}+U_-(\tau)^{1/2}.
\end{align}
Using again a combination of the pointwise estimate  \eqref{pointwise_diff} and Gaussian tail estimates we get
\begin{equation}
\|w(\tau_0)-\sqrt{2}-\hat{u}(\tau_0)\|_{\cH}\leq \frac{\kappa}{|\tau_0|^{100}}. 
\end{equation}
Moreover, by Claim \ref{claim_dom_mode} (dominant mode) and the inequality $U_0(\tau)^{1/2}\leq C|\tau|^{-1}$ we have
\begin{equation}
U_{-}(\tau)^{1/2}\leq 2\left(\frac{C_0}{\rho(\tau_0)}\right)^{1/2} \frac{C}{|\tau|}
\end{equation}
Taking $\tau_\ast$ sufficiently negative, this implies the assertion.
\end{proof}

As a consequence, we now obtain uniform sharp asymptotics depending only on $\kappa$-quadraticity:

\begin{theorem}[uniform sharp asymptotics]\label{cor_unique_asympt}
For every $\eps>0$ there exists $\kappa>0$ and $\tau_{\ast}>-\infty$, such that if $M$ is $\kappa$-quadratic at time $\tau_{0}$ for some $\tau_0 \leq \tau_{\ast}$, then for every $\tau \leq \tau_{0}$ the following holds:
\begin{enumerate}
\item Parabolic region: The renormalized profile function satisfies
\begin{equation}
\left| v(y, \tau)-\sqrt{2}\left(1-\frac{y^2-2}{4 |\tau|}\right) \right| \leq\frac{\eps}{|\tau|} \qquad (|y |\leq \eps^{-1}).
\end{equation}
\item Intermediate region: The function $\bar{v}(z,\tau):=v(|\tau|^{1/2}z,\tau)$ satisfies
\begin{equation}
|\bar{v}(z,\tau)-\sqrt{2-z^2}|\leq \eps,
\end{equation}
on $[-\sqrt{2}+\eps,\sqrt{2}-\eps]$.
\item Tip regions: We have the estimate
\begin{equation}
\| Z(\cdot,\tau)-Z_0(\cdot)\|_{C^{100}(B(0,\eps^{-1}))}\leq \eps,
\end{equation}
where $Z_0(\rho)$ is the profile function of the $2d$-bowl with speed $1/\sqrt{2}$.
\end{enumerate}
Moreover, we have the estimate
\begin{equation}
\|\fp_{-}(v_{\cC}(\tau_0))\|_{\fH} \leq \frac{\kappa}{100|\tau_0|}.
\end{equation}
Furthermore, for every $\tau\leq \tau_0$ the renormalized hypersurface $\bar{M}_{\tau}=e^{-\tau/2}M_{-e^{-\tau}}$ can be expressed locally as a graph of a function $u(y_1,y_2,\tau)$ over the cylinder $\mathbb{R}^2\times S^1(\sqrt{2})$ with the estimate
\begin{equation}
\| u \|_{{C^4}(B(0,2|\tau|^{1/10})}\leq |\tau|^{-1/5}.
\end{equation}
\end{theorem}

\begin{proof}
This follows combining Theorem \ref{thm_unique_asympt} (uniform sharp asymptotics assuming strong $\kappa$-quadraticity), Theorem \ref{thm_precise_from_one_time} ($\kappa$-quadraticity implies strong $\kappa$-quadraticity) and Corollary \ref{projection_estimate} (projection estimate).
\end{proof}

\bigskip

\section{From spectral uniqueness to classification}\label{sec_HIMW}

In this section, we explain how to derive the main classification theorem from spectral uniqueness.

\subsection{The Hoffman-Ilmanen-Martin-White class}\label{sec_existence}

In this subsection, we introduce the HIMW class by slightly generalizing the construction from \cite[Cor. 8.2]{HIMW}, and establish some of its basic properties. We also fix notations that will be used throughout the remaining subsections.\\

For every $a\in [0,\frac{1}{3}]$ and every $R<\infty$, consider the ellipsoidal domain given by
\begin{equation}
\Omega_{a,R}:=\left\{(x_2,x_3,x_4)\, \Big|\, a^2x_2^2+\left(\frac{1-a}{2}\right)^2x_3^2+\left(\frac{1-a}{2}\right)^2x_4^2 < R^2\right\}.
\end{equation}
Let $u_{a,R}$ be the solution to the upward moving translator equation\footnote{In contrast to \cite{HIMW}, we use the convention that translators move upwards. In particular, we have $u_{a,R}\leq 0$.}
\begin{align}
\mathrm{div}\left(\frac{\nabla u}{\sqrt{1+|\nabla u|^2}}\right)-\frac{1}{\sqrt{1+|\nabla u|^2}}=0\;\;\;&\textrm{on } \Omega_{a,R},\label{graph_transl}\\
 u=0\;\;\;&\textrm{on } \partial \Omega_{a,R}. \nonumber
\end{align}
As shown in \cite[Section 9]{HIMW}, it follows from the moving plane method that $u_{a,R}(x_2,x_3,x_4)$ attains its minimum $\xi=\xi(a,R)\in (-\infty,0)$ at $x_2=x_3=x_4=0$, and that $u_{a,R}$ is $\mathrm{SO}(2)$-symmetric in the $x_3x_4$-plane, and reflection symmetric in the $x_2$-coordinate. Using interior and exterior bowl barriers one easily sees that $\xi(a,R)\rightarrow -\infty$ as $R\rightarrow \infty$ and $\xi(a,R)\rightarrow 0$ as $R\rightarrow 0$. Observing also that for any fixed $a$, the function  $R\mapsto \xi(a,R)$ is strictly decreasing, it follows that for every $(\xi, a)$ there is a unique $R=R(\xi,a)$, depending continuously on $(\xi,a)$, such that $u_{a,R}(0)=\xi$. By abuse of notation, write $u_{a,\xi}=u_{a,R(\xi,a)}$.\\

We also recall that the gradient estimate from \cite[Theorem 7.4]{EvansSpruck}, together with standard higher derivative estimates, gives uniform estimates (depending only on a bound for $R$) for all derivatives of solutions of the problem \eqref{graph_transl}. In particular, this yields smooth compactness for sequences of translators-with-boundary with bounded $R$, and also yields locally smooth compactness for sequences along which $R\to \infty$.\\

We now shift the tip to the origin, namely we consider the translator (with boundary) defined by
\begin{equation}
M^{a,\xi}:= \mathrm{graph}(u_{a,\xi}-\xi). 
\end{equation}

We can now introduce the HIMW class as the collection of all translators that are obtained as limits of the above translators 
  $M^{a_i,\xi_i}$, for any sequences $a_i\in [0,\frac{1}{3}]$ and $ \xi_i\to -\infty$:

\begin{definition}[HIMW class]\label{gen_HIMW_def}
The HIMW class is
\begin{equation}
\mathcal{A}:=\left\{ \lim_{i\to \infty} M^{a_i,\xi_i} \, | \,  a_i\in [0,1/3] \textrm{ and } \xi_i\to -\infty \right\}. 
\end{equation}
\end{definition}

Note that, inheriting the properties from $M^{a_i,\xi_i}$, all elements in $\mathcal{A}$ are $\mathrm{SO}(2)$-symmetric in the $x_3x_4$-plane, and reflection symmetric in the $x_2$-coordinate. Moreover, the proof of \cite[Theorem 8.1, Corollary 8.2]{HIMW} carries through to our setting, showing that every $M\in \mathcal{A}$ is an entire graph.
Furthermore, the circular symmetry together with \cite[Theorem 9.2]{HIMW} implies that
\begin{equation}
\textrm{the principal curvatures at the tip $0\in M$ are equal to $(k,\tfrac{1-k}{2},\tfrac{1-k}{2})$ for some $k\in [0,\tfrac{1}{3}]$.} 
\end{equation}

Let us next explain the relationship with the construction from \cite[Cor. 8.2]{HIMW}. To this end, note first that when $a=0$ then $M^{a,\xi}$ splits off the $x_2$-direction by \cite[Theorem 3.2]{HIMW} hence is a piece of $\mathbb{R}\times$2d-bowl, and when $a=\frac{1}{3}$ then $M^{a,\xi}$ is $\mathrm{O}(3)$-symmetric hence a piece of the 3d round bowl.   For each fixed $\xi$, we consider the tip curvature map
\begin{equation}
F^{\xi}:[0,\tfrac{1}{3}]\rightarrow  [0,\tfrac{1}{3}],\quad a\mapsto  k.
\end{equation}
Observe that $F^\xi$ is continuous as a consequence of the uniqueness and the uniform derivative estimates that we recalled above. It thus follows from the intermediate value theorem that $F^\xi$ is surjective.\\

In the construction from \cite[Cor. 8.2]{HIMW} one fixes the tip curvature $k\in [0,1/3]$ and then for $\xi_i\to -\infty$ chooses $a_i$ with $F^{\xi_i}(a_i)=k$ and passes to a limit of $M^{a_i,\xi_i}$. Here, we slightly generalized the construction by also allowing that $k_i\to k$ depends on $i$, which a priori leads a larger class of translators (a posteriori it will be the same) and is important for the argument in Section \ref{sec_spectral_eccent}.

\begin{theorem}[noncollapsing and convexity]\label{thm_non_collapsed}
Every $M\in \mathcal{A}$ is noncollapsed and convex.
\end{theorem}

\begin{proof}
Consider the associated mean curvature flow $M_t=M+te_1$. Since $M$ is an entire graph, $M_t$ foliates the entire space and thus by mean-convexity has  polynomial volume growth (indeed this follows from a standard calibration argument as explained e.g. in \cite[Remark 2.6]{HaslhoferKleiner_meanconvex}). Therefore, the entropy $\mathrm{Ent}[M]$ is finite. Hence, we can let $N_t$ be a tangent flow to $M_t$ at $-\infty$.   
We claim that $N_t$ cannot be a hyperplane (of any multiplicity).
To this end, note that \cite[Theorem 9.3]{HIMW} implies that for every  $M\in \mathcal{A}$ and height $h>0$, the level sets $\Sigma^h:=M\cap \{x_1=h\}$ satisfy
\begin{equation}\label{large_ax_HIMW}
\max_{x\in \Sigma^h}x_2 \geq \max_{x\in \Sigma^h}x_3.
\end{equation}
Now, if $N_t=Q$ for some hyperplane $Q$, then clearly $e_1\in Q$. Moreover, by counting dimensions  we see that $Q\cap \mathrm{span}\{e_3,e_4\}\neq \{0\}$. Together with the $\mathrm{SO}(2)$-symmetry this implies $Q=\mathrm{span}\{e_1,e_3,e_4\}$, contradicting \eqref{large_ax_HIMW}. Hence, $N_t$ is not a hyperplane.

\begin{claim} $N_t$ is a smooth multiplicity-one self-shrinker.
\end{claim}

\begin{proof}
Since we have already excluded hyperplanes, in particular the ones of multiplicity-two, this follows from the methods of White \cite{White_size}. Indeed,  first observe that every tangent flow to $N_t$ has to be a static or quasi-static hyperplane (with multiplicity one or two), being a one-sided minimizing stationary cone. Note also that since $N_t$ is self-shrinking, quasi-static hyperplanes (of any multiplicities) are excluded by the clearing out lemma.  Now, letting $K_t$ be the domain enclosed by $N_t$, the above implies that a point $x\in N_t$ is regular with multiplicity-one if and only if $x\in \mathrm{Cl}(\mathrm{Int}K_t)$. This is a closed condition, so the regular set is closed. Also, the regular set is always open by the local regularity theorem. We will next show that $0\in\mathrm{Int}(K_t)$ for $t<0$. 
To this end, note that in addition to \eqref{large_ax_HIMW} the moving plane method also yields that $\max_{x\in \Sigma^h}x_3$ is attained at a point with $x_2=x_4=0$, and that  
 $\Sigma^h\cap \{x_4=0\}\cap \{x_2>0\}\cap \{x_3>0\}$ is graphical both over the $x_2$-axis and the $x_3$-axis. Hence, if $0$ was not an interior point of $K_t$, then we would have $x_3=0$ on $K_t$. Together with the circular symmetry this would imply that $K_t \subseteq \mathrm{span}\{e_1,e_2\}$, which is a contradiction. Thus, $\mathrm{Int}(K_t)\neq \emptyset$. Since $N_t$ is connected (being a limit of graphs), we conclude that all points are regular with multiplicity-one. This proves the claim.
\end{proof}

Thanks to the claim, we can apply Huisken's classification of smooth mean-convex shrinkers \cite{Huisken_shrinker}, which gives that $N_t$ must be a generalized cylinder. In particular, we infer that
\begin{equation}
\mathrm{Ent}[M] \leq \mathrm{Ent}[S^1]<2.
\end{equation}
It is well known to experts that this implies that $M$ is $\alpha$-noncollapsed. For convenience of the reader we provide a short proof using methods from the work of White  \cite{White_nature} (alternatively, one could apply the recent local noncollapsing estimate from Brendle-Naff \cite{BN_noncollapsing}).
Suppose towards a contradiction that there is a sequence of points $x_j\in M$ whose maximal interior tangent ball is of radius $r_j \leq   j^{-1}H(x_j)^{-1}$. Consider the sequence of flows $N^j_t$ that is obtained from $M_t$ by centering at $(x_j,0)$ and parabolically rescaling by $r_j^{-1}$. Passing to a subsequential  limit \cite{Ilmanen_book,White_Currents}, we obtain an ancient, cyclic, unit-regular, integral Brakke flow $\hat{M}_t$, with $0 \in \mathrm{spt}(\hat{M}_0)$.
As the tangent flow  to $\hat{M}_t$ at $(0,0)$ is contained in a half-space, it must be a hyperplane, with multiplicity-one, by the entropy bound. Hence, $(0,0)$ is a regular point \cite{White_regularity}, and $H(0,0)=0$.
By the strong maximum principle (see  Lemma \ref{max_prin} below), this implies that $\{\hat{M}_t\}_{t\leq 0}$ is a static hyperplane. For $j$ large, this contradicts the fact that $r_j$ was maximal. This establishes interior noncollapsing. A similar argument yields exterior noncollapsing. Finally, by \cite[Theorem 1.10]{HaslhoferKleiner_meanconvex} the noncollapsing implies convexity.
\end{proof}

In the above proof we have used the following lemma:

\begin{lemma}[{White's strong maximum principle, c.f. \cite[Theorem 6]{White_nature}}]\label{max_prin}
Suppose $\{M_t\}_{t\leq 0}$ is an ancient, cyclic, unit-regular, integral Brakke flow in $\mathbb{R}^{4}$ with entropy strictly less than two and such that $H\geq 0$ at regular points. If $(0,0)$ is a regular point and $H(0,0)=0$, then $\{M_t\}_{t\leq 0}$ is a flat hyperplane. 
\end{lemma}

\begin{proof}
By the smooth strong maximum principle, there is an $\eps>0$ such that  $M_t\cap B(0,\eps)$ is a smooth minimal hypersurface
$\Sigma$ for $t\in(-\eps^2,0]$. Furthermore, the assumptions of the lemma and \cite[Theorem 9]{White_stratification} imply that the singular set of $\{M_t\}_{t\leq 0}$ has parabolic Hausdorff dimension at most $2$. We claim that 
\begin{equation}\label{in_sup}
\Sigma \subseteq \mathrm{spt}(M_t),\;\;\textrm{for all}\; t\in (-\infty,0].
\end{equation}
Indeed, taking any $x_0\in \Sigma$ and $t_0<0$, the smallness of the singular set implies that $(x_0,t_0)$ can be connected to $(0,0)$ by a time-like space-time curve $\gamma$ that stays in the regular part of the flow. Hence, by the smooth strong maximum principle we obtain $H=0$ along $\gamma$. This proves \eqref{in_sup}.
 It follows that the tangent flow to $\{M_t\}_{t\leq 0}$ at $-\infty$ must be a flat hyperplane. Hence, $\{M_t\}_{t\leq 0}$ itself is a flat hyperplane.
\end{proof}

\bigskip

\subsection{Monotonicity of the tip curvature map}
Recall that by definition
\begin{equation}
F^{\xi}:[0,1/3]\rightarrow  [0,1/3], \quad a\mapsto k,
\end{equation}
maps $a$ to the smallest principal curvature $k$ of the tip $0\in M^{a,\xi}=\textrm{graph}(u_{a,\xi}-\xi)$. Recall also that since $F^{\xi}$ is continuous and fixes the endpoints, it must be surjective. The goal of this subsection is to prove:

\begin{theorem}[monotonicity]\label{rado_thm}
$F^{\xi}$ is strictly monotone.
\end{theorem}
\begin{proof}
If not, there exist $a_1\neq a_2$ such that $M^{a_1,\xi}$ and $M^{a_2,\xi}$ agree at the origin to more than second order.
Consider the difference function
\begin{equation}
w:=u_{a_1,\xi}-u_{a_2,\xi},
\end{equation}
defined over the intersection of ellipsoidal domains
\begin{equation}
\Omega:=\Omega_{a_1,R(a_1,\xi)}\cap \Omega_{a_2,R(a_2,\xi)}.
\end{equation}
We will analyze the nodal set
\begin{equation}
Z:=\{w=0\}.
\end{equation}
To this end, for any $p\in Z$ denoting by $d=d(p)$ be the leading order of $w$ around $p$,  we write
\begin{equation}
w=w_p+ E_p,
\end{equation}
where $w_p$ is the degree $d$ Taylor polynomial and the error satisfies
\begin{equation}\label{error_taylor}
E_p=O(|x-p|^{d+1}),\quad \nabla E_p=O(|x-p|^{d}),\quad \nabla^2 E_p=O(|x-p|^{d-1}).
\end{equation}
Here, $d$ is finite by Almgren's frequency function argument (see for instance \cite[Theorem 6.1]{CM_book}). Now, observe that
\begin{equation}\label{to_show_harm}
\left(\delta_{ij}-\frac{\nabla_i u_{a_1,\xi}(p)\nabla_j u_{a_1,\xi}(p)}{1+|\nabla u_{a_1,\xi}(p)|^2}\right) \nabla_i\nabla_j w_p =0.
\end{equation}
Indeed, using the translator equation \eqref{graph_transl} and the product rule for differences we see that
\begin{align}
\Delta w =&\frac{\nabla_i u_{a_1,\xi}\nabla_j u_{a_1,\xi}\nabla_i\nabla_j u_{a_1,\xi}}{1+|\nabla u_{a_1,\xi}|^2}-\frac{\nabla_i u_{a_2,\xi}\nabla_j u_{a_2,\xi}\nabla_i\nabla_j u_{a_2,\xi}}{1+|\nabla u_{a_2,\xi}|^2}\nonumber\\
=&\frac{\nabla_i u_{a_1,\xi} \nabla_j u_{a_1,\xi}\nabla_i\nabla_j w}{1+|\nabla u_{a_1,\xi}|^2}+\frac{\nabla_i u_{a_1,\xi} \nabla_j w\nabla_i\nabla_j u_{a_2,\xi}}{1+|\nabla u_{a_1,\xi}|^2}
+\frac{\nabla_i w \nabla_j u_{a_2,\xi}\nabla_i\nabla_j u_{a_2,\xi}}{1+|\nabla u_{a_1,\xi}|^2}\nonumber\\
&-\frac{\nabla_k(u_{a_1,\xi}+u_{a_2,\xi})\nabla_k w}{(1+|\nabla u_{a_1,\xi}|^2)(1+|\nabla u_{a_2,\xi}|^2)} \nabla_i u_{a_2,\xi} \nabla_j u_{a_2,\xi}\nabla_i\nabla_j u_{a_2,\xi}.
\end{align}
Since $\nabla w = O(|x-p|^{d-1})$ this yields
\begin{equation}
\left(\delta_{ij}-\frac{\nabla_i u_{a_1,\xi}\nabla_j u_{a_1,\xi}}{1+|\nabla u_{a_1,\xi}|^2}\right) \nabla_i\nabla_j w = O(|x-p|^{d-1}).
\end{equation}
Moreover, thanks to \eqref{error_taylor} replacing $w$ by $w_p$ only introduces an error of size $O(|x-p|^{d-1})$, and likewise freezing the coefficients only introduces an error of size $O(|x-p|^{d-1})$ as well. We thus obtain
\begin{equation}
\left(\delta_{ij}-\frac{\nabla_i u_{a_1,\xi}(p)\nabla_j u_{a_1,\xi}(p)}{1+|\nabla u_{a_1,\xi}(p)|^2}\right) \nabla_i\nabla_j w_p = O(|x-p|^{d-1}),
\end{equation}
which, since $\nabla_i\nabla_j w_p$ has degree at most $d-2$, implies \eqref{to_show_harm}.\\

Now, by the circular symmetry it suffices to analyze the set
\begin{equation}
 \hat{Z}:=Z\cap\{x_4=0\}.
 \end{equation}

\begin{claim}\label{origin_inersect_claim}
There exists a neighborhood of $0$ where $\hat{Z}$ consists of $d=d(0)$ smooth curves intersecting transversally at $0$. Moreover, crossing any of the $2d$ rays,  $w$ changes  sign.
\end{claim}

\begin{proof}[Proof of the claim]
Since $0$ lies on the axis of circular symmetry, $w_0$ is a spherical harmonic that is invariant under rotations in the $x_3x_4$-plane. Thus, in suitable spherical coordinates we have
\begin{equation}
w_0=c_0r^dP_d(\cos\vartheta),
\end{equation}
where $P_d$ is the $d$-th Legendre polynomial, and $c_0\neq 0$ is a constant. As $P_d$ has $d$ distinct roots in $(-1,1)$, we infer that that near $p=0$ the set $\{w_0=0\}\cap \{x_4=0\}$ consists of $d$ curves intersecting transversally, and that $w_0$ changes sign whenever one crosses any of the 2d rays. The corresponding behavior of $\{w=0\}\cap \{x_4=0\}$ now follows from Lemma \ref{lemma_below} below.
\end{proof}

Next, setting $\hat{\Omega}:=\Omega\cap \{x_4=0\}$ we have:

\begin{claim}
There exists a connected component $\hat{D}$ of $\mathrm{Cl}(\hat{\Omega})\setminus\mathrm{Cl}(\hat{Z})$ that does not meet $\partial \hat{\Omega}$.
\end{claim}

\begin{proof} Writing $\hat{\Omega}_i:=\Omega_{a_i,R(a_i,\xi)}$, observe that the ellipses $\partial\hat{\Omega}_1$ and $\partial\hat{\Omega}_2$ intersect at 4 points.
Hence,  $\partial\hat{\Omega}$ consists of $4$ arcs meeting these 4 intersection points $p_1,\ldots,p_4$. Note that $w=0$ on those four intersection points, but $w\neq 0$ anywhere else on $\partial \hat{\Omega}$ by the maximum principle. Hence, $\mathrm{Cl}(\hat{Z})\cap\partial \hat{\Omega}=\{p_1,\ldots,p_4\}$, and consequently there are at most 4 connected components of $\mathrm{Cl}(\hat{\Omega})\setminus \mathrm{Cl}(\hat{Z})$ that meet $\partial \hat{\Omega}$.

On the other hand, 
 by Claim \ref{origin_inersect_claim} around $0$ the set $\hat\Omega\setminus\hat{Z}$ looks like $2 d$ sectors. Let $q_1^{+},\ldots q_d^{+}, q_1^{-},\ldots q_d^{-}$ be points in those distinct sectors, where the sign is according to the sign of $w$.  Note that $d=d(0)\geq 3$, since $M^{a_1,\xi}$ and $M^{a_2,\xi}$ agree at the origin to more than second order. 
 
Suppose towards a contradiction that all connected components of $\mathrm{Cl}(\hat{\Omega})\setminus\mathrm{Cl}(\hat{Z})$ meet $\partial \hat{\Omega}$. Then, since $2d >4$, by the pigeonhole principle two points of the set $\{q_{i}^{\pm}\}$ must be in the same connected component $A$, and these points moreover must be of the same sign, as otherwise $A=(A\cap \{w>0\})\sqcup (A\cap \{w<0\})$. Since open connected sets in $\mathbb{R}^2$ are path connected, we can assume without loss of generality that there is a continuous path $\gamma$ from $q_1^{-}$ to $q_2^{-}$ in $\mathrm{Cl}(\hat{\Omega})\setminus \mathrm{Cl}(\hat{Z})$. We can further assume that $\gamma$ is injective. Finally, let us complete $\gamma$ to a simple closed curve $\tilde{\gamma}$ in  $\mathrm{Cl}(\hat{\Omega})\setminus\mathrm{Cl}(\hat{Z})\cup \{0\}$ by connecting $q_{1}^{-}$ and $q_{2}^{-}$ to $0$ in the small neighborhood of $0$.  Now, by the Jordan curve theorem, $\tilde{\gamma}$ encloses a bounded domain $B$. Letting $q_{1}^+$ and $q_2^{+}$ be the points in the two sectors neighboring $q_1^-$,  one of them, without loss of generality $q_1^{+}$, is necessarily in $B$. But then the connected component $\hat{D}$ of $q_1^{+}$ in $\mathrm{Cl}(\hat{\Omega})\setminus \mathrm{Cl}(\hat{Z})$ does not intersect the boundary: If it did, a curve from $q_1^{+}$ to the boundary would have had to intersect $\tilde{\gamma}$, which is impossible by intermediate value theorem. This proves the claim.
\end{proof}

Finally, considering the orbit of the enclosed region $\hat{D}$ under the $\mathrm{SO}(2)$-symmetry, this implies that there is a domain $D \subseteq \Omega$ such that $w=0$ on $\partial D$ and $w>0$ or $w<0$ in $D$. This contradicts the maximum principle, and thus concludes the proof of the theorem.
\end{proof}

In the above proof we used the following lemma:

\begin{lemma}\label{lemma_below}
Let $\gamma_0^k(r)=(r\cos\theta_0^k,r\sin\theta_0^k)$ be a zero ray of $w_0$ around $0$. Then there exists a  corresponding zero curve  $\gamma^k(r)=(r\cos\theta^k(r) ,r\sin\theta^k(r))$ of $w$, such that $\lim_{r\rightarrow 0} \theta^k(r)=\theta^k_0$.  Moreover, there exists some $r_0>0$ such that all the zeros of $w$ in $B(0,r_0)$ lie on such a curve. 
\end{lemma}
\begin{proof}
Note that there exist $c>0$ and $C<\infty$ such that  $|w(\gamma_0^k(r))|\leq Cr^{d+1}$, and $ |\langle \nabla w(\gamma_0^k(r)),T \rangle|\geq cr^{d-1}$, and $|\nabla^2 w|\leq Cr^{d-2}$, where $T$ denotes the unit tangent vector to $S(0,r)$. Thus,  there exists $\eps>0$ such that $|\langle \nabla w,T \rangle|\geq \frac{c}{2}r^{d-1}$ on $S(0,r)\cap B(\gamma_0^k(r),\eps r)$. It follows that the equation $w=0$ has a unique solution in $S(0,r)\cap B(\gamma_0^k(r),\eps r)$, and this solution $x$ must in fact lie in $S(0,r)\cap B(\gamma_0^k(r),Dr^2)$, where $D<\infty$. The quantitative version of the implicit function theorem gives that the function $r\rightarrow x(r)$ is smooth, proving the existence of such asserted zero curve $\gamma^k(r)$. Moreover, as $|\gamma^k(r)-\gamma_0^k(r)| \leq Dr^2$, the curve starts at the same angle, as asserted.    Finally, note that there exists $\delta>0$ such that when $|x|=r$ and 
\begin{equation}
x\notin \bigcup_{k=1}^{2d} \left( S(0,r)\cap B(\gamma_0^k(r),\eps r)\right),
\end{equation}
then $|w_0(x)| \geq \delta r^d$. Choosing $r_0$ small enough, this completes the proof of the lemma.
\end{proof}

\bigskip

\subsection{The spectral eccentricity of the HIMW class}\label{sec_spectral_eccent}
In this subsection, we prove that the HIMW construction realizes all spectral eccentricities. Together with spectral uniqueness this will immediately yield that the HIMW class  is homeomorphic to an interval. Furthermore, we will also show that the tip curvature function on the HIMW class is weakly monotone.

Recall that we work with the Hilbert space $\fH=L^2(\mathbb{R},e^{-y^2/4}dy)$, and that $\fp_{+}$ denotes the orthogonal projection to $\fH_+$, which is spanned by the unstable eigenfunctions $\psi_1=1$ and $\psi_2=y$. Moreover, recall from Definition \ref{gen_HIMW_def} (HIMW class) that we work with the class of all HIMW translators,
\begin{equation}
\mathcal{A}=\left\{ \lim_{i\to \infty} M^{a_i,\xi_i} \, | \,  a_i\in [0,1/3] \textrm{ and } \xi_i\to -\infty \right\}. 
\end{equation}
We equip $\mathcal{A}$ with the smooth topology corresponding to smooth convergence on compact subsets.\\
Let us first suitably shift these translators so that their spectral center agrees with the one of the cylinder:

\begin{proposition}[shift map]\label{shift_to_zero_plus}
Given any $\tau_0<0$, for every $M\in \mathcal{A}$ there exists a unique $\alpha=\alpha(M,\tau_0)\in \mathbb{R}$ such that the cylindrical profile function $v_{\cC}=\varphi_{\cC}(v)v$ of the shifted translator $M+\alpha e_1$ satisfies
\begin{equation}
\fp_{+}(v_{\cC}(\tau_0)-\sqrt{2})=0,
\end{equation}
and setting $\mathcal{A}':=\{ M+\alpha(M,\tau_0)e_1 | M\in \mathcal{A} \}$ endowed with the smooth topology, the shift map
\begin{equation}
\mathcal{S}:\mathcal{A}\to \mathcal{A}',\quad M\mapsto M+\alpha(M,\tau_0)e_1
\end{equation}
is a homeomorphism. Moreover, for every $\kappa>0$ there exist $\kappa'>0$ and $\tau_\ast>-\infty$ such that if $\tau_0\leq\tau_\ast$ and $M\in\mathcal{A}$ is strongly $\kappa'$-quadratic from time $\tau_0+1$, then $\mathcal{S}(M)$ is $\kappa$-quadratic at time $\tau_0$.
\end{proposition} 

\begin{proof}
Note that for every $M\in \mathcal{A}$, every $\alpha$ and $\tau_0$,  the renormalized profile function $v^{M,\alpha}(\tau_0)$ of $(M+\alpha e_1)\cap \{x_1=e^{-\tau_0}\}$ can be viewed as an entire function in  $y\in\mathbb{R}$, with the convention that it is equal to zero above the diameter. By reflection-symmetry of the HIMW translators, we always have
\begin{equation}\label{inner_p_y}
\langle v^{M,\alpha}_{\cC}(\tau_0), y \rangle_{\fH} =0.
\end{equation}  
Thus, we only have to analyze the inner product with the constant function $1$. To this end, note  that by convexity of $M$ and the definitions of $v^{\alpha}$ and $v^{\alpha}_{\cC}$,  for every $y\in \mathbb{R}$ the function $\alpha \mapsto v_{\cC}^{M,\alpha}(y,\tau_0)$ is monotonically decreasing. Thus, the function 
\begin{equation}
p^M(\alpha):= \langle  v_{\cC}^{M,\alpha}(\tau_0),1 \rangle_{\fH} 
\end{equation}
is monotonically decreasing. Note that the monotonicity is strict as long as $p^M$ does not vanish. Moreover, observe that $p^M(\alpha)=0$ for $\alpha> e^{-\tau_0}$ and $p^M(\alpha)\to \infty$ for $\alpha\to -\infty$. Hence, by strict monotonicity and continuity there exists a unique $\alpha=\alpha(M,\tau_0)$ such that $p^M(\alpha)=\langle 1,\sqrt{2} \rangle_{\fH}$. In other words, remembering \eqref{inner_p_y}, this is the unique $\alpha$ with
\begin{equation}\label{eq_unique_alpha}
\fp_{+}(v^{M,\alpha}_{\cC}(\tau_0)-\sqrt{2})=0.
\end{equation}
This defines the shift map $\mathcal{S}$. Since each member of $\mathcal{A}$, except $\mathrm{Bowl}_2\times \mathbb{R}$, has its unique tip point at the origin, no two elements of $\mathcal{A}$ are vertical shifts of one another. Hence, $\mathcal{S}$ is injective. To establish the continuity of $\mathcal{S}$, note first that $(M,\alpha,y)\mapsto v^{M,\alpha}(y,\tau_0)$  is continuous, and that
\begin{equation}
\lim_{\alpha\rightarrow -\infty}v^{M,\alpha}(y,\tau_0)=\infty 
\end{equation}
uniformly on compact sets  of $\mathcal{A}\times \mathbb{R}$. Therefore, if $M_i\rightarrow M$, then the sequence $\{\alpha(M_i,\tau_0)\}_{i=1}^{\infty}$ is bounded, so it converges up to a subsequence to some $\alpha\in \mathbb{R}$. But then, by continuity again, we infer that $\fp_{+}(v^{M,\alpha}(\tau_0))=\fp_{+}(\sqrt{2})$, hence $\alpha=\alpha(M,\tau_0)$ by uniqueness. As this is true for every converging subsequence, it follows that $\alpha(M_i,\tau_0)\rightarrow \alpha(M,\tau_0)$, proving the continuity of $M\mapsto \alpha(M,\tau_0)$, and thus of $\mathcal{S}$. 
Finally, for every $M'\in \mathcal{A}'$ denoting $h(M')$ the height of the tip of $M'$, we have the equation
\begin{equation}
\mathcal{S}^{-1}(M')=M'-h(M')e_1.
\end{equation}
Since the height of the tip is continuous, $\mathcal{S}^{-1}$ is continuous as well.\\
Moreover, if $M\in\mathcal{A}$ is strongly $\kappa'$-quadratic from time $\tau_0+1$, then in light of  the proof of Lemma \ref{prop_cut_compatibility}, remembering in particular \eqref{pointwise_diff}, for $\tau\leq \tau_0$ we  get
\begin{equation}\label{strong_k_imp}
\left\|v(\tau)-\sqrt{2}\Big(1-\frac{y^2-2}{4|\tau|}\Big)\right\|_{\fH} \leq C\frac{\kappa'}{|\tau|}\, .
\end{equation}
Since the profile functions of $M+\alpha e_1$ and $M$ are related by
\begin{equation}\label{transofrm_scale_1}
v^{\alpha}(y,\tau)=(1+a)v\left(\frac{y}{1+a},\tau-2\log(1+a)\right), \quad \textrm{where}\quad a= \sqrt{1+\alpha  e^{\tau}}-1,
\end{equation} 
we can expand
\begin{align}
    v^{\alpha}(y,\tau)-\sqrt{2}&=\sqrt{2}a-(1+a)\frac{\left(\frac{y}{1+a}\right)^2-2}{\sqrt{8}|\tau-2\log(1+a)|}+O(\kappa'/|\tau|)\, 
    \end{align}
in $\fH$-norm. It follows that for the unique solution of the orthogonality condition \eqref{eq_unique_alpha} we have
\begin{equation}
|a|\leq C\kappa'/|\tau|\, .
\end{equation}
Thus, choosing $\kappa'$ sufficiently small and $\tau_0\leq\tau_\ast$ sufficiently negative, we conclude that $\mathcal{S}(M)$ is $\kappa$-quadratic at time $\tau_0$. This finishes the proof of the proposition.
\end{proof}

We also need the following version for translators-with-boundary:

\begin{proposition}[shift map with boundary]\label{shift_to_zero_plus_boundary}
Given any $\tau_0<0$, there exists a constant $H\in (e^{-\tau_0},\infty)$  with the following significance. For every $a\in [0,\frac{1}{3}]$ and for every $\xi\leq -H$ there exists a unique $\alpha(a,\xi)\in [-H/2,H/2]$ such that the cylindrical profile function $v_{\mathcal{C}}$ of $M^{a,\xi}+\alpha e_1$ satisfies  
\begin{equation}
\fp_{+}(v_{\cC}(\tau_0)-\sqrt{2})=0.
\end{equation}
Moreover, for each fixed $\xi$, the function $a\mapsto \alpha(a,\xi)$ is continuous.  
\end{proposition}

\begin{proof}
The reasoning is similar as above, but we need to be a tad more careful as we do not know that the $M^{a,\xi}$ are convex. To begin with, let us observe that the same argument as above with $\sqrt{2}$ replaced by $1$ and $2$, respectively, yields the existence of two continuous maps $\alpha_{1}:\mathcal{A}\rightarrow \mathbb{R}$ and $\alpha_{2}:\mathcal{A}\rightarrow \mathbb{R}$ such that the cylindrical profile functions of the shifted translators $M+\alpha_1(M) e_1$ and $M+\alpha_2(M) e_1$ satisfy
\begin{equation}
\fp_{+}(v^1_{\cC}(\tau_0)-1)=0\quad \textrm{and} \quad\fp_{+}(v^2_{\cC}(\tau_0)-2)=0.
\end{equation}
Since $\mathcal{A}$ is compact, we get that $|\alpha_1|,|\alpha_2| \leq H_0/2$ for some constant $H_0<\infty$.\\
Now, suppose there was a sequence $(a_i,\xi_i)$ with $\xi_i\rightarrow -\infty$ such that the asserted $\alpha(a_i,\xi_i)$ does not exist. After passing to a subsequence the translators-with-boundary $M^{a_i,\xi_i}$  converge to some $M\in \mathcal{A}$, and hence for $i$ large enough the  cylindrical profile functions of $M^{a_i,\xi_i}+\alpha_1(M)$ and $M^{a_i,\xi_i}+\alpha_2(M,\tau_0)$ satisfy
\begin{equation}
\fp_{+}(v^{1,i}_{\cC}(\tau_0)-\sqrt{2}) < 0 \quad \textrm{and} \quad \fp_{+}(v^{2,i}_{\cC}(\tau_0)-\sqrt{2}) > 0.
\end{equation}
However, then by the intermediate value theorem, we can find some $\alpha(a_i,\xi_i)$ between $\alpha_2(M)$ and $\alpha_1(M)$ such that  the  cylindrical profile function of $M^{a_i,\xi_i}+\alpha(a_i,\xi_i)e_1$ satisfies 
\begin{equation}
\fp_{+}(v^i_{\cC}(\tau_0)-\sqrt{2})=0.
\end{equation}
Next, to address uniqueness, observe that for every $M\in \mathcal{A}$ the function $p^M$ from the proof of Proposition \ref{shift_to_zero_plus} is strictly monotone with nonvanishing derivative whenever $p^M\neq 0$. Therefore, it follows again from smooth convergence and compactness that for $\xi$ large enough, the value $\alpha(a,\xi)$ is unique. 
Finally, continuity of the map $a\mapsto \alpha(a,\xi)$ follows from uniqueness and boundedness as before.
\end{proof}

Now, consider the eccentricity map
\begin{equation}\label{ecc_map_dep}
\mathcal{E}:\mathcal{A}' \rightarrow \mathbb{R}, \quad M \mapsto  \langle v^M_{\mathcal{C}}(\tau_0),2-y^2 \rangle_{\fH} .
\end{equation}
Observe that the expected value of $\mathcal{E}$ for translators satisfying the sharp asymptotics at time $\tau_0$ is
\begin{equation}
e_0:= \frac{4\sqrt{2\pi}}{|\tau_0|}\, .
\end{equation}

\begin{theorem}[existence with prescribed eccentricity]\label{all_achived_thm} 
There exist constants $\kappa>0$ and $\tau_\ast>-\infty$ with the following significance. For every $\tau_0\leq \tau_\ast$ and every $x\in \mathbb{R}$ with $|x-e_0|\leq \tfrac{\kappa}{10|\tau_0|}$ there exists a shifted HIMW translator $M\in  \mathcal{A}'$ that is $\kappa$-quadratic at time $\tau_0$ and satisfies
\begin{equation}
\mathcal{E}(M)=x.
\end{equation}
\end{theorem}

\begin{proof}
Let $\kappa>0$ and $\tau_\ast>-\infty$ be constants such that Theorem \ref{thm:uniqueness_eccentricity_intro} (spectral uniqueness) applies. 
Let us fix $\tau_0 \leq \tau_{\ast}$ and denote by $\mathcal{B}_\kappa=\mathcal{B}_{\kappa}(\tau_0)$  the set of all translators $M\in\mathcal{A}'$ that are $\kappa$-quadratic at time $\tau_0$. Note that Theorem \ref{thm:uniqueness_eccentricity_intro} (spectral uniqueness) implies that the restricted eccentricity map $\mathcal{E}|_{\mathcal{B}_\kappa}:\mathcal{B}_\kappa\to \mathbb{R}$ is injective. Our goal is to show that the image of $\mathcal{E}|_{\mathcal{B}_\kappa}$ contains the interval
\begin{equation}
I:=\left[e_0-\frac{\kappa}{10|\tau_0|},e_0+\frac{\kappa}{10|\tau_0|}\right]\, .
\end{equation}
Possibly after decreasing $\tau_\ast$,
by Corollary \ref{cor_strong_kappa} (strong $\kappa$-quadraticity) and Proposition \ref{shift_to_zero_plus} (shift map)
for any $\tau_0\leq \tau_\ast$ we can find a reference translator $M_0\in \mathcal{A}'$ that is $\frac{\kappa}{100}$-quadratic at time $\tau_0$. In particular, observe that $\mathcal{E}(M_0) \in \mathrm{Int}(I)$.
Let $M_i:=M^{c_i,\xi_i}+\alpha(c_i,\xi_i) e_1$ be a sequence of shifted HIMW translators-with-boundary converging to $M_0$, where the shift parameters $\alpha(c_i,\xi_i)$ are chosen according to Proposition \ref{shift_to_zero_plus_boundary} (shift map with boundary) to ensure that $\fp_{+}(v^i_{\cC}(\tau_0)-\sqrt{2})=0$.\\
Now, for each $i$, choose the maximal  interval $[a_i,b_i]$ containing $c_i$ such that for every $a\in [a_i,b_i]$,  the translator-with-boundary $M^a_i:=M^{a,\xi_i}+\alpha(a,\xi_i) e_1$ satisfies:
\begin{enumerate}
\item $M^a_i$ is $\kappa$-quadratic at time $\tau_0$, and\label{cont_cond1}
\item we have that\label{cont_cond2}
\begin{equation}\label{image_cont}
\mathcal{E}(M^a_i)\in I.
\end{equation}  
\end{enumerate}  
Here, we interpret Definition \ref{def_mu_qudratic} ($\kappa$-quadraticity) in the setting of translators-with-boundary by demanding that its inequalities must hold literally. This is possible since $\xi_i\to -\infty$, while $\tau_0$ and the constant $H$ from Proposition \ref{shift_to_zero_plus_boundary} (shift map with boundary) are fixed.
Recall that the HIMW construction at any fixed level $\xi_i$ depends continuously on the ellipsoidal parameter and interpolates between a piece of the 3d round bowl and a line times a piece of the 2d bowl. Taking also into account that for any fixed $\xi_i$ the shift function $a\mapsto \alpha(a,\xi_i)$ from Proposition \ref{shift_to_zero_plus_boundary} (shift map with boundary) is continuous, it follows that
\begin{equation}
0<a_i<c_i<b_i<\tfrac{1}{3}.
\end{equation}  

\begin{claim}[endpoints]\label{claim_endpoints} The endpoints elements are mapped to the boundary of the interval $I$, namely for all large $i$ we have
\begin{equation}\label{cont_claim_eq}
\mathcal{E}(M^{a_i}_i),\mathcal{E}(M^{b_i}_i)\in  \partial I\, .
\end{equation}
\end{claim}

\begin{proof}[Proof of the claim]
Since the interval $[a_i,b_i]$ is maximal, either condition (i) or condition (ii) must be saturated at its endpoints. Suppose towards a contradiction that $\mathcal{E}(M^{b_i}_i)\notin  \partial I$ for increasingly high values of $i$. Then, for $M^{b_i}_i$ the condition (i) must be saturated, i.e. at least one of the weak inequalities
\begin{equation}\label{mu_quad_back_not}
\left\|v^{M_i^{b_i}}_{\cC}(y,\tau_{0})-\sqrt{2}+\frac{y^2-2}{\sqrt{8}|\tau_{0}|}\right\|_{\fH} \leq \frac{\kappa}{|\tau_{0}|},
\end{equation}
and
\begin{equation}\label{mu_quad_rad_not}
\sup_{\tau\in [2\tau_0,\tau_0]}\|u^{M_i^{b_i}}(\cdot,\cdot,\tau)\|_{C^4(B(0,2|\tau_0|^{1/100})}  \leq |\tau_0|^{-1/50},
\end{equation}
must be an equality. After passing to a subsequence the $M^{b_i}_i$ converge to a limit $M\in \mathcal{A}'$, which by \eqref{mu_quad_back_not} and \eqref{mu_quad_rad_not} is $\kappa$-quadratic at time $\tau_0$. Thus, by Theorem \ref{thm_precise_from_one_time} ($\kappa$-quadraticity implies strong $\kappa$-quadraticity), the translator $M$ is strongly $5\kappa$-quadratic from time $\tau_{0}$. In particular, $\rho^M(\tau)=|\tau|^{1/10}$ is an admissible graphical radius function for $\tau\leq \tau_0$, so inequality \eqref{mu_quad_rad_not} is a strict inequality for $i$ large enough. Thus, it must be the case that
\begin{equation}\label{mu_quad_back__realy_not}
\left\|v_{\cC}^{M}(\tau_{0})-\sqrt{2}+\frac{y^2-2}{2\sqrt{2}|\tau_{0}|}\right\|_{\fH} = \frac{\kappa}{|\tau_{0}|}
\end{equation}
On the other hand, by the centering condition we have
\begin{equation}
\fp_{+}(v_{\cC}^{M}(\tau_0)-\sqrt{2})=0,
\end{equation}
and Corollary \ref{projection_estimate} (projection estimate) tells us that
\begin{equation}
\|\fp_{-}(v_{\cC}^M(\tau_{0}))\|_{\fH}\\ \leq \frac{\kappa}{100|\tau_0|},
\end{equation}
and the fact that $\mathcal{E}(M)\in   I$ yields
\begin{equation}\label{still_in_int}
\left\|\fp_{0}(v_{\cC}^M(\tau_{0}))-\frac{y^2-2}{2\sqrt{2}|\tau_{0}|}\right\|_{\fH} \leq \frac{\kappa}{10|\tau_0|}\|\psi_0\|_{\fH} \leq \frac{6\kappa}{10|\tau_0|}.
\end{equation}
Adding these estimates implies that
\begin{equation}\label{mu_quad_back__realy_not_less}
\left\|v_{\cC}^{M}(\tau_{0})-\sqrt{2}+\frac{y^2-2}{2\sqrt{2}|\tau_{0}|}\right\|_{\fH} < \frac{\kappa}{|\tau_{0}|}.
\end{equation}
This contradicts \eqref{mu_quad_back__realy_not}, and thus proves the claim.
\end{proof}

Now, if along our sequence we have $\mathcal{E}(M^{a_i}_i)\neq \mathcal{E}(M^{b_i}_i)$ for infinitely many $i$, then we are done. Indeed, in this case by the claim and the intermediate value theorem for each $x\in  I$  we can find some $d_i\in [a_i,b_i]$ such that $\mathcal{E}(M^{d_i}_i)=x$. Passing to a subsequential limit, we get a translator $M\in \mathcal{B}_\kappa$ with $\mathcal{E}(M)=x$.\\

On the other hand, if $\mathcal{E}(M^{a_i}_i)= \mathcal{E}(M^{b_i}_i)$ for all large $i$, then we argue as follows. After passing to a subsequence $M^{a_i}_i$ and $M^{b_i}_i$ converge to some limits $M_1,M_2\in \mathcal{B}_\kappa$ with $\mathcal{E}(M_1)= \mathcal{E}(M_2)\in \partial I$. By Theorem \ref{thm:uniqueness_eccentricity_intro} (spectral uniqueness) we see that $M_1=M_2$. Then, applying Theorem \ref{rado_thm} (monotonicity) we infer that the tip curvature is constant along the construction, namely $k(M)=k(M_0)$ for all $M$ that are obtained as limit of a sequence $M_i^{d_i}$ with $d_i\in [a_i,b_i]$. 
Since $\partial I$ has only two elements it follows that the preimage $\mathcal{E}|_{\mathcal{B}_\kappa}^{-1}(I)$ realizes at most two different tip curvatures. However, choosing $\tfrac{\kappa}{100}$-quadratic reference translators $M_0,M_0',M_0''$ such that their tip curvatures $k(M_0),k(M_0'),k(M_0'')$ are all distinct, this yields the desired contradiction, and thus concludes the proof of the theorem.
\end{proof}

\bigskip

As a corollary we obtain that the HIMW class $\mathcal{A}$ is homeomorphic to a closed interval, and that under any such identification the tip curvature map $k:\mathcal{A}\to [0,1/3]$ becomes weakly monotone:

\begin{corollary}[HIMW class]\label{uniq_2_thm_unshift}
There exists a homeomorphism $\psi:[0,1] \rightarrow \mathcal{A}$, and for any such $\psi$ the composed map $k\circ \psi:[0,1]\to [0,1/3]$ is weakly monotone.
\end{corollary}

\begin{proof}
Theorem \ref{all_achived_thm} (existence with prescribed eccentricity) together with Theorem \ref{thm:uniqueness_eccentricity_intro} (spectral uniqueness) and Proposition \ref{shift_to_zero_plus} (shift map) shows that every $M\in \mathcal{A}^{\mathrm{o}}:= \mathcal{A}-\{\mathbb{R}\times\mathrm{Bowl}_2,\mathrm{Bowl}_3 \}$ has a neighborhood homeomorphic to an interval. Moreover, since every $M\in \mathcal{A}^{\mathrm{o}}$ is $\frac{\kappa}{100}$-quadratic from some time, it also follows that  $\mathcal{A}^{\mathrm{o}}$ is connected. Observing also that $\mathcal{A}^{\mathrm{o}}$ is Hausdorff and second countable, we thus infer that $\mathcal{A}^{\mathrm{o}}$ is an open interval.\\
Let us fix a homeomorphism $\psi_{\mathrm{o}}:(0,1)\to\mathcal{A}^{\mathrm{o}}$. To show monotonicity of $k\circ \psi_{\mathrm{o}}$ it suffices to show, using the notions introduced in the above proof, that for every $\tau_0\leq \tau_\ast$ the map $F_\kappa:= k\circ (\mathcal{E}|_{\mathcal{B}_\kappa})^{-1}: I \to (0,1/3)$ is weakly monotone. To this end, recall that by Claim \ref{claim_endpoints} (endpoints) and the final paragraph of the proof of the theorem for $i$ large enough the eccentricity map sends the endpoint elements to the two different boundary points of the interval $I$. Assume without loss of generality that $\mathcal{E}(M^{a_i}_i)=\min I$ and $\mathcal{E}(M^{b_i}_i)=\max I$. Now, given any $s\in \textrm{Int}(I)$ for  $i$ large enough choose $c_i$ with $\mathcal{E}(M^{c_i}_i)=s$. Then, for any $t\in (s,\max I)$, the intermediate value theorem gives $d_i\in (c_i,b_i)$ with $\mathcal{E}(M^{d_i}_i)=t$. By Theorem \ref{rado_thm} (monotonicity) the tip curvature of the translators-with-boundary satisfies
\begin{equation}
k(M^{c_i}_i) < k(M^{d_i}_i).
\end{equation}
Since the left hand side converges to $F_\kappa(s)$ while the right hand side converges to $F_\kappa(t)$, this shows that $k\circ \psi_{\mathrm{o}}:(0,1)\to (0,1/3)$ is weakly monotone. Possibly after adjusting the definition of $\psi_0$, we can assume without loss of generality that $k\circ \psi_{\mathrm{o}}:(0,1)\to (0,1/3)$ is weakly monotone increasing.\\
Finally, by the above monotonicity, for every sequence $t_i\to 0$ or $t_i\to 1$ a subsequence of $\psi_{\mathrm{o}}(t_i)$ converges to a translator in $\mathcal{A}\setminus \mathcal{A}^{\mathrm{o}}$ whose tip curvature $k$ is $0$ or $1/3$, respectively. Since $\mathcal{A}\setminus \mathcal{A}^{\mathrm{o}}$ consists of two elements whose tip curvatures are $0$ and $1/3$, respectively, the subsequential convergence in fact entails full convergence. We conclude that $\psi_{\mathrm{o}}$ can be extended to a homeomorphism $\psi:[0,1] \rightarrow \mathcal{A}$ such that the composed map $k\circ \psi:[0,1]\to [0,1/3]$ is weakly monotone.
\end{proof}

\bigskip

\subsection{Conclusion of the proof} 

In this subsection, we conclude the proof of the classification theorem, modulo the proof of the spectral uniqueness theorem. To this end, we first show that every noncollapsed translator in $\mathbb{R}^4$ is realized by the HIMW construction:

\begin{theorem}[representation theorem for noncollapsed translators]\label{thm_all_are_HIMW}
Every noncollapsed translator in $\mathbb{R}^4$ is, up to rigid motion and scaling, a member of the HIMW class $\mathcal{A}$.
\end{theorem}

In essence, this will follow by combining several results from other sections. The only ingredient that has not been discussed yet, is that we can shift our translator so that it has the same spectral center as the cylinder. To state this recentering result precisely, recall that for any noncollapsed translator $M\subset\mathbb{R}^4$ (normalized as before) the cylindrical profile function is defined by
\begin{equation}
v_{\mathcal{\cC}} = \varphi_{\cC}(v)v,
\end{equation}
where $v=v(y,\tau)$ is the renormalized profile function of $M\cap \{x_1=e^{-\tau}\}$, and $\varphi_{\cC}$ is a suitable cutoff function. Recall also that we work with the Hilbert space $\fH=L^2(\mathbb{R},e^{-y^2/4}dy)$, and that $\fp_+$ denotes the orthogonal projection to $\fH_+$, which is spanned by the unstable eigenfunctions $\psi_1=1$ and $\psi_2=y$.

\begin{proposition}[recentering]\label{prop_orthogonality}
Given any noncollapsed translator $M\subset\mathbb{R}^4$ (normalized as before), with $M\neq \mathrm{Bowl}_3, \mathrm{Bowl}_2\times \mathbb{R}$, and $\kappa>0$, there exists $\tau_{\ast}={\tau}_\ast(M,\kappa)>-\infty$ so that for any $\tau_0\leq {\tau}_{\ast}$  we can find $\alpha,\beta$ 
so that the cylindrical profile function $v^{\alpha\beta}_{\cC}$ of the shifted translator $M^{\alpha\beta}=M+\alpha e_1 + \beta e_2$ satisfies
\begin{equation}
\fp_+ (v^{\alpha\beta}_\cC(\tau_0)-\sqrt{2})=0,
\end{equation}
and so that $M^{\alpha\beta}$ is $\kappa$-quadratic at time $\tau_0$. 
\end{proposition}

\begin{proof}
We will use a mapping degree argument, similarly as in \cite[Section 4]{ADS2}. For convenience, we set 
\begin{equation}\label{new_vars_scale}
a= \sqrt{1+\alpha  e^{\tau}}-1, \;\;\;\;\;\;\;b=\beta e^{\tau/2}.
\end{equation}
Then, the renormalized profile function $v^{\alpha\beta}$ for the level sets of $M+\alpha e_1+\beta e_2$ relates to the renormalized profile function $v$ for the level sets of $M$ by
\begin{equation}\label{transofrm_scale}
v^{\alpha\beta}(y,\tau)=(1+a)v\left(\frac{y-b}{1+a},\tau-2\log(1+a)\right).
\end{equation}
Our goal is to find a suitable zero of the map
\begin{equation}
\Psi(a,b)=\left(\left\langle \psi_1,v^{ab}_\cC-\sqrt{2} \right\rangle_{\fH},\left\langle \psi_2,v^{ab}_\cC-\sqrt{2}  \right\rangle_{\fH}\right)\, ,
\end{equation}
where $a=a(\alpha,\beta)$ and $b=b(\alpha,\beta)$ are defined via \eqref{new_vars_scale}, while maintaining $\kappa$-quadraticity. To this end, we start with the following estimate:

\begin{claim}\label{tranform_est}
For every $\kappa>0$ there exists $\tau_{\kappa}=\tau_{\kappa}(M)>-\infty$ such that for every $\tau\leq\tau_{\kappa}$ and all $(a,b)\in [-1/|\tau|,1/|\tau|]\times[-1,1]$ 
we have
\begin{equation}
\left|  \Big\langle \frac{\psi_1}{\| \psi_1\|^2_{\fH}},v^{ab}_\cC-\sqrt{2} \Big\rangle_{\fH}-\sqrt{2}\Big(a -\frac{b^2}{4|\tau|}\Big) \right|+\left|  \Big\langle \frac{\psi_2}{\| \psi_2\|^2_{\fH}},v^{ab}_\cC-\sqrt{2}  \Big\rangle_{\fH} -  \frac{b}{\sqrt{2}|\tau|} \right|\leq \frac{\kappa}{100|\tau|}\, .
\end{equation}
\end{claim}

\begin{proof}  By Corollary \ref{cor_strong_kappa} (strong $\kappa$-quadraticity), given any $\kappa'>0$ the translator $M$ is strongly $\kappa'$-quadratic from some time $\tau_{\ast}$. In light of  the proof of Lemma \ref{prop_cut_compatibility}, remembering in particular \eqref{pointwise_diff}, for $\tau\leq \tau_\ast-1$ we  get
\begin{equation}\label{strong_k_imp}
\left\|v(\tau)-\sqrt{2}\Big(1-\frac{y^2-2}{4|\tau|}\Big)\right\|_{\fH} \leq C\frac{\kappa'}{|\tau|}\, .
\end{equation}
Since $|a|\leq 1/|\tau|$ and $|b|\leq 1$, together with \eqref{transofrm_scale} this implies
\begin{align}
    v^{\alpha\beta}(y,\tau)-\sqrt{2}&=\sqrt{2}a-(1+a)\frac{\left(\frac{y-b}{1+a}\right)^2-2}{\sqrt{8}|\tau-2\log(1+a)|}+O(\kappa'/|\tau|)\nonumber\\
    &=\sqrt{2}a-\frac{b^2}{\sqrt{8}|\tau|} +  \frac{b}{\sqrt{2}|\tau|}y+O(\kappa'/|\tau|)
\end{align}
in $\fH$-norm. Choosing $\kappa'\ll \kappa$, together with standard Gaussian tail estimates, the claim follows.
\end{proof}

Now, consider the map
\begin{equation}
\Psi_0(a,b)=\sqrt{2}\Big(a-\frac{b^2}{4|\tau|}, \frac{b}{2|\tau|}\Big).
\end{equation}
By Claim \ref{tranform_est}, for $\kappa$ small enough, for $\tau\leq \tau_\kappa$ the maps $\Psi$ and $\Psi_0$ are homotopic when restricted to the boundary of
\begin{equation}
D:=\left\{(a,b)\;|\; |\tau|^2a^2+b^2 \leq 100 \kappa^2\right\},
\end{equation}
where the homotopy can be chosen through maps avoiding the origin. Because the winding number of $\Psi_0|_{\partial D}$ around the origin is $1$, there exists $(a,b)\in D$ with $\Psi(a,b)=0$.  Finally,  by the above estimates, the shifted translator $M^{\alpha\beta}$ is  $\kappa$-quadratic at time $\tau_0$.
\end{proof}

We can now prove the representation theorem:

\begin{proof}[Proof of Theorem \ref{thm_all_are_HIMW}]
Recall first that by Theorem \ref{thm_symmetry} (circular symmetry) every noncollapsed translator in $\mathbb{R}^4$ is $\mathrm{SO}(2)$-symmetric. Now, let $M\subset\mathbb{R}^4$ be a noncollapsed translator that is neither a 3d round bowl nor the product of a line and a 2d bowl. After a rigid motion and rescaling, we can assume that $M$ translates with unit speed in positive $x_1$-direction, and that the circular symmetry is in the $x_3x_4$-plane centered at the origin. 
Furthermore, by Proposition \ref{prop_orthogonality} (recentering) given $\kappa>0$ and $\tau_0\ll 0$ after a suitable shift in the $x_1x_2$-plane we can assume that the cylindrical profile function of $M$ satisfies
\begin{equation}
\fp_{+}(v_{\cC}(\tau_0)-\sqrt{2})=0,
\end{equation}
and that $M$ is $\tfrac{\kappa}{100}$-quadratic from time $\tau_0$. Let us fix a reference translator $M_0\in \mathcal{A}'$ as in the previous subsection. Possibly after decreasing $\tau_0$ we can assume that $M_0$ is also $\tfrac{\kappa}{100}$-quadratic at time $\tau_0$. Since both $M$ and $M_0$ are $\frac{\kappa}{100}$-quadratic at time $\tau_0$, it follows that 
\begin{equation}
|\mathcal{E}(M)-\mathcal{E}(M_0)| \leq \frac{\kappa}{10|\tau_0|}.
\end{equation}
Hence, by Theorem \ref{all_achived_thm} (existence with prescribed eccentricity) there exists a HIMW translator $M'\in \mathcal{A}'$ that is $\kappa$-quadratic at time $\tau_0$ and satisfies
\begin{equation}
\mathcal{E}(M')=\mathcal{E}(M).
\end{equation}
Therefore, remembering also that $\fp_+(v^{M'}_{\cC}(\tau_0))=\fp_{+}(\sqrt{2})=\fp_{+}(v^M_{\cC}(\tau_0))$ by construction, we can apply Theorem \ref{thm:uniqueness_eccentricity_intro} (spectral uniqueness) to conclude that the translators $M$ and $M'$ coincide.
\end{proof}

Modulo the spectral uniqueness theorem, which will be established in the next section, we can now conclude the proof  of our main classification result and its corollary, which we restate here:

\begin{theorem}[{classification of noncollapsed translators in $\mathbb{R}^4$}]\label{thm_main_restated}
Every noncollapsed translator in $\mathbb{R}^4$ is, up to rigid motion and scaling, either (i) $\mathbb{R}\times \mathrm{Bowl}_2$, or (ii) the 3d round bowl $\mathrm{Bowl}_3$, or (iii) belongs to the one-parameter family of 3d oval bowls $\{M_k\}_{k\in (0,1/3)}$ constructed by Hoffman-Ilmanen-Martin-White.
\end{theorem}

\begin{proof} By Theorem \ref{thm_all_are_HIMW} (representation theorem for noncollapsed translators) every noncollapsed translator in $\mathbb{R}^4$ that is neither a 3d round bowl nor a line times a 2d bowl, is up to rigid motion and scaling a member of $\mathcal{A}^{\mathrm{o}}:=\mathcal{A}-\{\mathrm{Bowl}_3,\mathbb{R}\times\mathrm{Bowl}_2\}$. Hence it suffices to classify members of $\mathcal{A}^{\mathrm{o}}$.
On the one hand, we have seen in Corollary \ref{uniq_2_thm_unshift} (HIMW class) that $\mathcal{A}^{\mathrm{o}}$ is homeomorphic to an open interval over which $k$ is weakly monotone. On the other hand,  \cite{CHH_tip} shows that $\mathcal{A}^{\mathrm{o}}$ is analytically equivalent to an interval over which $k$ is analytic. As any weakly monotone analytic function is strictly monotone, we conclude that $k:\mathcal{A}\to [0,1/3]$ is bijective. This finishes the proof of the classification theorem.
\end{proof}

\bigskip

\section{Proof of the spectral uniqueness theorem}\label{sec_profile}

The goal of this section is to prove the spectral uniqueness theorem, which we restate here for convenience of the reader:

\begin{theorem}[spectral uniqueness]\label{thm:uniqueness_eccentricity} There exist $\kappa>0$ and $\tau_{\ast}>-\infty$ with the following significance:
Suppose $M^1$ and $M^2$ are noncollapsed translators in $\mathbb{R}^4$ (neither 3d round bowl, nor $\mathbb{R}\times$ 2d-bowl, and normalized and centered as before) that are $\kappa$-quadratic at time $\tau_0 \leq {\tau}_{\ast}$. If their cylindrical profile functions $v^1_{\cC}$ and $v^2_{\cC}$  satisfy
\begin{equation}\label{rest_eq_cent}
\fp_{+}(v^1_{\cC}(\tau_0))=\fp_{+}(v^2_{\cC}(\tau_0))\;\;\;\;\textrm{(equal spectral center)},
\end{equation}
and
\begin{equation}
\fp_{0}(v^1_{\cC}(\tau_0))=\fp_{0}(v^2_{\cC}(\tau_0))\;\;\;\; \textrm{(equal spectral eccentricity)},
\end{equation}
then
\begin{equation}
M^1=M^2.
\end{equation} 
\end{theorem}

We recall that given $M\subset \mathbb{R}^4$ (normalized as before, namely such that it translates with unit speed in positive $x_1$-direction and such that the circular symmetry is in the $x_3x_4$-plane centered at the origin), we denote by $V(x,t)$ the profile function of the level sets $M\cap \{x_1 = -t\}$, and write
\begin{align}
V(x,t)=\sqrt{-t}\, v(y,\tau), \quad\textrm{where}\quad y=\frac{x}{\sqrt{-t}},\quad\tau=-\log (-t).
\end{align}
The cylindrical profile function is defined by
\begin{equation}
v_{\cC}(y,\tau)=\varphi_{\cC}(v(y,\tau))v(y,\tau).
\end{equation}
Here, we  fix a sufficiently small constant $\theta>0$ and a smooth cutoff function $\varphi_\cC:\mathbb{R}^+\to [0,1]$, such that
\begin{align}
&\varphi_\cC(v)=0 \quad\text{if}\; v\leq \tfrac58 \theta, && \varphi_\cC(v)=1\quad \text{if}\; v\geq  \tfrac78 \theta, 
\end{align}
and
\begin{align}\label{eq_def.collar.cutoff_pre}
 0\leq \varphi_\cC'\leq 5/\theta, &&|\varphi_\cC''|\leq 25/\theta^2, &&|\varphi_\cC'''|\leq 125/\theta^3. 
\end{align}
We also recall that the evolution of $v_{\cC}$ is governed by the Ornstein-Uhlenbeck operator
\begin{equation}
\mathfrak{L}=\partial_y^2-\frac{y}{2}\partial_y+1,
\end{equation}
which is a self-adjoint operator on the Hilbert space $\fH:=L^2(\mathbb{R},e^{-y^2/4} dy)$, and
\begin{equation}
\fH=\fH_+\oplus \fH_0 \oplus \fH_-,
\end{equation} 
where $\fH_+$ is spanned by the unstable eigenfunctions $\psi_1=1$ and $\psi_2=y$, and $\fH_0$ is spanned by the neutral eigenfunction $\psi_0=y^2-2$, and that we write $\mathfrak{p}_\pm$ and $\mathfrak{p}_0$ for the orthogonal projections on $\fH_\pm$ and $\fH_0$. Finally, we recall that \eqref{rest_eq_cent} is in fact automatically satisfied as a consequence of our centering condition
\begin{equation}
\mathfrak{p}_{+}(v_{\cC}(\tau_0)-\sqrt{2})=0.
\end{equation}

Now, similarly as in \cite[Figure 1]{ADS2} we consider the following regions:
\begin{definition}[regions]
Fixing $\theta>0$ sufficiently small and $L<\infty$ sufficiently large, we call
\begin{itemize}
\item { $\mathcal{C}=\{ v \geq \tfrac58 \theta\}$} the \emph{cylindrical region},
\item $\mathcal{T}=\{ v \leq 2 \theta\}$ the \emph{tip region}, which can be decomposed as the union of the \emph{soliton region}  $\mathcal{S}=\{ v \leq L /\sqrt{|\tau|} \}$ and the  \emph{collar region}
 $\mathcal{K}=\{L /\sqrt{|\tau|}\leq  v \leq 2 \theta\}$.
\end{itemize}
\end{definition}
Observe that the cutoff function $\varphi_\cC$ from above localizes in the cylindrical region, namely
\begin{equation}
\mathrm{spt}(v_\cC)\subset\mathcal{C}.
\end{equation}
To localize in the tip region, we fix a smooth cutoff function $\varphi_\cT(v)\in [0,1]$, such that
\begin{align}\label{cutoff_tip_region}
&\varphi_\cT(v)=1 \quad\text{if}\; v\leq \theta, && \varphi_\cT(v)=0\quad \text{if}\; v\geq  2 \theta\, .
\end{align}
In the tip region, say the one with $y>0$, we consider the inverse profile function $Y(v,\tau)$ defined by
\begin{equation}
Y(v(y,\tau),\tau)=y,
\end{equation}
and its zoomed in version $Z$ defined by
\begin{equation}\label{zoomed_in_5}
Z(\rho,\tau)= |\tau|^{1/2}\left(Y(|\tau|^{-1/2}\rho,\tau)-Y(0,\tau)\right).
\end{equation}
By convention during the whole section  $\theta$ is a fixed small constant and $L$ is a fixed large constant. During the proof one is allowed to decrease $\theta$ and increase $L$ at finitely many instances, as needed or convenient.

\bigskip
 
\subsection{Evolution equations}

In this subsection we compute the evolution equations of the profile functions, both in the cylindrical region and the tip regions.

As before, we denote by $V(x,t)$ the profile function of the level set $M\cap \{x_1 = -t\}$ of our translator, and write
\begin{align}
v(y,\tau) = e^{\tau/2} V(e^{-\tau/2} y, -e^{-\tau})\, .
\end{align}

\begin{proposition}[evolution equation for profile function]\label{prop_evol_profile}
The profile function $V(x,t)$ and its renormalized version $v(y,\tau)$ satisfy\footnote{For comparison, the profile function $U$ and renormalized profile function $u$ of a mean curvature flow of surfaces would satisfy the simpler equations $U_t = \frac{U_{xx}}{1+U_x^2}-\frac{1}{U}$ and $u_\tau=\frac{u_{yy}}{1+u_y^2}-\frac{y}{2}u_y+\frac{u}{2}-\frac{1}{u}$, respectively.}
\begin{equation}\label{V eq}
V_t=\frac{(1+V_t^2)V_{xx}+(1+V_x^2)V_{tt}-2V_xV_tV_{xt}}{1+V_x^2+V_t^2}-\frac{1}{V}.
\end{equation}
and 
\begin{align}
v_\tau=&\frac{v_{yy}}{1+v_y^2}-\frac{y}{2}v_y+\frac{v}{2}-\frac{1}{v}+ e^{\tau}\cN[v],
\end{align}
where 
\begin{multline}
\cN[v]=\frac{\left( v_y(v_\tau-\frac{v}{2})-\frac{y}{2}\right)^2}{(1+v_y^2)\left(1+v_y^2+e^{\tau}(v_\tau+\frac{y}{2}v_y-\frac{v}{2})^2\right) }v_{yy}\\
+\frac{(1+v_y^2)v_{\tau\tau}-2\left(v_y(v_\tau -\frac{v}{2})-\frac{y}{2}\right) v_{\tau y}+\frac{1}{4}(1+v_y^2)(yv_y-v)}{1+v_y^2+e^{\tau}(v_\tau+\frac{y}{2}v_y-\frac{v}{2})^2 }\, .
\end{multline}
\end{proposition}

\begin{proof} We parametrize our translator $M\subset \mathbb{R}^4$ by
\begin{equation}
X(x,t,\theta)=(-t,x,V(x,t)\cos\theta,V(x,t)\sin\theta).
\end{equation}
Setting $e_r=\cos\theta e_3+\sin\theta e_4$ and $e_s=-\sin\theta e_3+\cos\theta e_4$ we can express the tangent vectors as
\begin{align}
& X_x=e_2+V_xe_r,&&X_t=-e_1+V_te_r,  && X_\theta=Ve_s.
\end{align}
Thus, the non-vanishing components of the induced metric are given by
\begin{align}
& g_{xx}=1+V_x^2, && g_{tt}=1+V_t^2, && g_{xt}=V_tV_x, && g_{\theta\theta}=V^2.
\end{align}
Hence, the non-vanishing components of the inverse metric are
\begin{align}
& g^{xx}=\frac{1+V_t^2}{1+V_x^2+V_t^2}, && g^{tt}=\frac{1+V_x^2}{1+V_x^2+V_t^2}, && g^{xt}=-\frac{V_tV_x}{1+V_x^2+V_t^2} && g^{\theta\theta}=V^{-2}.
\end{align}
Next, the upwards unit normal equals
\begin{align}\label{eq_unitnorm}
&N=\frac{V_te_1-V_xe_2+e_r}{\sqrt{1+V_x^2+V_t^2}}.
\end{align}
Furthermore, we have
\begin{align}
& X_{xx}=V_{xx}e_r, &&X_{tt}=V_{tt}e_r,  && X_{xt}=V_{xt}e_r, && X_{\theta\theta}=-Ve_r.
\end{align}
Using the above formulas, we can now compute
\begin{equation}
H=\langle \Delta X,N\rangle=\left(\frac{(1+V_t^2)V_{xx}+(1+V_x^2)V_{tt}-2V_xV_tV_{xt}}{1+V_x^2+V_t^2}-\frac{1}{V}\right)\langle e_r,N\rangle.
\end{equation}
Together with the translator equation $H=\langle e_1,N\rangle$ and equation \eqref{eq_unitnorm}, this yields
\begin{equation}
V_t=\frac{(1+V_t^2)V_{xx}+(1+V_x^2)V_{tt}-2V_xV_tV_{xt}}{1+V_x^2+V_t^2}-\frac{1}{V},
\end{equation}
which proves the first evolution equation.\\

Next, observing that
\begin{equation}
V_x=v_y,
\end{equation}
and
\begin{equation}
V_t=-\frac{1}{2}(-t)^{-\frac{1}{2}}v+\frac{x}{2}(-t)^{-1}v_y+(-t)^{-\frac{1}{2}}v_\tau=e^{\frac{\tau}{2}}\left(v_\tau+\frac{y}{2}v_y-\frac{v}{2}\right),
\end{equation}
as well as
\begin{equation}
V_{xx}=e^{\frac{\tau}{2}}v_{yy},
\end{equation}
and 
\begin{equation}
V_{tt}=e^{\frac{3\tau}{2}}\left(v_{\tau\tau}+yv_{\tau y}+\frac{y^2}{4}v_{yy}+\frac{y}{4}v_y-\frac{v}{4}\right),
\end{equation}
and
\begin{equation}
V_{xt}=e^{\tau}\left(v_{y\tau}+\frac{y}{2}v_{yy}\right),
\end{equation}
we infer that
\begin{align}
e^{\frac{\tau}{2}}\left(v_\tau+\frac{y}{2}v_y-\frac{1}{2}v\right)=
\frac{1+e^{\tau}\left(v_\tau+\frac{y}{2}v_y-\frac{v}{2}\right)^2}{1+v_y^2+e^{\tau}\left(v_\tau+\frac{y}{2}v_y-\frac{v}{2}\right)^2}e^{\frac{\tau}{2}}v_{yy}
+\frac{(1+v_y^2)e^{\frac{3\tau}{2}}\left(v_{\tau\tau}+yv_{\tau y}+\frac{y^2}{4}v_{yy}+\frac{y}{4}v_y-\frac{v}{4}\right)}{1+v_y^2+e^{\tau}\left(v_\tau+\frac{y}{2}v_y-\frac{v}{2}\right)^2}\nonumber\\
-\frac{2v_ye^{\frac{3\tau}{2}}\left(v_\tau+\frac{y}{2}v_y-\frac{v}{2}\right)\left(v_{y\tau}+\frac{y}{2}v_{yy}\right)}{1+v_y^2+e^{\tau}\left(v_\tau+\frac{y}{2}v_y-\frac{v}{2}\right)^2}-\frac{e^{\frac{\tau}{2}}}{v}\, .
\end{align}
Together with the formula
\begin{equation}\label{identity.fraction}
\frac{1+b}{1+a+b}=\frac{1}{1+a}+\frac{ab}{(1+a)(1+a+b)}
\end{equation}
this implies
\begin{align}
v_\tau=&\frac{v_{yy}}{1+v_y^2}-\frac{y}{2}v_y+\frac{v}{2}-\frac{1}{v}+ e^{\tau}\cN[v],
\end{align}
where
\begin{multline}
\cN[v]=\frac{v_y^2(v_\tau+\frac{y}{2}v_y-\frac{v}{2})^2}{(1+v_y^2)\left(1+v_y^2+e^{\tau}(v_\tau+\frac{y}{2}v_y-\frac{v}{2})^2\right) } v_{yy}\\
+\frac{(1+v_y^2)(v_{\tau\tau}+yv_{\tau y}+\frac{y^2}{4}v_{yy}+\frac{y}{4}v_y-\frac{v}{4})-2v_y(v_\tau+\frac{y}{2}v_y-\frac{v}{2})(v_{\tau y}+\frac{y}{2}v_{yy})}{1+v_y^2+e^{\tau}(v_\tau+\frac{y}{2}v_y-\frac{v}{2})^2 } \, .
\end{multline}
Grouping together terms proportional to $v_{yy}$, $v_{\tau\tau}$ and $v_{\tau y}$, respectively, this proves the proposition.
\end{proof}

As before, in the tip regions, we consider the inverse profile function $Y(v,\tau)$ defined as the inverse function of $v(y,\tau)$, and its zoomed in version $Z$ defined by
\begin{equation}
Z(\rho,\tau)= |\tau|^{1/2}\left(Y(|\tau|^{-1/2}\rho,\tau)-Y(0,\tau)\right).
\end{equation}

\begin{proposition}[evolution equation for inverse profile function]\label{prop_inverse.evol}
We have
\begin{equation}\label{eq_Y.evolution}
Y_\tau=\frac{Y_{vv}}{1+Y_v^2}+\frac{1}{v}Y_v +\frac{1}{2}(Y-vY_v)+e^{\tau}\mathcal{M}[Y],
\end{equation}
where
\begin{multline}
\mathcal{M}[Y]= \frac{\left( (\tfrac{1}{2}Y-Y_\tau)Y_v+\tfrac{v}{2}  \right)^2}{(1+Y_v^2)\left(1+Y_v^{2}+e^{\tau}(\frac{Y}{2}-Y_\tau-\frac{v}{2}Y_v)^2\right) } Y_{vv}\\
+\frac{(1+Y_v^2)Y_{\tau\tau}+(v+YY_v-2Y_vY_\tau)Y_{v\tau}+\frac{1}{4}(1+Y_v^2)(vY_v-Y)}{1+Y_v^{2}+e^{\tau}(\frac{Y}{2}-Y_\tau-\frac{v}{2}Y_v)^2}\, .
\end{multline}
\end{proposition}

\begin{proof} Differentiating $y=Y(v(y,\tau),\tau)$ yields
\begin{align}\label{eq_impl}
&0=Y_\tau+Y_v v_\tau, &&1=Y_v v_y.
\end{align}
Differentiating again gives
\begin{align}
&0=Y_{\tau\tau}+2Y_{\tau v}v_\tau +Y_{vv}v_\tau^2+Y_v v_{\tau\tau}, && 0=Y_{vv}v_y^2+Y_v v_{yy}, && 0= Y_{\tau v}v_y+Y_{vv}v_yv_\tau + Y_v v_{\tau y}.
\end{align}
Solving these equations we obtain
\begin{align}
& v_\tau=-Y_v^{-1}Y_\tau, && v_y=Y_v^{-1},
\end{align}
and
\begin{align}
& v_{\tau\tau}=-Y_v^{-1}Y_{\tau\tau}+ 2Y_v^{-2}Y_\tau Y_{\tau v} -Y_v^{-3}Y_\tau^2Y_{vv} , && v_{yy}=-Y_v^{-3}Y_{vv}, && v_{\tau y}=-Y_v^{-2}Y_{\tau v}+Y_v^{-3}Y_\tau Y_{vv}.
\end{align}
Together with the evolution equation for $v$ this yields
\begin{align}
Y_\tau&=-Y_v\left(\frac{v_{yy}}{1+v_y^2}-\frac{y}{2}v_y+\frac{v}{2}-\frac{1}{v}+ e^{\tau}\cN[v]\right)\nonumber\\
&=\frac{Y_{vv}}{1+Y_v^2}+\frac{1}{v}Y_v +\frac{1}{2}(Y-vY_v)-e^{\tau}Y_v \cN[v].
\end{align}
Finally, to express $\cN[v]$ in terms of $Y$, we compute
\begin{align}
Y_v\cN[v]\left(1+Y_v^2+e^{\tau}(Y_\tau-\tfrac{1}{2}Y+\tfrac{v}{2}Y_v)^2 \right)&=Y_v^3\cN[v]\left(1+v_y^2+e^{\tau}(v_\tau+\tfrac{y}{2}v_y-\tfrac{v}{2})^2 \right)\nonumber\\
&=-A_{vv}Y_{vv}-A_{v\tau}Y_{v\tau}-A_{\tau\tau}Y_{\tau\tau}-\tfrac14(1+Y_v^2)(vY_v-Y),
\end{align}
where
\begin{equation}
A_{\tau\tau}=1+Y_v^2,
\end{equation}
and
\begin{align}
A_{v\tau}=-2Y_vY_\tau(1+v_y^2) -2Y_v \left(v_y(v_\tau-\tfrac{v}{2})-\tfrac{y}{2}\right)
=-2Y_vY_\tau  +Y_v Y  +  v,
\end{align}
and
\begin{align}
A_{vv}= \frac{(v_y(v_\tau-\tfrac{v}{2})-\tfrac{y}{2})^2}{(1+v_y^2)}
+(1+v_y^2)Y_\tau^2+2\left( v_y(v_\tau-\tfrac{v}{2})-\tfrac{y}{2} \right) Y_\tau
=\frac{\left(   (\tfrac{1}{2}Y-Y_\tau)Y_v +\tfrac{v}{2} \right)^2}{(1+Y_v^2)}.
\end{align}
This proves the proposition.
\end{proof}

\bigskip

\subsection{Maximum principle estimates}

The goal of this subsection is to prove the following a priori estimate:

\begin{proposition}[almost quadratic concavity]\label{prop_almost_quad_conc} There exist constants $\kappa>0$ an $\tau_{\ast}>-\infty$ with the following significance. If $M$ is $\kappa$-quadratic at time $\tau_0 \leq {\tau}_{\ast}$, then its profile function $v$ satisfies
\begin{equation}
(v^2)_{yy}\leq \frac{e^\tau}{v^2}.
\end{equation}
for every $\tau \leq \tau_0$. 
\end{proposition}

To show this, we will adapt the argument from \cite[Section 5]{ADS2} to our setting. To begin with, we have the following cylindrical derivative estimates away from the tip:

\begin{lemma}[derivative estimates]\label{lemma_derivatives}
For every $\eps>0$, there exist $\kappa(\eps)>0$,  $L_0(\eps)<\infty$ and $T_\ast(\eps)>-\infty$ so that the profile function $V(x,t)$ of any $\kappa$-quadratic solution  satisfies
\begin{multline}\label{multi_est}
|V_x|+V|V_{xx}|+V^2|V_{xxx}|+V^3|V_{xxxx}|+|VV_t+1|+ V^2|V_{tx}| + V^3|V_{txx}|+V^4|V_{txxx}| \\
+ |V^3V_{tt}-1|+ V^4|V_{ttx}|+V^5 |V_{ttxx}|  \leq \eps
\end{multline}
at all points where $V(x,t)\geq L_0\left(\frac{-t}{\log(-t)}\right)^{1/2}$ and $t\leq T_\ast$. 
\end{lemma}

\begin{proof}
By the sharp asymptotics in the tip region from Corollary \ref{cor_unique_asympt} (uniform sharp asymptotics) and convexity, for every $\eps_1>0$ there exist $L_1<\infty$ and $T_1>-\infty$ such that
\begin{equation}\label{eq_small_slope}
|V_x| \leq \eps_1
\end{equation}
at all points where $V(x,t)\geq L_0\left(\frac{-t}{\log(-t)}\right)^{1/2}$ and $t\leq T_1$.\\
Observe that the left hand side of \eqref{multi_est} is scale invariant and vanishes on $\mathbb{R}^2\times S^1$. Now suppose towards a contraction there are times $t_i\to -\infty$ and points $x_i$ such that
\begin{equation}\label{utaulimit}
\left(\tfrac{\log(-t_i)}{-t_i}\right)^{1/2} V(x_i,t_i)\to \infty,
\end{equation}
but such that the left hand side of \eqref{multi_est} is bigger than $\eps$. Note also that by Corollary \ref{prop_inscribed_lev} (inscribed radius), letting $p_i\in M_{t_i}$ be a point corresponding to $x_i$, for all large $i$ we have
\begin{equation}\label{eq_mean_curv_lower}
H(p_i,t_i)\geq \frac{1}{2V(x_i,t_i)}.
\end{equation}
Let $M^i_t$ be the sequence of flows that is obtained from $M_t$ by shifting $(p_i,t_i)$ to the origin, and parabolically rescaling by $H(p_i,t_i)^{-1}$. By the global convergence theorem \cite[Theorem 1.12]{HaslhoferKleiner_meanconvex}, we can pass to a subsequential limit $M^\infty_t$. It follows from \eqref{eq_small_slope}, \eqref{utaulimit}, \eqref{eq_mean_curv_lower} and Proposition \ref{lemma_asympt_slope} (asymptotic slope) that $M_t^\infty$ splits off two lines. Hence, applying \cite[Lemma 3.14]{HaslhoferKleiner_meanconvex} we infer that $M_t^\infty$ must be a round shrinking $\mathbb{R}^2\times S^{1}$. This yields the desired contradiction, and thus proves the proposition.
\end{proof}

After this preparation, we can now establish the main maximum principle estimate:

\begin{lemma}[maximum principle]\label{lemma_max_princ}
Given a sufficiently large $L<\infty$, if $\max \big((v^2)_{yy}-e^{\tau}v^{-2}\big) >0$ in $\{v \geq L/\sqrt{|\tau|}\}$, then we have $\partial_\tau \big( (v^2)_{yy}-e^{\tau}v^{-2}\big) <0$ at any interior maximum.
\end{lemma}

\begin{proof}[{Proof of Lemma \ref{lemma_max_princ}}]
For this proof it is convenient to work in the $(x,t)$ variables instead of the $(y,\tau)$ variables. Set $Q=V^2$. We will apply the maximum principle to the function
\begin{equation}
\Phi:=Q_{xx}-Q^{-1}=(v^2)_{yy}-e^{\tau}v^{-2}.
\end{equation}
By Proposition \ref{prop_evol_profile} (evolution equation for profile function), remembering also \eqref{identity.fraction}, the function $V$ satisfies
\begin{equation}\label{V eq_rewritten}
V_t=\frac{V_{xx}}{1+V_x^2}-\frac{1}{V}+ \frac{(1+V_x^2)^{-1}V_x^2V_t^2V_{xx}+(1+V_x^2)V_{tt}-2V_xV_tV_{xt}}{1+V_t^2+V_x^2}.
\end{equation}
This implies
\begin{equation}\label{Q eq}
Q_t=\frac{4QQ_{xx}-2 Q_x^2}{4Q+  Q_x^2}-2+\mathcal{E},
\end{equation}
where
\begin{equation}
\mathcal{E}= 2V\frac{(1+V_x^2)^{-1}V_x^2V_t^2V_{xx}+(1+V_x^2)V_{tt}-2V_xV_tV_{xt}}{1+V_t^2+V_x^2}.
\end{equation}
Differentiating \eqref{Q eq} with respect to $x$ yields
\begin{equation}
Q_{xt}=\frac{4QQ_{xxx}}{4Q+Q_x^2}+\frac{4Q_x(2+Q_{xx})(Q_x^2-2QQ_{xx})}{(4Q+Q_x^2)^2}+\mathcal{E}_x.
\end{equation}
Differentiating again gives
\begin{align}
Q_{xxt}=&\frac{4QQ_{xxxx}+4Q_xQ_{xxx}}{4Q+Q_x^2}-\frac{16QQ_xQ_{xxx}(2+Q_{xx})}{(4Q+Q_x^2)^2}\nonumber\\
&+\frac{\left[4Q_{xx}(2+Q_{xx})+4Q_xQ_{xxx}\right](Q_x^2-2QQ_{xx})}{(4Q+Q_x^2)^2}-\frac{16Q_x^2(2+Q_{xx})^2(Q_x^2-2QQ_{xx})}{(4Q+Q_x^2)^3}+\mathcal{E}_{xx}.
\end{align}
In addition, we have
\begin{align}
(Q^{-1})_t=-\frac{4QQ_{xx}-2 Q_x^2}{Q^2(4Q+  Q_x^2)}+\frac{2}{Q^2}-\frac{\mathcal{E}}{Q^2}
=\frac{4Q(Q^{-1})_{xx} }{4Q+  Q_x^2}-\frac{6 Q_x^2}{Q^2(4Q+  Q_x^2)}+\frac{2}{Q^2}-\frac{\mathcal{E}}{Q^2}.
\end{align}
Taking the difference of the above equations, we obtain
\begin{align}
\Phi_{t}=&\frac{4Q}{4Q+Q_x^2}\Phi_{xx}+\frac{4Q_xQ_{xxx}}{4Q+Q_x^2}+\frac{6 Q_x^2}{Q^2(4Q+  Q_x^2)}-\frac{2}{Q^2}
-\frac{16Q_x^2(2+Q_{xx})^2(Q_x^2-2QQ_{xx})}{(4Q+Q_x^2)^3}\nonumber\\
&+\frac{\left[4Q_{xx}(2+Q_{xx})+4Q_xQ_{xxx}\right](Q_x^2-2QQ_{xx})-16QQ_xQ_{xxx}(2+Q_{xx})}{(4Q+Q_x^2)^2}
+\frac{\mathcal{E}}{Q^2}+\mathcal{E}_{xx}.
\end{align}
Now, at an interior maximum of $\Phi$ we have
\begin{align}
\Phi_{xx}\leq 0,\quad\textrm{ and }\quad 0=\Phi_x=Q_{xxx}+Q^{-2}Q_x,
\end{align}
hence
\begin{multline}
\Phi_t\leq  \frac{2 Q_x^2}{Q^2(4Q+  Q_x^2)}-\frac{2}{Q^2}-\frac{16Q_x^2(2+Q_{xx})^2(Q_x^2-2QQ_{xx})}{(4Q+Q_x^2)^3}\\
+\frac{\left[4Q^2Q_{xx}(2+Q_{xx})-4Q_x^2\right](Q_x^2-2QQ_{xx})+16QQ_x^2(2+Q_{xx})}{Q^2(4Q+Q_x^2)^2}+\frac{\mathcal{E}}{Q^2}+\mathcal{E}_{xx}.
\end{multline}
Note that Lemma \ref{lemma_derivatives} (derivative estimates) implies
\begin{align}\label{easy_pointwise_q}
\frac{2 Q_x^2}{Q^2(4Q+  Q_x^2)}+\frac{16QQ_x^2(2+Q_{xx})}{Q^2(4Q+Q_x^2)^2}\leq \frac{C\varepsilon^2}{Q^2}
\end{align}
for some constant $C<\infty$, provided $L$ is large enough and $t\leq T_\ast$. In a similar vain we have:
\begin{claim}[error estimate]\label{claim_error_est_q} We have
\begin{equation}
|\mathcal{E}|\leq \frac{C}{Q}\qquad\textrm{and}\qquad |\mathcal{E}_{xx}|\leq \frac{C\eps}{Q^2}
\end{equation}
for some constant $C<\infty$, provided $L$ is large enough and $t\leq T_\ast$. 
\end{claim}

\begin{proof}[Proof of the claim]
We will repeatedly apply Lemma \ref{lemma_derivatives} (derivative estimates). To begin with, note that
\begin{equation}
V_x^2 V_t^2 |V_{xx}|\leq \eps^2 \frac{(1+\eps)^2}{V^2}\frac{\eps}{V}\leq \frac{C\eps}{V^3},
\end{equation}
hence
\begin{equation}
\left|V\frac{(1+V_x^2)^{-1}V_x^2V_t^2V_{xx}}{1+V_t^2+V_x^2}\right|\leq \frac{C\eps}{Q}.
\end{equation}
Moreover, we have
\begin{equation}
(1+V_x^2) |V_{tt}| \leq \frac{C}{V^3}\qquad \textrm{and}\qquad |V_xV_tV_{xt}| \leq \frac{C\eps}{V^3}.
\end{equation}
This yields the estimate
\begin{equation}
|\mathcal{E}|\leq \frac{C}{Q}.
\end{equation}
Concerning the first and second derivatives observe that
\begin{equation}
\left|(V_x^2 V_t^2 V_{xx})_x\right|\leq \frac{C\eps}{V^4} \qquad \textrm{and}\qquad \left|(V_x^2 V_t^2 V_{xx})_{xx}\right|\leq \frac{C\eps}{V^5},
\end{equation}
as well as
\begin{equation}
\left|\left(\frac{(1+V_x^2)^{-1}}{1+V_t^2+V_x^2}\right)_x\right|\leq \frac{C\eps}{V} \qquad \textrm{and}\qquad \left|\left(\frac{(1+V_x^2)^{-1}}{1+V_t^2+V_x^2}\right)_{xx}\right|\leq \frac{C\eps}{V^2}.
\end{equation}
Together with the product rule this implies
\begin{equation}
\left|\left(V\frac{(1+V_x^2)^{-1}V_x^2V_t^2V_{xx}}{1+V_t^2+V_x^2}\right)_{xx}\right|\leq \frac{C\eps}{Q^2}.
\end{equation}
Arguing similarly we see that
\begin{equation}
\left|\left(V\frac{(1+V_x^2)V_{tt}}{1+V_t^2+V_x^2}\right)_{xx}\right|\leq \frac{C\eps}{Q^2} \qquad \textrm{and}\qquad
\left|\left(V\frac{V_xV_tV_{xt}}{1+V_t^2+V_x^2}\right)_{xx}\right|\leq \frac{C\eps}{Q^2}.
\end{equation}
We conclude that
\begin{equation}
|\mathcal{E}_{xx}|\leq \frac{C\eps}{Q^2}.
\end{equation}
This finishes the proof of the claim.
\end{proof}

Now, thanks to \eqref{easy_pointwise_q} and Claim \ref{claim_error_est_q} (error estimate), taking also into account the fact that $Q\gg 1$ in the region under consideration, we thus obtain
\begin{align}
\Phi_t\leq -\frac{1}{Q^2}-\frac{16Q_x^2(2+Q_{xx})^2(Q_x^2-2QQ_{xx})}{(4Q+Q_x^2)^3}+\frac{\left[4Q^2Q_{xx}(2+Q_{xx})-4Q_x^2\right](Q_x^2-2QQ_{xx})}{Q^2(4Q+Q_x^2)^2}.
\end{align}
Moreover, since $Q=V^2$ and since $V$ is concave, we have
\begin{equation}\label{V concavity}
2QQ_{xx}<Q_x^2.
\end{equation}
Thus, considering signs yields
\begin{align}
\Phi_t\leq -\frac{1}{Q^2}-\frac{16Q_x^2(2+Q_{xx})^2(Q_x^2-2QQ_{xx})}{(4Q+Q_x^2)^3}+\frac{ 4 Q_{xx}(2+Q_{xx}) (Q_x^2-2QQ_{xx})}{(4Q+Q_x^2)^2}.
\end{align}
Furthermore, \eqref{V concavity} implies $Q_{xx}(4Q+Q_x^2)\leq Q_x^2(2+Q_{xx})$, and Lemma \ref{lemma_derivatives} (derivative estimates) gives us $2+Q_{xx}>0$. We thus  conclude that
\begin{align}
\Phi_t\leq -\frac{1}{Q^2}-\frac{12Q_x^2(2+Q_{xx})^2(Q_x^2-2QQ_{xx})}{(4Q+Q_x^2)^3}.
\end{align}
Hence, $\Phi_t< 0$ holds at interior maximum points of $\Phi$ in $\{V \geq L\sqrt{-t/\log(-t)}\}$. This proves the assertion.
\end{proof}

We can now prove the main result of this subsection:

\begin{proof}[{Proof of Proposition \ref{prop_almost_quad_conc} (almost quadratic convexity)}] Fix $\kappa>0$ small enough and $\tau_\ast>-\infty$ negative enough so that the above results apply.
We recall that
\begin{equation}
\Phi= Q_{xx}-Q^{-1}=(v^2)_{yy}-e^{\tau}v^{-2}.
\end{equation}
By Corollary \ref{cor_unique_asympt} (uniform sharp asymptotics) in the tip regions we have that  $Z(\rho,\tau)$ is $\eps$-close to $Z_0(\rho)$, where $Z_0$ is the profile function of the bowl soliton.
Hence, applying \cite[Lemma 4.4]{ADS2} we get that in the soliton region $\mathcal{S}$ we have $(v^2)_{yy}<0$ for $\tau \leq \tau_0$. In particular, $\Phi<0$ in the soliton region.\\
Now, suppose towards a contradiction there is a point $(y_0,\tau_0)$, where $\tau_0\leq \tau_\ast$, with $\Phi(y_0,\tau_0)>0$. It the follows from the paragraph above, and from \ref{lemma_max_princ} (maximum principle) that $\max \Phi(\cdot,\tau)\geq \Phi(y_0,\tau_0)$ for every $\tau \leq \tau_0$. In particular, we have $(v^2)_{yy}(y_{\tau},\tau)\geq c$ for some $c>0$, whenever $v(y_{\tau},\tau)=\max \Phi(\cdot,\tau)$.
Together with $(v^2)_{yy}=2vv_{yy}+2v_y^2<2v_y^2$, which holds by concavity, we infer that $v_{y}^2(y_{\tau},\tau)\geq c/2$. This is in contradiction with $v(y_{\tau},\tau)\sqrt{|\tau|} \to \infty$ and the fact that the soliton region converges to a bowl soliton, and thus proves the proposition.
\end{proof}

In particular, we see that $Y\sim Ce^{-v^2/4}$ in the collar region:

\begin{corollary}[almost Gaussian collar]\label{cor_gaussian_collar}
Given $\eps>0$, there exist $\theta(\eps)>0$, $L_0(\eps)< \infty$, ${\tau}_{\ast}(\eps)>-\infty$ and $\kappa(\eps)>0$ such that if $M$ is $\kappa$-quadratic at time $\tau_0 \leq {\tau}_{\ast}$, then for $\tau\ \leq \tau_0$ in the collar region $\{L_0 /\sqrt{|\tau|}\leq  v \leq 2 \theta\}$ we have
\begin{equation}
\bigg|1+\frac{Yv}{2Y_v}\bigg|\leq \eps\, .
\end{equation}
\end{corollary}

\begin{proof}
Suppose that $-\tau$ is large enough so that  $\Phi\leq 0$ holds. It is enough to show that
\begin{equation}
1-\eps \leq -\frac{y(v^2)_y}{4} \leq 1+\eps.
\end{equation}
First of all, using the description of the intermediate region from Corollary \ref{cor_unique_asympt} (uniform sharp asymptotics) we see that in the region $\{v\leq 2\theta\}$ we have
\begin{equation}\label{collar.y}
\sqrt{2|\tau|}(1-4\theta^2-\delta) \leq y \leq \sqrt{2|\tau|}(1+\delta),\;\;\;\tau \leq {\tau_0}
\end{equation}
for any $M$ that is $\kappa(\delta,\theta)$-quadratic from time $\tau_0 \leq \tau_{\ast}(\delta,\theta)$. 
By Proposition \ref{prop_almost_quad_conc} (almost quadratic concavity), after decreasing $\kappa$ and ${\tau}_{\ast}$ we can assume that $-(v^2)_{yy}+ e^\tau v^{-2}\geq 0$, from which we infer that
\begin{equation}
-(v^2)_y|_{v=2\theta}- e^{\tau}\int^{\sqrt{2|\tau|}(1+\delta)}_{\sqrt{2|\tau|}(1-4\theta^2)} v^{-2} dy \leq -(v^2)_y \leq -(v^2)_y|_{v=L_0/\sqrt{|\tau|}}+e^{\tau}\int^{\sqrt{2|\tau|}(1+\delta)}_{\sqrt{2|\tau|}(1-4\theta^2)} v^{-2} dy\,.
\end{equation}
In the considered region, we have $  v^{-2}\leq L_0^{-2}|\tau|$ and thus
\begin{equation}
-(v^2)_y|_{v=2\theta}-10\theta^2L_0^{-2} e^{\tau}|\tau|^{\frac{3}{2}} \leq -(v^2)_y \leq -(v^2)_y|_{v=L_0/\sqrt{|\tau|}}+10\theta^2L_0^{-2} e^{\tau}|\tau|^{\frac{3}{2}}.
\end{equation}
Finally, using again Corollary \ref{cor_unique_asympt} (uniform sharp asymptotics) and arguing similarly as in the proof of \cite[Lemma 5.7]{ADS2} we obtain
\begin{equation}
-(v^2)_y|_{v=2\theta} \geq \frac{2\sqrt{2}}{\sqrt{|\tau|}}\sqrt{1-2\theta^2-\delta},
\end{equation}
and
\begin{equation}
-(v^2)_y|_{v=L_0/\sqrt{|\tau|}}\leq \frac{2\sqrt{2}}{\sqrt{|\tau|}}(1+CL_0^{-1}),
\end{equation}
for $L_0$ large enough, possibly after decreasing $\kappa$ and ${\tau}_{\ast}$.
Combining the above inequalities yields the desired result.
\end{proof}

\bigskip

\subsection{Difference between solutions}
Given our translators $M^1$ and $M^2$, we consider the difference function of their renormalized profile functions
\begin{equation}
w:=v_1-v_2,
\end{equation}
and its truncated version
\begin{equation}
w_\cC:=v_1 \varphi_\cC (v_1)-v_2\varphi_{\cC}(v_2),
\end{equation}
as well as the difference of the inverse profile functions
\begin{equation}
W:=Y_1-Y_2,
\end{equation}
and its truncated version
\begin{equation}
W_\cT:=W \varphi_\cT (v)\, .
\end{equation}

\begin{proposition}[evolution of difference]\label{prop_evol_w}
The difference function $w$ satisfies the evolution equation
\begin{equation}\label{eq_w.evolution}
w_\tau=\mathfrak{L}w+\mathcal{E}[w]+e^{\tau}\mathcal{F}[w],
\end{equation}
with
\begin{equation}
\mathfrak{L}w=w_{yy}-\frac{y}{2}w_y+w,
\end{equation}
and
\begin{equation}\label{eq_def.E[w]}
\mathcal{E}[w]=-\frac{v_{1,y}^2}{1+ v_{1,y}^2}w_{yy}-\frac{(v_{1,y}+v_{2,y})v_{2,yy}}{(1+v_{1,y}^2)(1+v_{2,y}^2)}w_y+\frac{2-v_1v_2}{2v_1v_2}w,
\end{equation}
and 
\begin{align}
\mathcal{F}[w]=&\frac{ \mathcal{P}[v_1,v_1,w]}{\mathcal{Q}[v_1,v_1]}+\mathcal{R}[v_1,v_2]\left(w_\tau-\tfrac{w}{2}\right)+\mathcal{S}[v_1,v_2]w_y,
\end{align}
where $\mathcal{P},\mathcal{Q},\mathcal{R},\mathcal{S}$ are certain second order differential expressions specified in the proof below.
\end{proposition}

\begin{proof}
We will denote derivatives by
\begin{align}
\partial_y^l\partial_\tau^m v_i=v_{i\, ,\underbrace{y\cdots y}_{l} \underbrace{\tau \cdots \tau}_{m}}.
\end{align}
Using the evolution equations for $v_1$ and $v_2$ from Proposition \ref{prop_evol_profile} (evolution equation for profile function), a straightforward computation shows that $w=v_1-v_2$ satisfies the claimed evolution equation with
\begin{multline}
\mathcal{R}[v_1,v_2]=\frac{ v_{1,y}(1+ v_{1,y}^2)^{-1}  \left[ v_{1,y}( v_{1,\tau}+v_{2,\tau}-\frac{v_1}{2}-\frac{v_2}{2})-y \right] v_{2,yy} -  2v_{1,y}v_{2,\tau y}}{\mathcal{Q}[v_1,v_1]}\\
-\frac{e^\tau (yv_{1,y}+v_{1,\tau}+v_{2,\tau}-\frac{v_1}{2}-\frac{v_2}{2})\mathcal{P}[v_1,v_2,v_2]}{\mathcal{Q}[v_1,v_1]\mathcal{Q}[v_1,v_2]},
\end{multline}
and
\begin{multline}
\mathcal{S}[v_1,v_2]=\frac{ (v_{2,\tau}-\frac{v_2}{2})  \left[ (v_{1,y}+v_{2,y})(v_{2,\tau}-\frac{v_2}{2})-y \right]v_{2,yy}-\frac{v_{1,y}+v_{2,y}}{1+ v_{2,y}^2}\left[ v_{2,y}(v_{2,\tau}-\frac{v_2}{2})-\frac{y}{2} \right]^2 v_{2,yy} }{(1+ v_{1,y}^2)\mathcal{Q}[v_1,v_2]}\\
-\frac{2(v_{2,\tau}- \frac{v_2}{2})   v_{2,y\tau}}{\mathcal{Q}[v_1,v_2]}  +\frac{(v_{1,y}+v_{2,y})\left[ v_{2,\tau\tau}-\frac14(yv_{2,y}-v_2)-\mathcal{N}(v_2)\right]}{\mathcal{Q}[v_1,v_2]},
\end{multline}
where the functions $\mathcal{P}$ and $\mathcal{Q}$ are defined by
\begin{multline}
\mathcal{P}[p,q,r](y,\tau)=(1+ p_y^2)^{-1}\left(p_y( q_{\tau}-\tfrac{q}{2})-\tfrac{y}{2}\right)^2r_{yy} +(1+p_{y}^2 )r_{\tau\tau}\\
-2\left( p_{y}(q_{\tau} -\tfrac{q}{2})-\tfrac{y}{2}\right) r_{\tau y}
+\tfrac{1}{4}(1+p_{y}^2)(yr_{y}-r),
\end{multline}
and
\begin{align}
\mathcal{Q}[p,q](y,\tau)=1+p_{y}^2+e^{\tau}(\tfrac{y}{2}p_{y}+q_{\tau}-\tfrac{q}{2})^2.
\end{align}
This proves the proposition.
\end{proof}

To capture some extra terms from the cutoff, similarly as in \cite[Equation (6.11)]{ADS2} we set
\begin{equation}
\overline{\mathcal{E}}[w,\varphi_{\cC}(v_1)]:= (\partial_\tau-\mathfrak{L})(w \varphi_{\cC}(v_1) )-\varphi_{\cC}(v_1)(\partial_\tau-\mathfrak{L})w +\varphi_{\cC}(v_1) \mathcal{E}[w]-\mathcal{E}[w\varphi_{\cC}(v_1)]\, .
\end{equation}
Moreover, given any scalar function $\varphi$, we write
\begin{equation}
D[\varphi](y,\tau):=\varphi(v_1(y,\tau))-\varphi(v_2(y,\tau)).
\end{equation}

\begin{corollary}[evolution of the truncated difference]\label{lemma_int.cyl.wc}
The function $w_\cC$ satisfies 
\begin{multline}
(\partial_\tau -\mathfrak{L}) w_\cC=\mathcal{E}[w_\cC]+\overline{\mathcal{E}}[w,\varphi_\cC(v_1)]+e^{\tau}\varphi_\cC(v_1)\mathcal{F}[w]\\
-\mathcal{E}[v_2 D[\varphi_\cC]]+D[\varphi_\cC]( v_{2,\tau}- v_{2,yy}+\tfrac{y}{2}v_{2,y})-2v_{2,y}\partial_y D[\varphi_\cC]+v_2 (\partial_\tau -\mathfrak{L})D[\varphi_\cC].
\end{multline}
\end{corollary}

\begin{proof}
First observe that
\begin{equation}
(\partial_\tau-\mathfrak{L})( w \varphi_\cC(v_1))=\mathcal{E}[w\varphi_\cC(v_1)]+\overline{\mathcal{E}}[w,\varphi_\cC(v_1)]+ e^{\tau}\varphi_\cC(v_1)\mathcal{F}[w]\, .
\end{equation}
In addition, we have
\begin{equation}
(\partial_\tau-\mathfrak{L})(v_2D[\varphi_\cC])=D[\varphi_\cC]( v_{2,\tau}- v_{2,yy}+\tfrac{y}{2}v_{2,y})-2v_{2,y}\partial_y D[\varphi_\cC]+v_2 (\partial_\tau -\mathfrak{L})D[\varphi_\cC]\, .
\end{equation}
Using $w_\cC=w\varphi_\cC (v_1)+v_2D[\varphi_\cC]$ and linearity this implies the assertion.
\end{proof}

\begin{proposition}[evolution of inverse difference]\label{prop_evol_w_inv}
The difference function $W$ satisfies the evolution equation
\begin{equation}\label{W.equation}
W_\tau=\frac{W_{vv}}{1+Y_{1,v}^2}+\left(\frac{1}{v}-\frac{v}{2}-\frac{Y_{2,vv}(Y_{1,v}+Y_{2,v})}{(1+Y_{1,v}^2)(1+Y_{2,v}^2)}\right)W_v+\frac{1}{2}W+e^\tau \mathcal{F}[W]\, ,
\end{equation}
with
\begin{equation}
\mathcal{F}[W]=\frac{\mathcal{P}[Y_1,Y_1,W]}{\mathcal{Q}[Y_1,Y_1]}+\mathcal{R}[Y_1,Y_2]\left(\frac{W}{2}-W_\tau\right)+\mathcal{S}[Y_1,Y_2]W_v,
\end{equation}
where $\mathcal{P},\mathcal{Q},\mathcal{R},\mathcal{S}$ are certain second order differential expressions specified in the proof below.
\end{proposition} 

\begin{proof}
Using the evolution equations for $Y_1$ and $Y_2$ from Proposition \ref{prop_inverse.evol} (evolution equation for inverse profile function), we see that $W$ satisfies the claimed evolution with
\begin{equation}
\mathcal{F}=\mathcal{M}[Y_1]-\mathcal{M}[Y_2].
\end{equation}
Then, a straightforward computation shows that $\mathcal{F}$ can be expressed as claimed with
\begin{align}
\mathcal{R}[Y_1,Y_2]=&\frac{(1+Y_{1,v}^2)^{-1}Y_{1,v}\left( (\tfrac{1}{2}Y_{1}-Y_{1,\tau}+\tfrac{1}{2}Y_2-Y_{2,\tau})Y_{1,v}+v  \right)Y_{2,vv}+2Y_{1,v}Y_{2,v\tau}}{Q[Y_1,Y_1]}\nonumber\\
&-\frac{e^\tau \left( \tfrac{1}{2}Y_1-Y_{1,\tau}+\tfrac{1}{2}Y_2-Y_{2,\tau}+vY_{1,v}  \right) \mathcal{P}[Y_1,Y_1,Y_2]}{Q[Y_1,Y_1]Q[Y_1,Y_2]},
\end{align}
and
\begin{align}
\mathcal{S}[Y_1,Y_2]=&\frac{ (\frac{Y_2}{2}-Y_{2,\tau})  \left[ (Y_{1,v}+Y_{2,v})(\frac{Y_2}{2}-Y_{2,\tau})+v \right]v_{2,yy}-\frac{Y_{1,v}+Y_{2,v}}{1+ Y_{2,v}^2}\left[ (\tfrac{1}{2}Y_2-Y_{2,\tau})Y_{2,v}+\tfrac{v}{2}\right]^2 Y_{2,vv} }{(1+ Y_{1,v}^2)Q[Y_1,Y_2]}\nonumber\\
&+\frac{2(\frac{Y_2}{2}-Y_{2,\tau})   Y_{2,v\tau}}{Q[Y_1,Y_2]}  +\frac{(Y_{1,v}+Y_{2,v})\left[ Y_{2,\tau\tau}+\frac14(vY_{2,v}-Y_2)-\mathcal{M}(Y_2)\right]}{Q[Y_1,Y_2]},
\end{align}
where
\begin{align}
\mathcal{P}[p,q,r](v,\tau)=&(1+p_v^2)^{-1} \left( (\tfrac{q}{2}-q_\tau)p_v+\tfrac{v}{2}  \right)^2 r_{vv}\nonumber\\
&+ (1+p_v^2)r_{\tau\tau}+2\left( (\tfrac{q}{2}-q_\tau)p_v+\tfrac{v}{2}  \right)r_{v\tau}+\tfrac{1}{4}(1+p_v^2)(vr_v-r).
\end{align}
and
\begin{equation}
\mathcal{Q}[p,q](v,\tau)=1+p_v^2+e^\tau\left(\tfrac{q}{2}-q_\tau-\tfrac{v}{2}p_v\right)^2.
\end{equation}
This proves the proposition.
\end{proof}

To conclude this subsection, we observe that since the cutoff function $\varphi_\cT(v)$ does not depend on $\tau$, the time derivative of $W_\cT =\varphi_\cT W$ is simply
\begin{equation}
(W_\cT)_\tau = \varphi_\cT W_\tau.
\end{equation}

\bigskip

\subsection{Energy estimates in the cylindrical region} The goal of this subsection is to prove the following energy estimate in the cylindrical region:

\begin{proposition}[energy estimate in the cylindrical region]\label{prop_int_cyl}
For every $\eps>0$  there exist  $\kappa>0$ and ${\tau}_{\ast}>-\infty$ with the following significance.  If $M_1$ and $M_2$ are $\kappa$-quadratic at time $\tau_0 \leq {\tau}_{\ast}$,  then 
\begin{equation}
\| w_\cC-\fp_0w_\cC \|_{\mathcal{D},\infty}\leq  \varepsilon\left( \|w_{\cC}\|_{\mathcal{D},\infty}+\|w\, 1_{\{\theta/2\leq v_1\leq\theta\}}\|_{\mathfrak{H},\infty}\right)+\frac{C}{|\tau_0|} \| w\|_{C^2_{\exp}(\mathcal{C})}.
\end{equation}
\end{proposition}

Recall that our definition of $\kappa$-quadratic imposes the centering condition $\mathfrak{p}_+(v^i_{\cC}(\tau_0)-\sqrt{2})=0$, and observe that this in particular implies that
\begin{equation}\label{eq_w_orth}
\mathfrak{p}_+(w_{\cC}(\tau_0))=0.
\end{equation}
The norms appearing in the energy estimate have been briefly described in the introduction, but let us discuss them in more detail now. Similarly as in \cite{ADS2}, in addition to the Gaussian $L^2$-norm
\begin{equation}
\|f\|_{\mathfrak{H}} := \left( \int f^2  e^{-y^2/4} dy \right)^{1/2},
\end{equation}
one also needs the Gaussian $H^1$-norm
\begin{equation}
\|f\|_{\mathcal{D}} := \left( \int (f^2 +f_y^2) e^{-y^2/4} dy \right)^{1/2},
\end{equation}
and its dual norm
\begin{equation}
\|f\|_{\mathcal{D}^\ast} := \sup_{\|g\|_{\mathcal{D}}\leq 1  } \langle f,g\rangle\, .
\end{equation}
For time-dependent functions this induces the parabolic norms 
\begin{equation}
\|f \|_{\mathcal{X},\infty}(\tau):=\sup_{\tau'\leq \tau }\left( \int_{\tau'-1}^{\tau'} \| f(\cdot,\sigma)\|^2_{\mathcal{X}} \, d\sigma \right)^{1/2},
\end{equation}
where $\mathcal{X}=\fH,\mathcal{D}$ or $\mathcal{D}^\ast$, and we often simply write
\begin{equation}
\| f \|_{\mathcal{X},\infty}:=\| f \|_{\mathcal{X},\infty}(\tau_0).
\end{equation}
In contrast to \cite{ADS2}, we also need exponentially weighted $C^2$-norms to control the higher derivative terms coming from the nonlinearity $e^{\tau}\mathcal{F}[w]$. Specifically, setting
\begin{equation}
 C_\tau:=\left\{ y:  v_1(y,\tau) \geq \tfrac{5}{8} \,\,\, \textrm{or} \,\,\, v_2(y,\tau) \geq \tfrac{5}{8} \theta \right\},
 \end{equation}
  we define
\begin{align}
\| f\|_{C^2_{\exp}(\mathcal{C})}(\tau):=\sup_{\tau'\leq \tau} \left( |\tau'| e^{\tau'}\sup_{y\in C_{\tau'} } \big(|f|+|f_y|+|f_{\tau}|+ |f_{yy}|+|f_{y\tau}|+|f_{\tau\tau}|\big)(y,\tau')\right),
\end{align}
and we often simply write
\begin{equation}
\| f\|_{C^2_{\exp}(\mathcal{C})}:=\| f\|_{C^2_{\exp}(\mathcal{C})}(\tau_0)\, .
\end{equation}

\bigskip

To prove Proposition \ref{prop_int_cyl}, we note that thanks to \eqref{eq_w_orth} and \cite[Lemma 6.7]{ADS2} we have the general energy inequality
\begin{equation}\label{gen_energy_ineq}
\|w_\cC-\fp_0w_\cC\|_{\mathcal{D},\infty}\leq C\|(\partial_\tau -\mathfrak{L})w_\cC\|_{\mathcal{D}^*,\infty}.
\end{equation}
Hence, our task is to estimate $(\partial_\tau -\mathfrak{L})w_\cC$ in the parabolic $\mathcal{D}^\ast$-norm. In contrast to  \cite[Section 6]{ADS2}, this will require estimating several new terms coming from the intrinsic cutoff and the nonlinearities. Specifically, rewriting the conclusion of Corollary \ref{lemma_int.cyl.wc} (evolution of truncated difference) in the form
\begin{align}\label{ev_trunc_diff}
(\partial_\tau -\mathfrak{L}) w_\cC=I+J+K+e^{\tau}\varphi_\cC(v_1)\mathcal{F}[w],
\end{align}
where
\begin{align}
I=&\,\mathcal{E}[w_\cC]+\overline{\mathcal{E}}[w,\varphi_\cC(v_1)],\\
J=&\,( v_{2,\tau}- v_{2,yy}+\tfrac{y}{2}v_{2,y}-\mathcal{E}[v_2])D[\varphi_\cC]-2v_{2,y}\partial_y D[\varphi_\cC],\\
K=&\mathcal{E}[v_2]D[\varphi_\cC]-\mathcal{E}[v_2D[\varphi_\cC]]+v_2 (\partial_\tau -\mathfrak{L})D[\varphi_\cC].
\end{align}
we will now estimate the $\mathcal{D}^\ast$-norm of $I$, $J$, $K$ and  $\varphi_\cC(v_1)\mathcal{F}[w]$ in turn. A term similar to $I$ already appeared in \cite[Section 6]{ADS2}, but the other three terms are new.  Let us recall a few basic facts that will be used frequently for estimating the $\mathcal{D}^\ast$-norm. By the  weighted Sobolev inequality (see e.g. \cite[Lemma 4.12]{ADS1}) multiplication with $1+|y|$ is a bounded operator from $\mathcal{D}$ to $\mathfrak{H}$, hence by duality
\begin{equation}\label{eq_weighted_sob_est}
\| ( 1+|y|) f \|_{\mathcal{D}^\ast} \leq C \| f \|_{\mathfrak{H}}\, .
\end{equation}
Consequently, $\partial_y$ and $\partial_y^\ast=-\partial_y + \frac{y}{2}$ are bounded operators from $\mathcal{D}$ to $\mathfrak{H}$ and from $\mathfrak{H}$ to $\mathcal{D}^\ast$, in particular
\begin{equation}\label{rec_basic_der_norm}
\|  f_y \|_{\mathcal{D}^\ast} \leq C \| f \|_{\mathfrak{H}}\, .
\end{equation}
Also, if $g\in \mathcal{D}$ and $h\in W^{1,\infty}$ then by the product rule $\| hg\|_{\mathcal{D}}\leq 2 \| h \|_{W^{1,\infty}}\| g\|_{\mathcal{D}}$, hence by duality
\begin{equation}\label{eq_product_rule_norm}
\| h f \|_{\mathcal{D}^\ast} \leq 2 \| h\|_{W^{1,\infty}} \| f \|_{\mathcal{D}^\ast}\, .
\end{equation}
In the following, we write
\begin{equation}
D_\tau:=\left\{ y:  \tfrac{5}{8}\theta \leq v_1(y,\tau) \leq  \tfrac{7}{8}\theta  \,\,\, \textrm{or} \,\,\,   \tfrac{5}{8}\theta \leq v_2(y,\tau) \leq  \tfrac{7}{8}\theta  \right\}
 \end{equation}
for the transition region.

\begin{lemma}[{estimate for $I$}]\label{lemma_int.cyl.I}
Given $\varepsilon>0$, there exist $\kappa>0$ and ${\tau}_\ast>-\infty$ such that for $\tau \leq \tau_*$ we have
\begin{align}
\|I(\tau)\|_{\mathcal{D}^*}\leq \varepsilon \big(\|w_\cC(\tau)\|_{\mathcal{D}}+\|w(\tau) 1_{D_{\tau}} \|_{\mathfrak{H}}\big).
\end{align}
\end{lemma}

\begin{proof}
Recall that
\begin{equation}\label{i_restated}
I=\,\mathcal{E}[w_\cC]+\overline{\mathcal{E}}[w,\varphi_\cC(v_1)]\, .
\end{equation}
By Lemma \ref{lemma_derivatives} (derivative estimates), given $\varepsilon>0$ there exists $\tau_*>-\infty$ such that
\begin{equation}\label{eq_cylinder.derivatives}
|v_{i,y}|+|v_{i,\tau}|+
|v_{i,yy}|+|v_{i,\tau y}|+|v_{i,\tau\tau}|
+|v_{i,yyy}|+|v_{i,\tau yy}|+|v_{i,\tau \tau y}|+|v_{i,\tau\tau\tau}|\leq \varepsilon
\end{equation}
holds on $C_\tau$ for $\tau\leq \tau_\ast$ and $i=1,2$. Using these derivative estimates, similarly as in \cite[Lemma 6.8 and Lemma 6.9]{ADS2} for $\tau \leq \tau_\ast$ we get
\begin{equation}
\|\mathcal{E}[w_\cC]\|_{\mathcal{D}^*}\leq  C\varepsilon \|w_\cC\|_{\mathcal{D}}\, ,
\end{equation}
and
\begin{equation}\label{eq_error_commute}
\left\|\overline{\mathcal{E}}[w,\varphi_\cC(v_1)]\right\|_{\mathcal{D}^*}\leq C \varepsilon\left\|w 1_{\left\{\tfrac{5}{8}\theta\leq v_1\leq \tfrac{7}{8}\theta\right\}} \right\|_{\mathfrak{H}}\leq  C\varepsilon\left\|w 1_{D_\tau}  \right\|_{\mathfrak{H}}\, ,
\end{equation}
where $C=C(\theta)$ only depends on $\theta$. Thus, replacing $C\varepsilon$ by $\varepsilon$ completes the proof.
\end{proof}

We next bound the error terms $J$ and $K$ coming from the intrinsic cutoff.

\begin{lemma}[{estimate for $J$}]\label{lemma_int.cyl.J}
Given $\varepsilon>0$, there exist $\kappa>0$ and ${\tau}_\ast>-\infty$ such that for $\tau \leq \tau_*$ we have
\begin{align}
\|J(\tau)\|_{\mathcal{D}^*}\leq \varepsilon  \|w(\tau)\, 1_{D_\tau}\|_{\mathfrak{H}} \,  .
\end{align}
\end{lemma}

\begin{proof} Recall that
\begin{equation}\label{j_restated}
J=\,( v_{2,\tau}- v_{2,yy}+\tfrac{y}{2}v_{2,y}-\mathcal{E}[v_2])D[\varphi_\cC]-2v_{2,y}\partial_y D[\varphi_\cC]\, .
\end{equation}
By Proposition \ref{prop_evol_profile} (evolution equation of profile function) and the pointwise derivative bounds from \eqref{eq_cylinder.derivatives} we have
\begin{equation}
\left\|  (v_{2,\tau}+\tfrac{y}{2}v_{2,y}- v_{2,yy}-\mathcal{E}[v_2] ) \right\|_{W^{1,\infty}}  \leq C\, .
\end{equation}
By definition of the dual norm, as explained in \eqref{eq_product_rule_norm}, this yields 
\begin{equation}
\left\|  (v_{2,\tau}+\tfrac{y}{2}v_{2,y}- v_{2,yy}-\mathcal{E}[v_2] )D[\varphi_\cC] \right\|_{\cD^\ast}  \leq C \left\| D[\varphi_\cC] \right\|_{\cD^\ast}\, .
\end{equation}
By the weighted Sobolev inequality, as explained in \eqref{eq_weighted_sob_est}, we can estimate
\begin{equation}
\left\|   D[\varphi_\cC] \right\|_{\cD^\ast}  \leq C \left\|  \frac{1}{1+|y|} D[\varphi_\cC]\right\|_{\mathfrak{H}} \, . 
\end{equation}
Next, note that $D[\varphi_{\cC}](\tau)=0$ outside the support of $1_{D_\tau}$, while on the support we have
\begin{equation}
D[\varphi_{\mathcal{C}}](y,\tau) =\int_{v_1(y,\tau)}^{v_2(y,\tau)}\varphi_{\mathcal{C}}'(s)\, ds\, ,
\end{equation}
hence
\begin{equation}
|D[\varphi_{\mathcal{C}}]| \leq C |w| 1_{D_\tau}\, .
\end{equation}
This yields
\begin{equation}
\left\|   D[\varphi_\cC] \right\|_{\cD^\ast} \leq C \left\|  \frac{1}{1+|y|} w 1_{D_\tau} \right\|_{\mathfrak{H}}. 
\end{equation}
Since $|y| \geq |\tau|^{1/2}$ on the support of $1_{D_\tau}$  for sufficiently large $-\tau$ by Corollary  \ref{cor_unique_asympt} (uniform sharp asymptotics), this shows that
\begin{equation}
\left\|   D[\varphi_\cC] \right\|_{\cD^\ast} \leq \frac{C}{|\tau|^{1/2}} \left\|  w 1_{D_\tau}  \right\|_{\mathfrak{H}} \, .
\end{equation}
To bound the other term in \eqref{j_restated}, we observe that the derivative bound $|v_{2,y}|+|v_{2,yy}|\leq \eps$ implies that
\begin{equation}
\left\|  v_{2,y}\partial_y D[\varphi_\cC] \right\|_{\cD^\ast}  \leq C\eps \left\| \partial_y D[\varphi_\cC] \right\|_{\cD^\ast}  \, ,
\end{equation}
and compute
 \begin{equation}\label{eq_Dphi_y}
\partial_y D[\varphi_\cC] =\varphi_\cC'(v_1)w_y +v_{2,y} D[\varphi_\cC'] \, .
\end{equation}
Noting also that
\begin{equation}
|D[\varphi'_{\mathcal{C}}]|=\left|\int_{v_1}^{v_2}\varphi_{\mathcal{C}}''(s)\, ds\right| 1_{D_\tau}  \leq C |w| 1_{D_\tau}\, ,
\end{equation}
and
\begin{equation}
\varphi_\cC'(v_1)w_y=\varphi_\cC'(v_1)(w1_{D_\tau})_y\, ,
\end{equation}
arguing similarly as above we can thus estimate
\begin{equation}
\left\|  \partial_y D[\varphi_\cC] \right\|_{\cD^\ast} \leq C \left\|  (w 1_{D_\tau})_y \right\|_{\cD^\ast} + C \left\|  D[\varphi_\cC']  \right\|_{\cD^\ast} \leq C \left\|  w 1_{D_\tau}  \right\|_{\mathfrak{H}}\, ,
\end{equation}
where we also used \eqref{rec_basic_der_norm}. Putting things together the assertion follows.
\end{proof}

\begin{lemma}[{estimate for $K$}]\label{lemma_int.cyl.K}
Given $\varepsilon>0$, there exist $\kappa>0$ and ${\tau}_\ast>-\infty$ such that for $\tau \leq \tau_*$ we have
\begin{align}
\|K(\tau)\|_{\mathcal{D}^*}\leq \varepsilon  \|w(\tau)\, 1_{D_\tau}  \|_{\mathfrak{H}}+\|v_2\varphi_{\cC}'(v_1) e^\tau \mathcal{F}[w]\|_{\mathcal{D}^*} \, .
\end{align}
\end{lemma}

\begin{proof}
Recall that
\begin{equation}\label{k_restated}
K=\mathcal{E}[v_2]D[\varphi_\cC]-\mathcal{E}[v_2D[\varphi_\cC]]+v_2 (\partial_\tau -\mathfrak{L})D[\varphi_\cC]\, .
\end{equation}
First, using the expression for $\mathcal{E}$ from \eqref{eq_def.E[w]} we compute
\begin{align}\label{eq_multiline_K}
\mathcal{E}[v_2]D[\varphi_\cC]-\mathcal{E}[v_2D[\varphi_\cC]]=&\frac{v_{1,y}^2}{1+v_{1,y}^2}(v_2 \partial_y^2D[\varphi_\cC]+2v_{2,y}\partial_yD[\varphi_\cC])\nonumber\\
&+\frac{(v_{1,y}+v_{2,y})v_{2,yy}}{(1+v_{1,y}^2)(1+v_{2,y}^2)}v_2\partial_yD[\varphi_\cC].
\end{align}
Differentiating  \eqref{eq_Dphi_y}, we obtain
\begin{equation}
\partial_y^2 D[\varphi_\cC] =\varphi_\cC'(v_1)w_{yy}+\varphi''_\cC(v_1)(v_{1,y}+v_{2,y}) w_y +v_{2,yy}D[\varphi_\cC'] +v_{2,y}^2D[\varphi_\cC''] \, .
\end{equation}
Combining the above equations, and remembering also \eqref{eq_def.E[w]} and \eqref{eq_Dphi_y}, we infer that
\begin{align}\label{J1}
\mathcal{E}[v_2]D[\varphi_\cC]-\mathcal{E}[v_2D[\varphi_\cC]]=
-\varphi_\cC'(v_1)v_2 \mathcal{E}[w]+a w_y + bw + cD[\varphi_\cC']+dD[\varphi_\cC'']\, ,
\end{align}
where
\begin{align}
a&=\frac{v_{1,y}^2}{1+v_{1,y}^2}\left(\varphi''_\cC(v_1) v_2 (v_{1,y}+v_{2,y})+2 \varphi'_\cC(v_1) v_{2,y} \right)\ ,\\
b&=\frac{2-v_1v_2}{2v_1v_2}\varphi'_\cC(v_1)v_2 \, , \\
c&=\frac{v_{1,y}^2}{1+v_{1,y}^2} (v_2v_{2,yy}+2v_{2,y}^2)  + \frac{(v_{1,y}+v_{2,y})v_{2,yy}}{(1+v_{1,y}^2)(1+v_{2,y}^2)} v_2 v_{2,y}\ ,\\
d&=\frac{v_{1,y}^2}{1+v_{1,y}^2} v_2 v_{2,y}^2 \, .
\end{align}
Now, using the pointwise derivative bounds from \eqref{eq_cylinder.derivatives} and arguing similarly as in the proof of the previous lemma we see that
\begin{align}
\left\| a w_y + cD[\varphi_\cC']+dD[\varphi_\cC'']\right\|_{\mathcal{D}^*} &\leq C\eps \left(\left\|  (w1_{D_\tau})_y \right\|_{\mathcal{D}^*} + \left\| D[\varphi_\cC']\right\|_{\mathcal{D}^*} + \left\| D[\varphi_\cC'']\right\|_{\mathcal{D}^*}\right)\nonumber\\
&\leq C\eps \left\|  w 1_{D_\tau}  \right\|_{\mathfrak{H}}\, .
\end{align}
Using also the weighted Sobolev inequality we can estimate
\begin{equation}
\left\| bw\right\|_{\mathcal{D}^*} \leq C \left\|   \frac{1}{1+|y|}w 1_{D_\tau}\right\|_{\fH} \leq \frac{C}{|\tau|^{1/2}}\left\| w 1_{D_\tau} \right\|_{\fH}\, .
\end{equation}
We remark that we do not have to estimate the term $\varphi_\cC'(v_1)v_2 \mathcal{E}[w]$ as it will cancel out later.
Next, to estimate the other term in \eqref{k_restated}, using Proposition \ref{prop_evol_w} (evolution of difference) we compute
\begin{multline}\label{J2}
(\partial_\tau -\mathfrak{L})D[\varphi_\cC]=\,\varphi'_\cC(v_1)\left(w+\mathcal{E}[w]+e^\tau \mathcal{F}[w] \right)-\varphi_\cC''(v_1)(v_{1,y}+v_{2,y})w_y\\
-D[\varphi_\cC]+\left(  v_{2,\tau}- v_{2,yy}+\tfrac{y}{2} v_{2,y}\right)D[\varphi_\cC']-v_{2,y}^2D[\varphi_\cC'']\, .
\end{multline}
Arguing as above, we see that
\begin{multline}
\left\| v_2 \varphi_\cC'(v_1)w \right\|_{\mathcal{D}^*} +\left\| v_2 \varphi_\cC''(v_1)(v_{1,y}+v_{2,y})w_y \right\|_{\mathcal{D}^*}  + \left\| v_2 D[\varphi_\cC] \right\|_{\mathcal{D}^*}\\
+ \left\| v_2\left(  v_{2,\tau}- v_{2,yy}+\tfrac{y}{2} v_{2,y}\right)D[\varphi_\cC']\right\|_{\mathcal{D}^*}+\left\| v_2 v_{2,y}^2D[\varphi_\cC'']\right\|_{\mathcal{D}^*}\leq C\eps  \left\|  w 1_{D_\tau}  \right\|_{\mathfrak{H}}\, .
\end{multline}
Finally, we observe that when we multiply equation \eqref{J2} by $v_2$, and add the result to equation \eqref{J1} then the term $\varphi_\cC'(v_1)v_2 \mathcal{E}[w]$ cancels out. Putting things together the assertion follows.
\end{proof}

Finally, we estimate the nonlinear error term:

\begin{lemma}[estimate for nonlinear error]\label{lemma_int.cyl.F}
There exist $\kappa>0$ and ${\tau}_\ast>-\infty$ such that for $\tau \leq \tau_*$ we have
\begin{align}
\|v_2\varphi_{\cC}'(v_1) e^\tau \mathcal{F}[w]\|_{\mathcal{D}^*,\infty}(\tau)+ \|\varphi_\cC(v_1)e^{\tau}\mathcal{F}[w]\|_{\mathcal{D}^*,\infty}(\tau) \leq \frac{C}{|\tau|} \| w\|_{C^2_{\exp}(\mathcal{C})}(\tau).
\end{align}
\end{lemma}

\begin{proof}
Observe that
\begin{equation}
\|v_2\varphi_{\cC}'(v_1) e^\tau \mathcal{F}[w]\|_{\mathcal{D}^*}+ \|\varphi_\cC(v_1)e^{\tau}\mathcal{F}[w]\|_{\mathcal{D}^*} \leq C e^{\tau} \|\mathcal{F}[w]\|_{\mathfrak{H}}\, .
\end{equation}
Inspecting the expression for $\mathcal{F}$ from Proposition \ref{prop_evol_w} (evolution of difference) and using the pointwise derivative bounds from  \eqref{eq_cylinder.derivatives}  we can estimate 
\begin{equation}
\|\mathcal{F}[w]\|_{\mathfrak{H}}\leq   C \sup_{y\in C_\tau}\left(|w|+|w_y|+|w_\tau| + |w_{yy}| +|w_{y\tau}|+|w_{\tau\tau}| \right)(y,\tau)\, .
\end{equation}
and the assertion follows.
\end{proof}

Combining the above results we can now conclude the proof of the energy estimate:

\begin{proof}[Proof of Proposition \ref{prop_int_cyl}]
Recall that the general energy inequality \eqref{gen_energy_ineq} tells us that
 \begin{equation}
\|w_\cC-\fp_0w_\cC\|_{\mathcal{D},\infty}\leq C\|(\partial_\tau -\mathfrak{L})w_\cC\|_{\mathcal{D}^*,\infty}.
\end{equation}
Combining the above lemmas, we can estimate
\begin{equation}
\|(\partial_\tau -\mathfrak{L})w_\cC\|_{\mathcal{D}^*,\infty}\leq  C\varepsilon \left(\|w_\cC \|_{\mathcal{D,\infty}}+\|w\, 1_{D}  \|_{\mathfrak{H},\infty}\right)+\frac{C}{|\tau_0|}\|w\|_{C^2_{\exp}(\cC)} \, ,
\end{equation}
where $D=\bigcup_{\tau\leq \tau_0} D_\tau \times \{\tau\}$. Replacing $C^2\varepsilon$ by $\varepsilon$ this gives 
\begin{equation}
\|w_\cC-\fp_0w_\cC\|_{\mathcal{D},\infty}\leq  \varepsilon\left( \|w_{\cC}\|_{\mathcal{D},\infty}+\|w\, 1_{D} \|_{\mathfrak{H},\infty}\right)+\frac{C}{|\tau_0|}\|w\|_{C^2_{\exp}(\cC)}.
\end{equation}
Finally, by Corollary  \ref{cor_unique_asympt} (uniform sharp asymptotics) we have
\begin{equation}
1_{D}\leq 1_{\{\theta/2\leq v_1\leq\theta\}}.
\end{equation}
This concludes the proof of the proposition.
\end{proof}

\bigskip

\subsection{Energy estimates in the tip region}  In this subsection, we generalize the arguments from \cite[Section 7]{ADS2} to our setting to establish the following energy estimate in the tip region:

\begin{proposition}[energy estimate in tip region]\label{prop_int_tip}
There exist  $\kappa>0$ and ${\tau}_{\ast}>-\infty$ with the following significance.  If $M_1$ and $M_2$ are $\kappa$-quadratic at time $\tau_0 \leq {\tau}_{\ast}$,  then for $\tau\leq \tau_0$ we have
\begin{equation}
\|W_{\mathcal{T}}\|_{2,\infty}(\tau)\leq \frac{C}{|\tau|}\left(\|W 1_{[\theta,2\theta]}\|_{2,\infty}(\tau)+\|W\|_{C^2_{\exp}(\cT)}(\tau)\right).
\end{equation}
\end{proposition}

For ease of notation, we will always assume that our tip satisfies $y>0$ (considering the map $y\mapsto -y$ this implies the estimate for the second tip), and will simply write $Y=Y_1$. In this notation, we have
\begin{equation}
W=Y-Y_2\, .
\end{equation}
We recall that $W_{\cT}=\varphi_\cT W$, where $\varphi_\cT$ is the cutoff function from  \eqref{cutoff_tip_region} that localizes in the tip region $\cT = \{v\leq 2\theta\}$. The norms appearing in the energy estimate have been briefly mentioned in the introduction. Let us define them in detail now. To this end, fix a smooth function $\zeta(v)$  satisfying $0\leq \zeta'\leq 5\theta^{-1}$ and 
\begin{align}
&\zeta(v)=0 \qquad \text{for}\quad v \leq \theta/4, && \zeta(v)=1 \qquad \text{for}\quad v \geq \theta/2.
\end{align}
 Then, similarly as in  \cite{ADS2} we consider the weight function
\begin{equation}\label{def_weight_fn}
\mu(v,\tau):=-\frac14 Y^2(\theta,\tau)+\int^\theta_v \left[ \zeta(\tilde{v}) \left(\frac{Y^2}{4}\right)_{\tilde{v}}-(1-\zeta(\tilde{v}))\frac{1+Y_{\tilde{v}}^2}{\tilde{v}} \right]\, d\tilde{v}\, .
\end{equation}
Notice that $\mu(v,\tau)=-\frac{1}{4}Y^2(v,\tau)$ for $v \geq \theta/2$. Set
\begin{equation}
\|F(\cdot,\tau)\|_{2}= \left[ \int^{2\theta}_0F^2(v,\tau)\, e^{\mu(v,\tau)}dv\right]^{1/2},
\end{equation}
and
\begin{equation}
\|F\|_{2,\infty}(\tau)=\sup_{\tau' \leq \tau}\frac{1}{|\tau'|^{1/4}}\left[\int^{\tau'}_{\tau'-1}\int^{2\theta}_0 F^2(v,\sigma)e^{\mu(v,\sigma)}\, dvd\sigma\right]^{1/2}\, .
\end{equation}
As before, we simply write
 \begin{equation}
\|F\|_{2,\infty}= \|F\|_{2,\infty}(\tau_0).
 \end{equation}
Finally, to capture the higher derivative error terms we consider the exponentially weighted $C^2$-norm
\begin{align}
\| F\|_{C^2_{\exp}(\mathcal{T})}(\tau):=\sup_{\tau'\leq \tau} \left( e^{\tau'}\sup_{v\leq 2\theta } \big(|F|+|F_v|+|F_{\tau}|+ |F_{vv}|+|F_{v\tau}|+|F_{\tau\tau}|\big)(v,\tau')\right),
\end{align}
and we often simply write
\begin{equation}
\| F\|_{C^2_{\exp}(\mathcal{T})}:=\| f\|_{C^2_{\exp}(\mathcal{T})}(\tau_0)\, .
\end{equation}

\bigskip

To prove Proposition \ref{prop_int_tip} (energy estimate in the tip region), we first establish certain a priori estimates for $Y$ and the weight $\mu$. The statements of these a priori estimates are similar to the ones in \cite[Section 7.1]{ADS2}, but the proofs are a bit more involved. Once these a priori estimates are established, similarly as in \cite[Section 7.2]{ADS2}, we will obtain a weighted Poincare inequality. Finally, using the weighted Poincare inequality we will implement the energy method by generalizing \cite[Section 7.3]{ADS2}.

\bigskip

To control some new error terms we need the following rough estimate:

\begin{lemma}[rough tip estimate]\label{prop_Y.collar.rough}
There exist $\kappa>0$, ${\tau}_{\ast}>-\infty$ and $C<\infty$ such that
\begin{align}
|Y|+|Y_v|+|Y_\tau|+|Y_{vv}|+|Y_{v\tau}|+|Y_{\tau\tau}|\leq C|\tau|^{\frac{5}{2}}
\end{align}
holds for $\tau \leq \tau_{\ast}$ and $v\leq 2\theta$.
\end{lemma}

\begin{proof}
By the tip region asymptotics from Corollary \ref{cor_unique_asympt} (uniform sharp asymptotics) the estimate clearly holds in the soliton region $\mathcal{S}=\{v\leq L/\sqrt{|\tau|}\}$. Thus, our task is to establish the estimate in the collar region $\mathcal{K}=\{ L/\sqrt{|\tau|} \leq v \leq 2 \theta\}$. To this end, note first that by convexity of our translator we have
\begin{equation}
\sup_{v\leq 2\theta}|Y_v(v,\tau)| \leq |Y_v(2\theta,\tau)|.
\end{equation}
Hence, together with Corollary \ref{cor_unique_asympt} (uniform sharp asymptotics) we infer that
\begin{equation}
|Y|+|Y_v|\leq C |\tau|^{1/2}\, .
\end{equation}
Next, recall from the proof of Proposition \ref{prop_evol_profile} (evolution equation for inverse profile function) that
\begin{align}
&Y_\tau=- v_\tau Y_v , && Y_{vv} = - v_{yy} Y_v^3
\end{align}
and
\begin{align}
&Y_{\tau v} = -v_yv_\tau Y_v Y_{vv} -  v_{\tau y}Y_v^2,&& Y_{\tau\tau}=-2v_\tau Y_{\tau v} -v_\tau^2 Y_{vv}-v_{\tau\tau}Y_v \, .
\end{align}
Together with Lemma \ref{lemma_derivatives} (derivative estimates) this implies the desired result.
\end{proof}

We also need the following standard cylindrical estimate:

\begin{lemma}[cylindrical estimate]\label{lemma_cyl_est} Given any $\eta>0$, for $L$ large enough and $\tau_\ast$ negative enough in the collar region we have
\begin{equation}
\frac{|Y_{vv}|}{1+Y_v^2}<\eta \left|\frac{Y_v}{v}\right|\, .
\end{equation}
\end{lemma}

\begin{proof}Observe that
\begin{align}
\frac{|vY_{vv}|}{|Y_v|(1+Y_v^2)}=\frac{|v v_{yy}|}{1+v_y^2}=\frac{\lambda_1}{\lambda_2},
\end{align}
where $\lambda_1$ and $\lambda_2$ are the principal curvatures of the level set. Since $\lambda_1/\lambda_2=0$ on $\mathbb{R}\times S^1$, arguing as in the proof of Lemma \ref{lemma_derivatives} (derivative estimates) the assertion follows.
\end{proof}

After these preparations we can now establish the tip estimates:

\begin{proposition}[{tip estimates, c.f. \cite[Lemma 7.4]{ADS2}}]\label{lemma_tip_estimates}
Given $\eta>0$, there exist $\theta>0$, $\kappa>0$ and ${\tau}_{\ast}>-\infty$  such that if $M$ is $\kappa$-quadratic at time $\tau_0\leq \tau_\ast$, then for $v\leq 2\theta$ and $\tau\leq \tau_0$ we have
\begin{align}\label{est_toshow_Y}
&\frac{1}{4} \sqrt{|\tau|}\leq \left|\frac{Y_v}{v}\right|\leq \sqrt{|\tau|}\, , && |Y_\tau| \leq \eta \left|\frac{Y_v}{v}\right|\, .
\end{align}
\end{proposition}

\begin{proof} We will argue similarly as in \cite{ADS2}, but we will encounter some new error terms coming from the fact that our profile function $v$ is only almost quadratically concave and from the nonlinearity.\\
Since by the tip region asymptotics from Corollary \ref{cor_unique_asympt} (uniform sharp asymptotics) the zoomed in profile function $Z$ is arbitrarily close to the profile function $Z_0$ of the 2d-bowl with speed $1/\sqrt{2}$ in the soliton region $\mathcal{S}$ we get
\begin{align}
&\frac{1-\eps}{2\sqrt{2}} \sqrt{|\tau|}\leq \left|\frac{Y_v}{v}\right|\leq \sqrt{|\tau|}\, , && |Y_\tau| \leq \eta \left|\frac{Y_v}{v}\right|\, .
\end{align}
Next, note that by the intermediate region asymptotics from Corollary \ref{cor_unique_asympt} (uniform sharp asymptotics) we have
\begin{equation}\label{first_tech_est}
\left|\frac{Y_v}{v}\right| (2\theta,\tau) \leq\frac{9}{10} \sqrt{|\tau|}\, .
\end{equation}
Together with the fact that the function $v\mapsto |Y_v|/v$ is almost monotone in the collar region $\mathcal{K}$ by Proposition \ref{prop_almost_quad_conc} (almost quadratic concavity), we infer that 
\begin{equation}\label{claimedfirstes}
\frac{1}{4} \sqrt{|\tau|}\leq \left|\frac{Y_v}{v}\right|\leq \sqrt{|\tau|}
\end{equation}
holds in the whole tip region $\mathcal{T}$. Indeed, using the results that we just cited, in the collar region we get
 \begin{equation}\label{monotonicity.Y_v/v_pre}
\left(\frac{|Y_v|}{v}\right)_v=\frac{-vY_{vv}+Y_v}{v^2}=-\frac{(v^2)_{yy}}{2v^2|v_y|^3} \geq  \frac{e^{\tau}}{v^4|v_y^3|}\geq -C|\tau|^5 e^{\tau}, 
 \end{equation}
which, possibly after decreasing $\tau_\ast$, yields \eqref{claimedfirstes}.

Finally, to check that $|Y_\tau|\leq \eta |Y_v/v|$ holds in the collar region $\mathcal{K}$ as well, we rewrite the evolution equation from Proposition \ref{prop_inverse.evol} (evolution equation for inverse profile function) in the form
\begin{equation}
     Y_{\tau}=\frac{Y_{vv}}{1+Y^2_{v}}+\frac{Y_{v}}{v}\left(1+\frac{vY}{2Y_{v}}-\frac{v^2}{2}\right)+e^{\tau}\mathcal{M}[Y].
\end{equation}
By Lemma \ref{lemma_cyl_est} (cylindrical estimate) choosing $L$ large enough we can ensure that 
\begin{equation}
\left|\frac{Y_{vv}}{1+Y^2_{v}}\right| \leq \frac{\eta}{4}\left|\frac{Y_v}{v}\right|.
\end{equation}
By Corollary \ref{cor_gaussian_collar} (almost Gaussian collar), for $\theta$ small enough we get
\begin{equation}
\left| 1+\frac{vY}{2Y_{v}}\right| \leq \frac{\eta}{4}.
\end{equation}
Since $v\leq 2\theta$ in the tip region, possibly after decreasing $\theta$ we have
\begin{equation}
\frac{v^2}{2} \leq \frac{\eta}{4}.
\end{equation}
And using Lemma \ref{prop_Y.collar.rough} (rough tip estimates) for $\tau\leq\tau_\ast$ we get
\begin{equation}
\left|\frac{v}{Y_v}\right||e^{\tau}\mathcal{M}[Y]|\leq \frac{\eta}{4}.
\end{equation}
Combining the above inequalities yields
\begin{equation}
|Y_\tau|\leq \eta \left|\frac{Y_v}{v}\right|,
\end{equation}
which concludes the proof.
\end{proof}

Using the above, we can now establish the following estimates for the weight function:

\begin{proposition}[{weight estimates, c.f. \cite[Lemma 7.5]{ADS2}}]\label{lemma_weight_estimates}
Given $\eta>0$, there exist $\theta>0$, $\kappa>0$ and ${\tau}_{\ast}>-\infty$  such that if $M$ is $\kappa$-quadratic at time $\tau_0\leq \tau_\ast$, then for $v\leq 2\theta$ and $\tau\leq \tau_0$ we have
\begin{align}
&1-\eta \leq \frac{v\mu_v}{1+Y_v^2}\leq 1+\eta\, , && \mu_\tau \leq \eta|\tau|\, .
\end{align}
\end{proposition}

\begin{proof} We will argue similarly as in \cite{ADS2}, but we will encounter some new error terms coming from the fact that our profile function $v$ is only almost quadratically concave and from the nonlinearity.\\
By definition of the weight function $\mu$ we have
\begin{equation}
v\mu_v  =\zeta(v)\left(\frac{-v Y Y_v}{2}\right) +\left(1-\zeta(v)\right)(1+Y^2_{v}).
\end{equation}
Hence, to prove the first estimate it suffices to show that for $\theta/4\leq v \leq 2\theta$ we have
\begin{equation}
\left|\frac{v Y Y_v}{2(1+Y^2_{v})}+1\right|\leq \eta.
\end{equation}
But this easily follows from Corollary \ref{cor_gaussian_collar} (almost Gaussian collar) since in the region under consideration we have $|Y_v|\gg 1$.\\
To prove the second estimate, using the results that we already established and arguing similarly as in \cite[proof of Lemma 7.5]{ADS2}, we can obtain the estimates (7.12), (7.14) and (7.15) therein, hence
\begin{equation}\label{mu-speed_pre}
\mu_\tau \leq  c\eta  |\tau|+2  \int_v^\theta (1-\zeta)\left( \frac{Y_{v'}}{v'}\right)_{v'}  Y_\tau   \, dv' \, .
\end{equation}
Using Proposition \ref{lemma_tip_estimates} (tip estimates) we can estimate
\begin{equation}\label{mu-speed}
 \int_v^\theta (1-\zeta)\left( \frac{Y_{v'}}{v'}\right)_{v'}  Y_\tau   \, dv' \leq \eta \sqrt{|\tau|}  \int_v^\theta \left|\left( \frac{Y_{v'}}{v'}\right)_{v'}\right| \, dv' \, .
\end{equation}
Thus, our task is to bound the latter integral by a multiple of $\sqrt{|\tau|}$. Using the almost positivity from \eqref{monotonicity.Y_v/v_pre} and Proposition \ref{lemma_tip_estimates} (tip estimates) we can estimate the collar region contribution to this integral via
\begin{align}
\int_{\frac{L}{\sqrt{\tau}}}^{\theta}\left|\left( \frac{Y_{\tilde{v}}}{v'}\right)_{v'}\right| \, dv' \leq \int_{\frac{L}{\sqrt{\tau}}}^{\theta} \left( \frac{Y_{v'}}{v'}\right)_{v'} \, dv' + C|\tau|^5 e^{\tau}\leq  \left|\frac{Y_v}{v}\right|(\theta,\tau)+C|\tau|^5e^\tau \leq C \sqrt{|\tau|}\, .
\end{align}
To deal with the soliton region, recall first from \eqref{zoomed_in_5} that
\begin{equation}
Y(v,\tau)=Y(0,\tau)+|\tau|^{-1/2}Z(|\tau|^{1/2}v,\tau)\, .
\end{equation}
Using this, we compute
\begin{equation}
\frac{Y_{vv}(v,\tau)}{v}-\frac{Y_v(v,\tau)}{v^2}=|\tau|\left( \frac{Z_{\rho\rho}(\rho,\tau)}{\rho} - \frac{Z_{\rho}(\rho,\tau)}{\rho^2} \right),\qquad \textrm{where} \quad \rho=|\tau|^{1/2}v\, .
\end{equation}
By Corollary \ref{cor_unique_asympt} (uniform sharp asymptotics) the function $Z(\rho,\tau)$ is $\eps$-close in $C^{100}(B(0,\eps^{-1}))$ to $Z_0(\rho)$, the profile function of the 2d-bowl with speed $1/\sqrt{2}$. Thus, given any $\rho_0>0$, for $\tau\leq \tau_\ast$ we get
\begin{equation}
\sup_{\rho_0\leq \rho\leq L} \left| \frac{Z_{\rho\rho}(\rho,\tau)}{\rho} - \frac{Z_{\rho}(\rho,\tau)}{\rho^2} \right| \leq C(\rho_0)\, .
\end{equation}
On the other hand, since the profile function of any surface of rotation satisfies
\begin{equation}
Z_{\rho}(0,\tau)=0,
\end{equation}
for $\rho\leq \rho_0(\tau)$ we can expand
\begin{equation}
Z(\rho,\tau)= a_0(\tau)+a_2(\tau)\rho^2 + R(\rho,\tau),
\end{equation}
with the estimate
\begin{equation}
\sup_{\rho\leq\rho_0(\tau)}\left( \frac{|R(\rho,\tau)|}{\rho^3}  + \frac{|R_{\rho}(\rho,\tau)|}{\rho^2}+\frac{|R_{\rho\rho}(\rho,\tau)|}{\rho} \right)\leq C(\tau)\, ,
\end{equation}
where $\rho_0(\tau)>0$ and $C(\tau)<\infty$ are constants that might initially depend on $\tau$.
This yields
\begin{equation}
\sup_{\rho\leq\rho_0(\tau)}\left| \frac{Z_{\rho\rho}(\rho,\tau)}{\rho} - \frac{Z_{\rho}(\rho,\tau)}{\rho^2} \right| =\sup_{\rho\leq\rho_0(\tau)} \left| \frac{R_{\rho\rho}(\rho,\tau)}{\rho} - \frac{R_{\rho}(\rho,\tau)}{\rho^2} \right| \leq C(\tau)\, .
\end{equation}
Now using again Corollary \ref{cor_unique_asympt} (uniform sharp asymptotics) we see that for $\tau\leq \tau_\ast$ this estimate in fact holds with uniform constants $\rho_0(\tau)=\rho_0>0$ and $C(\tau)=C<\infty$. Hence,
\begin{equation}
\sup_{ \rho\leq L} \left| \frac{Z_{\rho\rho}(\rho,\tau)}{\rho} - \frac{Z_{\rho}(\rho,\tau)}{\rho^2} \right| \leq C\, .
\end{equation}
We have thus shown that
\begin{equation}
\left|\left(\frac{Y_{v'}}{v'} \right)_{v'} \right| \leq C|\tau|
\end{equation} 
for $v'\leq L/\sqrt{\tau}$ and $\tau\leq \tau_\ast$. Integrating gives
\begin{equation}
\int_0^{\frac{L}{\sqrt{\tau}}}\left|\left( \frac{Y_{v'}}{v'}\right)_{v'}\right| \, dv' \leq  C \sqrt{|\tau|}\, .
\end{equation}
Putting things together, and adjusting $\eta$, this concludes the proof of the proposition.
\end{proof}

\begin{corollary} [{weighted Poincare inequality, c.f. \cite[Proposition 7.6]{ADS2}}]\label{prop_poincare}
There are constants $C_0<\infty$, $\kappa>0$ and $\tau_\ast\ll 0$ with the following significance. If $M$ is $\kappa$-quadratic at time $\tau_0\leq\tau_\ast$, then for  any $\theta\ll 1$ and all $\tau\leq \tau_0$, we have
\begin{equation}\label{poincare}
  \int^{2\theta}_{0}F^2(v) e^{\mu(v, \tau)}dv \leq \frac{C_0}{|\tau|}\int^{2\theta}_{0}\frac{F^2_{v}}{1+Y^2_{v}}e^{\mu(v, \tau)}dv
\end{equation}
for all smooth functions $F$ satisfying $F'(0)=0$ and $\mathrm{spt}(F)\subseteq [0,2\theta]$.
\end{corollary}
\begin{proof}Having established Proposition \ref{lemma_tip_estimates} (tip estimates) and Proposition \ref{lemma_weight_estimates} (weight estimates), the argument from \cite[proof of Proposition 7.6]{ADS2} goes through.
\end{proof}

\bigskip

Having established the weighted Poincare inequality, we can now implement the energy method. Recall from Proposition \ref{prop_evol_w_inv} (evolution of inverse difference) that
the function $W=Y-Y_2$ satisfies
\begin{equation}\label{W.equation_rest}
W_\tau=\frac{W_{vv}}{1+Y_{v}^2}+\left(\frac{1}{v}-\frac{v}{2}+D\right)W_v+\frac{1}{2}W+e^\tau \mathcal{F}[W]\, ,
\end{equation}
where
\begin{equation}
D=-\frac{Y_{2,vv}(Y_{v}+Y_{2,v})}{(1+Y_{v}^2)(1+Y_{2,v}^2)}\, .
\end{equation}

\begin{lemma}[energy inequality]\label{lemma_energy_tip_with_errors}
There exist $\theta>0$, $\kappa>0$ and $\tau_\ast>-\infty$ with the following significance. If $M^1$ and $M^2$ are $\kappa$-quadratic at time $\tau_0\leq \tau_\ast$, then for $\tau\leq\tau_0$
we have
\begin{align}
\frac{d}{d\tau}\int W_\cT^2e^\mu dv\leq &-\frac{1}{2}\int \frac{|\partial_v W_\cT|^2}{1+Y_{v}^2} e^\mu dv +\int \bar G W_\cT^2 e^\mu dv+\frac{C(\theta)}{|\tau|}\int^{2\theta}_\theta W^2 e^\mu dv\nonumber\\
&+e^\tau \sup_{v\leq 2\theta } \big(|W|+|W_v|+|W_{\tau}|+ |W_{vv}|+|W_{v\tau}|+|W_{\tau\tau}|\big) \left(\int W_\cT^2e^\mu dv\right)^{1/2}\, ,
\end{align}
where
\begin{equation}
\bar G=(1+Y_v^2)G^2   +1 +2\mu_\tau,\quad \textrm{with}\quad G= \frac{1}{v}-\frac{v}{2}-\frac{\mu_v}{1+Y_{v}^2}+\frac{2Y_{v}Y_{vv}}{(1+Y_{v}^2)^2}+D\, .
\end{equation}
\end{lemma}

\begin{proof}
Using \eqref{W.equation_rest} and integration by parts we compute
\begin{align}
\frac{1}{2}\frac{d}{d\tau}\int W_\cT^2e^\mu dv=&-\int \frac{\varphi^2W_v^2}{1+Y_{v}^2} e^\mu dv +\int G\varphi^2 WW_ve^\mu dv
-2\int \frac{\varphi\varphi_vWW_v}{1+Y_{v}^2}e^\mu dv\nonumber\\
&+\int \left(\tfrac12 +\mu_\tau\right)W_\cT^2 e^\mu dv+e^\tau\int  \mathcal{F}\varphi W_\cT e^\mu dv, 
\end{align}
where for simplicity we write $\varphi=\varphi_\cT$. Using $ab\leq \tfrac{1}{2}a^2+\tfrac{1}{2}b^2$ we can estimate the second term by
\begin{equation}
\int G \varphi^2 WW_ve^\mu dv \leq \frac{1}{2}\int \frac{\varphi^2W_v^2}{1+Y_{v}^2} e^\mu dv+\frac{1}{2}\int G^2(1+Y_v^2) W_\cT^2e^\mu dv\, .
\end{equation}
Via absorption, and observing also that $\partial_v W_\cT=\varphi W_v+\varphi_v W$ implies the pointwise identity
\begin{equation}
\varphi^2 W_v^2 = (\partial_v W_{\cT})^2 - \varphi_v^2W^2- 2 \varphi\varphi_v WW_v  \, ,
\end{equation}
this yields
\begin{align}
\frac{1}{2}\frac{d}{d\tau}\int W_\cT^2e^\mu dv\leq &-\frac{1}{2}\int \frac{|\partial_v W_\cT|^2}{1+Y_{v}^2} e^\mu dv+\frac{1}{2}\int \frac{\varphi_v^2 W^2}{1+Y_{v}^2} e^\mu dv -\int \frac{\varphi\varphi_v WW_v}{1+Y_{v}^2}e^\mu dv\nonumber\\
&+\int \left(\tfrac12 G^2(1+Y_v^2)+\tfrac12 +\mu_\tau\right) W_{\cT}^2 e^\mu dv+e^\tau\int  \mathcal{F}\varphi W_\cT e^\mu dv\, .
\end{align}
We can then use
\begin{equation}
-\varphi\varphi_vWW_v=-\varphi_vW(\partial_v W_\cT)+\varphi_v^2W^2\leq \frac{1}{4}(\partial_v W_\cT)^2+2\varphi_v^2 W^2,
\end{equation}
to absorb the third term into the first two terms. This yields
\begin{align}
\frac{d}{d\tau}\int W_\cT^2e^\mu dv\leq -\frac{1}{2}\int \frac{|\partial_v W_\cT|^2}{1+Y_{v}^2} e^\mu dv+\int \bar{G}W_{\cT}^2 e^\mu dv+5\int \frac{W^2}{1+Y_{v}^2}\varphi_v^2 e^\mu dv +2e^\tau\int  \mathcal{F}\varphi W_\cT e^\mu dv\, .
\end{align}
Now, since $\textrm{spt}(\varphi_v)\subset [\theta,2\theta]$ and since by Proposition \ref{lemma_tip_estimates} (tip estimates), for $\theta$ and $\kappa$ sufficiently small and $\tau_\ast$ sufficiently negative, we have    
\begin{equation}
\sup_{\theta\leq v\leq 2\theta}\frac{1}{1+Y_{v}^2(v,\tau)} \leq \frac{16}{\theta^2|\tau|}\, ,
\end{equation}
we can estimate the third time by
\begin{align}
5\int \frac{W^2}{1+Y_{v}^2}\varphi_v^2 e^\mu dv\leq \frac{C(\theta)}{|\tau|} \int_{\theta}^{2\theta} W^2 e^\mu dv\, .
\end{align}
Finally, using the Cauchy-Schwarz inequality we can estimate the last term by
\begin{equation}
2e^\tau\int  \mathcal{F}\varphi W_\cT e^\mu dv \leq 2e^{\tau} \left( \int \mathcal{F}^2\varphi^2 e^{\mu}\right)^{1/2} \left(\int W_\cT^2e^\mu dv\right)^{1/2}\ .
\end{equation}
Note that by Corollary \ref{cor_unique_asympt} (uniform sharp asymptotics) for $v\leq 2\theta$ we have the rough estimate
\begin{equation}\label{rough_mu_decay}
e^{\mu(v,\tau)} \leq e^{-\frac{1}{4}\tau}\, .
\end{equation}
Hence, remembering the structure of $\mathcal{F}$ and applying Lemma \ref{prop_Y.collar.rough} (rough tip estimate) we get
\begin{equation}
2\left( \int \mathcal{F}^2\varphi^2 e^{\mu}\right)^{1/2} \leq \sup_{v\leq 2\theta}\big(|W|+|W_v|+|W_{\tau}|+ |W_{vv}|+|W_{v\tau}|+|W_{\tau\tau}|\big)\, .
\end{equation}
Putting everything together, this proves the lemma.
\end{proof}
  
We can now prove the main result of this subsection:

\begin{proof}[Proof of Proposition \ref{prop_int_tip}]
Recall that by Lemma \ref{lemma_energy_tip_with_errors} (energy inequality) we have
\begin{align}
\frac{d}{d\tau}\|W_\cT\|_2^2\leq -\frac{1}{2}\int \frac{|\partial_v W_\cT|^2}{1+Y_{v}^2} e^\mu dv +\int \bar G W_\cT^2 e^\mu dv+\frac{C}{|\tau|}\|W1_{[\theta,2\theta]}\|_2^2 
+e^\tau \|W \|_{C^2|T_\tau} \|W_\cT\|_2\, ,
\end{align}
where
\begin{equation}
\|W \|_{C^2|T_\tau}:=\sup_{v\leq 2\theta } \big(|W|+|W_v|+|W_{\tau}|+ |W_{vv}|+|W_{v\tau}|+|W_{\tau\tau}|\big)\, .
\end{equation}
Now for $\kappa$ and $\theta$ sufficiently small, $L$ sufficiently large, and $\tau_\ast$ sufficiently negative, similarly as in \cite[Claim 7.7]{ADS2} we can estimate
\begin{equation}\label{prop_bar_G_est}
\bar G \leq \eta|\tau|\, .
\end{equation}
Moreover, possibly after adjusting the constants, by Proposition \ref{prop_poincare} (weighted Poincare inequality) we have
\begin{equation}
\|W_\cT\|_2^2 \leq  \frac{C_0}{|\tau|}\int \frac{|\partial_v W_\cT|^2}{1+Y_{v}^2}e^\mu dv\, .
\end{equation}
Combining the above facts and taking $\eta=\tfrac{1}{4C_0}$ we infer that
\begin{equation}
\frac{d}{d\tau}\|W_\cT\|_2^2\leq -\eta |\tau|\|W_\cT\|_2^2 +\frac{C}{|\tau|}\|W1_{[\theta,2\theta]}\|_2^2 +e^\tau  \|W \|_{C^2|T_\tau}    \|W_\cT\|_2\, .
\end{equation}
Using the Peter-Paul inequality the last term can be estimated by
\begin{equation}
e^\tau  \|W \|_{C^2|T_\tau}   \|W_\cT\|_2 \leq \frac{\eta}{2}|\tau| \|W_\cT\|_2^2 +  \frac{1}{2\eta|\tau|}e^{2\tau}  \|W \|_{C^2|T_\tau}^2 \, .
\end{equation}
Hence, setting $c=\eta/2$ we obtain
\begin{equation}
\frac{d}{d\tau}\|W_\cT\|_2^2\leq -c|\tau|\|W_\cT\|_2^2 +C|\tau|^{-1}\|W1_{[\theta,2\theta]}\|_2^2+ C|\tau|^{-1} \|W \|_{C^2_{\exp}(\cT)}^2(\tau)\, 
\end{equation}
for all $\tau\leq \tau_0$.

\bigskip

To analyze this differential inequality, similarly as in \cite{ADS2},  we define
\begin{align}
&f(\tau)=\|W_\cT\|_2^2\, , && g(\tau)=\|W1_{[\theta,2\theta]}\|_2^2\, , 
\end{align}
and
\begin{align}
&F(\tau)=\int^{\tau}_{\tau-1} f(\tau')d\tau', && G(\tau)=\int^{\tau}_{\tau-1}  g(\tau')d\tau'.
\end{align}
Then, we obtain
\begin{align}
\frac{d}{d\tau}F(\tau)=\int^\tau_{\tau-1}\frac{d}{d\tau'}f(\tau')\, d\tau' \leq -c|\tau|F(\tau)+C|\tau|^{-1}G(\tau)+ C |\tau|^{-1}\|W \|_{C^2_{\exp}(\cT)}^2(\tau).
\end{align}
We rewrite this as
\begin{equation}
\frac{d}{d\tau} \left[ e^{-\frac{c\tau^2}{2}}F(\tau)\right] \leq e^{-\frac{c\tau^2}{2}}\left( C|\tau|^{-1}G(\tau)+C|\tau|^{-1}\|W \|_{C^2_{\exp}(\cT)}^2(\tau)\right).
\end{equation}
Observing also that thanks to \eqref{rough_mu_decay} the functions $F(\tau)$ and $G(\tau)$ converge (exponentially fast) to zero as $\tau\to -\infty$, we thus infer that
\begin{align}
e^{-\frac{c\tau^2}{2}}F(\tau) &\leq C\int_{-\infty}^{\tau} e^{-\frac{c\tau'^2}{2}}  \left(|\tau'|^{-1}G(\tau')+|\tau'|^{-1}\|W \|_{C^2_{\exp}(\cT)}^2(\tau')\right)d\tau'\nonumber\\
& \leq C\left(\int_{-\infty}^\tau |\tau'| e^{-\frac{c\tau'^2}{2}}d\tau'\right) \left(\sup_{\tau'\leq \tau}|\tau'|^{-2} G(\tau')+ |\tau'|^{-2} \|W \|_{C^2_{\exp}(\cT)}^2(\tau) \right)\\
& \leq Ce^{-\frac{c\tau^2}{2}} \left( |\tau|^{-\frac{3}{2}}\sup_{\tau'\leq \tau} |\tau'|^{-\frac{1}{2}}G(\tau')+|\tau|^{-2} \|W \|_{C^2_{\exp}(\cT)}^2(\tau)\right)\, .\nonumber
\end{align}
Hence, we conclude that 
\begin{equation}
|\tau|^{-\frac{1}{2}}F(\tau)\leq C|\tau|^{-2}\sup_{\tau'\leq \tau}|\tau'|^{-\frac{1}{2}}G(\tau')+C|\tau|^{-5/2}\|W \|_{C^2_{\exp}(\cT)}^2(\tau)\, ,
\end{equation}
from which 
\begin{equation}
\|W_\cT\|_{2,\infty}(\tau)\leq \frac{C}{|\tau|}\left(\|W 1_{[\theta,2\theta]}\|_{2,\infty}(\tau)+ \|W \|_{C^2_{\exp}(\cT)}(\tau)\right)
\end{equation}
readily follows. This finishes the proof of the proposition.
\end{proof}

\bigskip

\subsection{Decay estimate}
The goal of this subsection is to prove the following estimate:

\begin{proposition}[decay estimate]\label{main_decay}
There exist $\kappa>0$ and $\tau_{\ast}>-\infty$ with the following significance. If $M^1$ and $M^2$ are are $\kappa$-quadratic from time $\tau_0 \leq \tau_{\ast}$, and if  $w_\cC=v_1 \varphi_\cC(v_1)-v_2\varphi_\cC(v_2)$ satisfies
\begin{equation}\label{ass_neut}
\mathfrak{p}_{+}w_\cC(\tau_0)=0\quad \textrm{and}\quad \mathfrak{p}_{0}w_\cC(\tau_0)=0,
\end{equation}
then
\begin{equation}
\|w_\cC\|_{\cD,\infty} +\| W_\cT \|_{2,\infty} \leq   C\left( \|w\|_{C^2_{\exp}(\cC)}+\|W\|_{C^2_{\exp}(\cT)} \right).
\end{equation}
\end{proposition}

To show this we will adapt the argument from \cite[Section 8]{ADS2} to our setting. To begin with, combining the energy estimates from the previous subsections, and observing also that the norms in the transition region are equivalent similarly as in \cite[Lemma 8.1]{ADS2},  we obtain:

\begin{lemma}[coercivity estimate]\label{prop_coercivity}
For every $\eps>0$ there exist $\kappa>0$ and $\tau_{\ast}>-\infty$  such that if $M^1$ and $M^2$ are $\kappa$-quadratic form time $\tau_0 \leq {\tau}_{\ast}$, and if the spectral condition \eqref{ass_neut} holds, then
\begin{equation}
\| w_\cC-\mathfrak{p}_0 w_\cC \|_{\cD,\infty}+\| W_\cT \|_{2,\infty}  \leq \eps \| \mathfrak{p}_0 w_\cC  \|_{\cD,\infty}+ C\left( \|w\|_{C^2_{\exp}(\cC)}+\|W\|_{C^2_{\exp}(\cT)} \right)\, .
\end{equation}
\end{lemma}

\begin{proof}
First, using Corollary \ref{cor_unique_asympt} (uniform sharp asymptotics) and arguing similarly as in \cite[proof of Lemma 8.1]{ADS2} we see that for $\kappa>0$ small enough and $\tau$ negative enough we get
\begin{equation}\label{eq_equiv_of_norms}
C(\theta)^{-1}\| W(\tau) 1_{[\theta,2\theta]} \|_{2}\leq \| w(\tau)\, 1_{\{\theta\leq v_1(\tau)\leq 2\theta\}} \|_{\fH}\leq  C(\theta)\| W(\tau) 1_{[\theta,2\theta]} \|_{2}\, .
\end{equation}
Now, recall that by Proposition \ref{prop_int_tip} (energy estimate in tip region) we have
\begin{equation}\label{eq_WT-WD_compatible}
\|W_\cT\|_{2,\infty}(\tau)\leq \frac{C}{|\tau|}\left(\|W 1_{[\theta,2\theta]}\|_{2,\infty}(\tau)+\|W\|_{C^2_{\exp}(\cT)}(\tau)\right)\, .
\end{equation}
Together with \eqref{eq_equiv_of_norms} and the observation that $\varphi_\cC(v_1(\cdot,\tau))\equiv \varphi_\cC(v_2(\cdot,\tau))\equiv 1$ when $\theta \leq v_1 \leq 2\theta$ by Corollary \ref{cor_unique_asympt} (uniform sharp asymptotics) provided $\kappa$ is small enough and $\tau$ is negative enough, this yields
\begin{equation}
\|W_\cT\|_{2,\infty}\leq  \eps\left( \| w_\cC   \|_{\mathcal{D},\infty}+ \|W\|_{C^2_{\exp}(\cT)}\right)\, .
\end{equation}
Next, recall that by Proposition \ref{prop_int_cyl} (energy estimate in the cylindrical region) we have
\begin{equation}
\|w_\cC-\mathfrak{p}_0 w_\cC \|_{\mathcal{D},\infty}\leq  \eps \left( \|  w_\cC  \|_{\cD,\infty} + \| w\, 1_{\{\theta/2\leq v_1\leq \theta\}} \|_{\fH,\infty}\right)
+C \|w\|_{C^2_{\exp}(\cC)}\, .
\end{equation}
Using again \eqref{eq_equiv_of_norms}, this time with $\theta$ replaced by $\theta/2$,
and the fact that $\varphi_\cT(v)\equiv 1$ for $v\leq \theta$, this yields
\begin{equation}
\|w_\cC-\mathfrak{p}_0 w_\cC \|_{\mathcal{D},\infty}\leq  \eps \left( \|  w_\cC  \|_{\cD,\infty} + C \| W_{\cT}\|_{2,\infty}\right)
+C \|w\|_{C^2_{\exp}(\cC)}\, .
\end{equation}
Finally, by the triangle inequality we clearly have
\begin{equation}
\|w_\cC \|_{\mathcal{D},\infty}\leq \|w_\cC-\mathfrak{p}_0 w_\cC \|_{\mathcal{D},\infty}+\|\mathfrak{p}_0 w_\cC \|_{\mathcal{D},\infty}\, .
\end{equation}
Combining the above inequalities, choosing $\tau_\ast$ negative enough, and replacing $C\varepsilon$ by $\varepsilon$, the assertion follows.
\end{proof}

We can now establish the decay estimate:

\begin{proof}[{Proof of Proposition \ref{main_decay}}]
In light of Lemma \ref{prop_coercivity} (coercivity estimate) our task boils down to controlling the expansion coefficient
\begin{equation}
a(\tau):=\langle w_\cC(\tau),\psi_0\rangle_{\mathfrak{H}}\, ,
\end{equation}
where  $\psi_0 = c (y^2-2)$ with $c=\| y^2-2\|_{\mathfrak{H}}^{-1}$. To this end, recall from equation \eqref{ev_trunc_diff} that  $w_\cC$ evolves by
\begin{align}
(\partial_\tau -\mathfrak{L}) w_\cC=\mathcal{E}[w_\cC]+\overline{\mathcal{E}}[w,\varphi_\cC(v_1)]+J+K+e^{\tau}\varphi_\cC(v_1)\mathcal{F}[w]\, .
\end{align}
Using also that $\mathfrak{L}\psi_0=0$, this implies
\begin{equation}
\frac{d}{d\tau}a(\tau)=\left\langle \mathcal{E}[w_\cC]+\overline{\mathcal{E}}[w,\varphi_\cC]+J+K+e^{\tau}\varphi_\cC\mathcal{F}[w],\psi_0\right\rangle_{\mathfrak{H}}\, .
\end{equation}
Since $\langle \psi_0,\psi_0^2\rangle =8$, we can rewrite this as
\begin{equation}\label{ODE_final}
\frac{d}{d\tau}a(\tau)=\frac{2a(\tau)}{|\tau|}+F(\tau),
\end{equation}
where
\begin{equation}\label{eq_error_F}
F(\tau):=\left\langle\mathcal{E}[w_\cC]-\frac{a(\tau)}{4|\tau|}\psi_0^2 ,\psi_0\right\rangle_{\mathfrak{H}}+  \left\langle \overline{\mathcal{E}}[w,\varphi_\cC],\psi_0\right\rangle_{\mathfrak{H}} +\left\langle J+K+e^{\tau}\varphi_\cC(v_1)\mathcal{F}[w], \psi_0\right\rangle_{\mathfrak{H}}.
\end{equation}
Solving the ODE \eqref{ODE_final}, and using that $a(\tau_0)=0$ thanks to the spectral condition \eqref{ass_neut}, we obtain
\begin{equation}\label{eq_solution_a}
a(\tau)=-\frac{1}{\tau^2}\int_{\tau}^{\tau_0} F(\sigma)\sigma^2\, d\sigma.
\end{equation}
In the following, we use the notation
\begin{equation}
A(\tau):= \sup_{\tau'\leq \tau}\left( \int_{\tau'-1}^{\tau'} a(\sigma)^2d\sigma \right)^{1/2}\, .
\end{equation}

\begin{claim}
For every $\varepsilon>0$, there exist $\kappa>0$ and ${\tau}_{\ast}>-\infty$, such that assuming $\kappa$-quadraticity at some $\tau_0\leq \tau_\ast$, the estimate 
\begin{equation}
\int_{\tau-1}^\tau |F(\sigma)|\, d\sigma \leq \frac{\varepsilon}{|\tau|} A(\tau_0) + \frac{C}{|\tau|}\left( \|w\|_{C^2_{\exp}(\cC)}+\|W\|_{C^2_{\exp}(\cT)}\right)
\end{equation}
holds for $\tau \leq \tau_0$.
\end{claim}

\begin{proof}[Proof of the claim] We will adapt the argument from \cite[proof of Claim 8.3]{ADS2} to our setting. During the proof we will frequently use the bound
\begin{equation}\label{eq_wC-a_compatible}
\|w_\cC-\mathfrak{p}_0 w_\cC\|_{\mathcal{D},\infty}+\| W_\cT \|_{2,\infty}\leq  \eps A(\tau_0)+C\left( \|w\|_{C^2_{\exp}(\cC)}+\|W\|_{C^2_{\exp}(\cT)} \right)\, ,
\end{equation}
which follows from Proposition \ref{prop_coercivity} (coercivity estimate).\\

Let us first estimate the terms that are not present in \cite{ADS2}. To this end, let us fix a smooth cutoff function $\chi:\mathbb{R}^+\to [0,1]$ such that
\begin{align}
&\chi(v)=1  \qquad\text{if}\; v\in [\tfrac{9}{16}\theta,\tfrac{15}{16}\theta], && \chi(v)=1\qquad \text{if}\; v\neq [\tfrac{\theta}{2},\theta], 
\end{align}
and  such that $|\chi'|+|\chi''|\leq C(\theta)$. Observe that since $\textrm{spt}(\varphi_\cC')\subset[\tfrac{5}{8}\theta,\tfrac{7}{8}\theta]$ we have
\begin{equation}
J=\chi(v_1)J,\qquad K=\chi(v_1)K\, .
\end{equation}
Hence, using Lemma \ref{lemma_int.cyl.J} (estimate for J) and Lemma \ref{lemma_int.cyl.K} (estimate for K) we can estimate
\begin{align}
\left|\left\langle J(\tau)+K(\tau), \psi_0\right\rangle_{\mathfrak{H}}\right| &\leq \left( \| J(\tau)\|_{\mathcal{D}^\ast} +\| K(\tau)\|_{\mathcal{D}^\ast}\right)  \|\, \chi(v_1(\tau))\psi_0\|_{\mathcal{D}}\nonumber\\
& \leq \left( \eps \| w(\tau) 1_{\{\theta/2 \leq v_1\leq \theta \}} \|_{\mathfrak{H}} +\| v_2\varphi_{\cC}'(v_1)e^{\tau} \mathcal{F}[w]\|_{\mathcal{D}^\ast}\right)  \|\, \chi(v_1(\tau))\psi_0\|_{\mathcal{D}}\, .
\end{align}
Now, considering the support of  $y\mapsto \chi(v_1(y,\tau))$ and using Corollary \ref{cor_unique_asympt} (uniform sharp asymptotics) we see that for $\tau$ negative enough we get the Gaussian tail estimate
\begin{equation}
\|\, \chi(v_1(\tau))\psi_0\|_{\mathcal{D}} \leq e^{\tau/4}\, .
\end{equation}
Also, similarly as in \eqref{eq_equiv_of_norms} by the equivalence of norms in the transition region  we have
\begin{equation}
 \| w(\tau)\, 1_{\{\theta/2\leq v_1\leq \theta\}} \|_{\fH}\leq  C\| W_\cT  \|_{2}\, .
\end{equation}
Moreover, using Lemma \ref{lemma_int.cyl.F} (estimate for nonlinear error) we can estimate
\begin{align}
\|v_2\varphi_{\cC}'(v_1) e^\tau \mathcal{F}[w]\|_{\mathcal{D}^*,\infty}(\tau)+ \|\varphi_\cC(v_1)e^{\tau}\mathcal{F}[w]\|_{\mathcal{D}^*,\infty}(\tau) \leq \frac{C}{|\tau|} \| w\|_{C^2_{\exp}(\mathcal{C})}(\tau).
\end{align}
Combining the above inequalities, and remembering also \eqref{eq_wC-a_compatible}, we infer that
\begin{equation}
\int_{\tau-1}^\tau \left| \left\langle (J+K+e^{\tau}\varphi_\cC(v_1)\mathcal{F}[w])(\sigma), \psi_0\right\rangle_{\mathfrak{H}} \right| \, d\sigma \leq \frac{\varepsilon}{|\tau|} A(\tau_0) + \frac{C}{|\tau|}\left( \|w\|_{C^2_{\exp}(\cC)}+\|W\|_{C^2_{\exp}(\cT)}\right)\, .
\end{equation}

\bigskip

Next, arguing similarly as above, and using also \eqref{eq_error_commute}, we see that
\begin{align}
\left|\left\langle \overline{\mathcal{E}}[w,\varphi_\cC](\tau), \psi_0\right\rangle_{\mathfrak{H}}\right|\leq 
\|  \overline{\mathcal{E}}[w,\varphi_\cC](\tau)\|_{\mathcal{D}^\ast}   \|\, \chi(v_1(\tau))\psi_0\|_{\mathcal{D}}\leq C\| W_\cT  \|_{2}e^{\tau/4}\, ,
\end{align}
and consequently, remembering again \eqref{eq_wC-a_compatible}, that
\begin{equation}
\int_{\tau-1}^\tau \left| \left\langle  \overline{\mathcal{E}}[w,\varphi_\cC](\sigma), \psi_0\right\rangle_{\mathfrak{H}} \right| \, d\sigma \leq \frac{\varepsilon}{|\tau|} A(\tau_0) + \frac{C}{|\tau|}\left( \|w\|_{C^2_{\exp}(\cC)}+\|W\|_{C^2_{\exp}(\cT)}\right)\, .
\end{equation}

\bigskip

Finally, let us estimate the first term on the right hand side of \eqref{eq_error_F}. Broadly speaking, we argue similarly as in \cite[proof of Claim 8.3]{ADS2}. However, since there are quite many technical tweaks (and also to fix some minor glitches in the quoted proof) let us provide full details. Recall from \eqref{eq_def.E[w]} that
\begin{equation}\label{eq_E[w_c]}
\mathcal{E}[w_\cC]-\frac{a(\tau)}{4|\tau|}\psi_0^2 =\left(\frac{2-v_1v_2}{2v_1v_2}w_\cC-\frac{a(\tau)}{4|\tau|}\psi_0^2\right)-\frac{(v_{1,y}+v_{2,y})v_{2,yy}}{(1+v_{1,y}^2)(1+v_{2,y}^2)}(w_\cC)_y-\frac{v_{1,y}^2}{1+ v_{1,y}^2}(w_\cC)_{yy},
\end{equation}
The inner product of the first term on the right hand side with $\psi_0$ can be estimated by 
\begin{align}\label{est_third_term2_F}
&\left|\left\langle \frac{2-v_1v_2}{2v_1v_2}w_\cC-\frac{a(\tau)}{4|\tau|}\psi_0^2,\psi_0\right\rangle_{\mathfrak{H}} \right| 
\leq \left|\left\langle \frac{2-v_1v_2}{2v_1v_2}\varphi_{\cC}(v_1) (w_\cC-a(\tau)\psi_0),\psi_0\right\rangle_{\mathfrak{H}} \right| \nonumber\\
&\qquad +|a(\tau)|\left|\left\langle\frac{2-v_1v_2}{2v_1v_2} \varphi_{\cC}(v_1)-\frac{1}{4|\tau|}\psi_0,\psi_0^2\right\rangle_{\mathfrak{H}} \right|+\left|\left\langle \frac{2-v_1v_2}{2v_1v_2}(1-\varphi_{\cC}(v_1))w_\cC,\psi_0\right\rangle_{\mathfrak{H}} \right|\, . 
\end{align}
To estimate the first term on the right hand side of \eqref{est_third_term2_F} we write
\begin{align}
&\left|\left\langle \frac{2-v_1v_2}{2v_1v_2}\varphi_{\cC}(w_\cC-a(\tau)\psi_0),\psi_0\right\rangle_{\mathfrak{H}} \right| \leq \left|\left\langle \frac{(\sqrt{2}-v_1)(\sqrt{2}+v_1)}{2v_1v_2}\varphi_{\cC}(w_\cC-a(\tau)\psi_0),\psi_0\right\rangle_{\mathfrak{H}} \right|\nonumber\\
& \qquad +\left|\left\langle \frac{v_1-\sqrt{2}}{2v_2}\varphi_{\cC}(w_\cC-a(\tau)\psi_0),\psi_0\right\rangle_{\mathfrak{H}} \right|+\left|\left\langle\frac{\sqrt{2}-v_2}{2v_2} \varphi_{\cC}(w_\cC-a(\tau)\psi_0),\psi_0\right\rangle_{\mathfrak{H}} \right|.
\end{align}
Using this decomposition, and observing that $v_i\geq \theta/2$ on the support of $\varphi_{\cC}(v_1) (w_\cC-a(\tau)\psi_0)$ by Corollary \ref{cor_unique_asympt} (uniform sharp asymptotics), we can estimate
\begin{align}
&\left|\left\langle \frac{2-v_1v_2}{2v_1v_2}\varphi_{\cC}(v_1)(w_\cC-a(\tau)\psi_0),\psi_0\right\rangle_{\mathfrak{H}} \right| \nonumber\\
&\qquad\qquad\leq C \sum_{i=1}^2\left\| (\sqrt{2}-v_i)\varphi_\cC(v_1)\sqrt{|\psi_0|}  \right\|_{\mathfrak{H}} \left\| (w_\cC-a(\tau)\psi_0)\sqrt{|\psi_0|}   \right\|_{\mathfrak{H}} \\
&\qquad\qquad\leq C\sum_{i=1}^2\left\|(\sqrt{2}-v_i)\varphi_{\cC}(v_1)\right\|_{\cD}\left\|(w_\cC-a(\tau)\psi_0)\right\|_{\cD},
\end{align}
where in the last step we used the weighted Sobolev inequality. Now, since the $v_i$ are $\kappa$-quadratic at time $\tau_0$, we have
\begin{equation}\label{diff_norm_small}
\left\|\sqrt{2}-v_i\right\|_{\fH}\leq \frac{C}{|\tau|}
\end{equation}
Furthermore, arguing similarly as in the proof of Claim \ref{second_evlove_claim} (error estimate) we see that
\begin{equation}
\left\|v_{i,y} 1_{C_\tau}\right\|_{\fH}\leq \frac{C}{|\tau|}\, .
\end{equation}
This yields
\begin{equation}
\left\|(\sqrt{2}-v_i)\varphi_{\cC}(v_1)\right\|_{\cD} \leq \frac{C}{|\tau|}.
\end{equation}
 Together with Lemma \ref{prop_coercivity} (coercivity estimate), remembering also that $a\psi_0=\mathfrak{p}_0 w_\cC$,  this implies
\begin{equation}
\int_{\tau-1}^{\tau}\left|\left\langle \frac{2-v_1v_2}{2v_1v_2}\varphi_{\cC}(v_1)(w_\cC-a(\sigma)\psi_0),\psi_0\right\rangle_{\mathfrak{H}} \right|d\sigma \leq \frac{\varepsilon}{|\tau|} A(\tau_0) + \frac{C}{|\tau|}\left( \|w\|_{C^2_{\exp}(\cC)}+\|W\|_{C^2_{\exp}(\cT)}\right).
\end{equation}
Next, similarly as in \cite[Equation (8.20) and (8.21)]{ADS2} we can estimate the contribution from the second term of the right hand side of \eqref{est_third_term2_F} by
\begin{align}
\int_{\tau-1}^{\tau}|a(\sigma)|\left|\left\langle \frac{2-v_1v_2}{2v_1v_2}\varphi_{\cC}(v_1)-\frac{\psi_0}{4|\tau|},\psi_0^2\right\rangle_{\mathfrak{H}}\right|d\sigma<\frac{\varepsilon}{|\tau|}A(\tau_0),
\end{align}
provided $\kappa$ is small enough and $\tau \leq \tau_0 \leq \tau_{\ast}$, with $\tau_\ast$ negative enough.\\
Furthermore, considering the support of $(1-\varphi_{\cC}(v_1))w_{\cC}$ we can estimate the third term of the right hand side of \eqref{est_third_term2_F} by
\begin{align}
\left|\left\langle \frac{2-v_1v_2}{2v_1v_2}(1-\varphi_{\cC}(v_1))w_\cC,\psi_0\right\rangle_{\mathfrak{H}} \right| 
\leq Ce^{\tau/4} \left\| w(\tau) 1_{\{\theta/2\leq v_1\leq \theta \}} \right\|_{\fH}\, .
\end{align}
Together with the equivalence of norms in the transition region, and \eqref{eq_wC-a_compatible}, this yields
\begin{equation}
\int_{\tau-1}^\tau \left|\left\langle \frac{2-v_1v_2}{2v_1v_2}(1-\varphi_{\cC}(v_1))w_\cC,\psi_0\right\rangle_{\mathfrak{H}} \right| \, d\sigma \leq \frac{\varepsilon}{|\tau|} A(\tau_0) + \frac{C}{|\tau|}\left( \|w\|_{C^2_{\exp}(\cC)}+\|W\|_{C^2_{\exp}(\cT)}\right)\, .
\end{equation}

\bigskip

To finish, similarly as in \cite[Equation (8.23) and (8.25)]{ADS2} we get
\begin{equation}
\int_{\tau-1}^{\tau}\left|\left\langle \frac{(v_{1,y}+v_{2,y})v_{2,yy}}{(1+v_{1,y}^2)(1+v_{2,y}^2)}(w_\cC)_y,\psi_0\right\rangle_{\mathfrak{H}}\right|\leq\frac{\varepsilon}{|\tau|}\|w_\cC\|_{\cD,\infty}\, ,
\end{equation}
and
\begin{equation}
\int_{\tau-1}^{\tau}\left|\left\langle \frac{v_{1,y}^2}{ 1+v_{1,y}^2}(w_\cC)_{yy},\psi_0\right\rangle_{\mathfrak{H}}\right|\leq \frac{\varepsilon}{|\tau|}\|w_\cC\|_{\cD,\infty}\, ,
\end{equation}
provided $\kappa$ is small enough and $\tau \leq \tau_0 \leq \tau_{\ast}$, with $\tau_\ast$ negative enough. 
Together with \eqref{eq_wC-a_compatible} this shows that
\begin{multline}
\int_{\tau-1}^{\tau}\left|\left\langle \frac{(v_{1,y}+v_{2,y})v_{2,yy}}{(1+v_{1,y}^2)(1+v_{2,y}^2)}(w_\cC)_y,\psi_0\right\rangle_{\mathfrak{H}}\right|+\left|\left\langle \frac{v_{1,y}^2}{ 1+v_{1,y}^2}(w_\cC)_{yy},\psi_0\right\rangle_{\mathfrak{H}}\right|\\
\leq \frac{\varepsilon}{|\tau|} A(\tau_0) + \frac{C}{|\tau|}\left( \|w\|_{C^2_{\exp}(\cC)}+\|W\|_{C^2_{\exp}(\cT)}\right).
\end{multline}
Combining the above inequalities establishes the claim.
\end{proof}

Now, using the claim we can estimate 
\begin{align}
\left| \int^{\tau_0}_\tau F(\sigma)\sigma^2d\sigma\right|&\leq \sum^{\lceil\tau_0\rceil}_{j=\lfloor\tau\rfloor}\int^j_{j-1}|F(\sigma)|\sigma^2 d\sigma \nonumber\\
& \leq \sum^{\lceil\tau_0\rceil}_{j=\lfloor\tau\rfloor} \frac{(|j|+1)^2}{|j|} \left(\eps A(\tau_0) + C( \|w\|_{C^2_{\exp}(\cC)}+\|W\|_{C^2_{\exp}(\cT)} )\right)\\
&\leq |\tau|^2\left(\eps A(\tau_0)+C( \|w\|_{C^2_{\exp}(\cC)}+\|W\|_{C^2_{\exp}(\cT)} )\right)\, .\nonumber
\end{align}
Remembering \eqref{eq_solution_a}, this shows that
\begin{equation}\label{a_small_always}
|a(\tau)|\leq \eps A(\tau_0)+C( \|w\|_{C^2_{\exp}(\cC)}+\|W\|_{C^2_{\exp}(\cT)} )
\end{equation}
for $\tau\leq \tau_0$. Choosing $\varepsilon =1/2$ this implies
\begin{equation}
A(\tau_0)\leq C\left( \|w\|_{C^2_{\exp}(\cC)}+\|W\|_{C^2_{\exp}(\cT)} \right)\, .
\end{equation}
Combining this with Lemma \ref{prop_coercivity} (coercivity estimate) we conclude that
\begin{equation}
\|w_\cC\|_{\cD,\infty} +\| W_\cT \|_{2,\infty} \leq   C\left( \|w\|_{C^2_{\exp}(\cC)}+\|W\|_{C^2_{\exp}(\cT)} \right)\, .
\end{equation}\
This finshes the proof of the proposition.
\end{proof}

\bigskip

\subsection{Interior estimates in the cylindrical region}
In this subsection, we establish interior $C^{2}$-estimates in the cylindrical region starting from bounds for the Gaussian $L^2$-norm.  Specifically, we will first prove the following weighted $L^\infty$-estimate.

\begin{proposition}[$L^\infty$-estimate in cylindrical region]\label{cor:C0_cylinder}
There exist $\kappa>0$, $\tau_{\ast}>-\infty$ and $C<\infty$, such that  whenever $M^1$ and $M^2$ are $\kappa$-quadratic at time $\tau_0\leq {\tau}_{\ast}$, then for all $\tau\leq\tau_0-1$ we have
\begin{equation}
\sup_{\tau' \leq \tau}e^{\frac{\tau'}{4}}\sup\left\{| w(y,\tau')|:   v_1(y,\tau')\geq \tfrac{8}{9}\theta \right\} \leq C \|w_\cC\|_{\fH,\infty}(\tau+1)\, .
\end{equation}
\end{proposition}

And then we prove the following $C^2$-estimate.\footnote{The supremum in the $C^2$-estimate is taken over a somewhat larger spatial region than in the $L^\infty$-estimate, but this does not cause problems for applications since we will establish corresponding estimates in the tip region in the next subsection.}

\begin{proposition}[$C^2$-estimate in cylindrical region]\label{cor:Schauder_cylinder}
There exist $\kappa>0$ and $\tau_{\ast}>0$,  such that  whenever $M^1$ and $M^2$ are $\kappa$-quadratic at time $\tau_0\leq {\tau}_{\ast}$, then for all $\tau \leq \tau_0-1$ we have 
\begin{equation}
\|w\|_{C^2|C_\tau} \leq e^{-\frac{\tau}{100}}\sup\left\{| w(y,\tau')|: \tau-1\leq \tau'\leq \tau+\tfrac{1}{{100}}, v_1(y,\tau')\geq \tfrac12 \theta\right\}\, .
\end{equation}
\end{proposition}

We recall that we use the notation
\begin{equation}\label{recall_c2_restricted}
\|w \|_{C^2|C_\tau}:=\sup_{y\in C_\tau } \big(|w|+|w_y|+|w_{\tau}|+ |w_{yy}|+|w_{y\tau}|+|w_{\tau\tau}|\big)\, ,
\end{equation}
where the time $\tau$-slice of the cylindrical region is defined by
\begin{equation}
C_\tau=\left\{y: v_1(y,\tau)\geq \tfrac58 \theta \,\textrm{ or }\, v_2(y,\tau)\geq \tfrac58\theta\right\}\, .
\end{equation}

Loosely speaking, both estimates follow from standard parabolic estimates for the mean curvature flow near a bubble-sheet. However, since some care is needed to scale and convert estimates for the bubble-sheet function to estimates for the profile function of the level sets, we provide the details.

\bigskip

To bring the translator equation back into parabolic form, we define 
\begin{equation}
\tilde{V}(x,s,t)=V(x,s+t),
\end{equation}
and consider the difference
\begin{equation}
V^D(x,s,t)=\tilde{V}_1(x,s,t)-\tilde{V}_2(x,s,t).
\end{equation}
Then, using Proposition \ref{prop_evol_profile} (evolution equation for profile function) we see that
\begin{align}\label{eq_V^D_evolution}
V^D_t=&\tfrac{1+\tilde{V}_{1,x}^2}{1+\tilde{V}_{1,x}^2+\tilde{V}_{1,s}^2}V^D_{ss}+\tfrac{1+\tilde{V}_{1,s}^2}{1+\tilde{V}_{1,x}^2+\tilde{V}_{1,s}^2}V^D_{xx}-\tfrac{2\tilde{V}_{1,x}\tilde{V}_{1,s}}{1+\tilde{V}_{1,x}^2+\tilde{V}_{1,s}^2}V^D_{xs}+\tfrac{1}{\tilde{V}_1\tilde{V}_2}V^D\nonumber\\
&+ \tfrac{(\tilde{V}_{1,x}+\tilde{V}_{2,x})(\tilde{V}_{2,ss}-\tilde{V}_{2,t}+1/\tilde{V}_2)-2\tilde{V}_{1,s}\tilde{V}_{2,xs}}{1+\tilde{V}_{1,x}^2+\tilde{V}_{1,s}^2} V^D_x+ \tfrac{ (\tilde{V}_{1,s}+\tilde{V}_{2,s})(\tilde{V}_{2,xx}-\tilde{V}_{2,t}+1/\tilde{V}_2)-2\tilde{V}_{2,x} \tilde{V}_{2,xs}}{1+\tilde{V}_{1,x}^2+\tilde{V}_{1,s}^2} V^D_s.
\end{align}
Let us introduce some notation for weighted parabolic H\"older norms. Given $\alpha \in (0,1)$, a nonnegative integer $k$, and a region $U$, we set
\begin{equation}
[f]^W_{k;U}=\sup_{(x,s,t)\in U}\sup_{i+j+2m= k} |t|^{\frac{k-1}{2}} \left| \partial_x^i \partial_s^j\partial_t^mf(x,s,t)\right|\, ,
\end{equation}
and
\begin{equation}
[f]^W_{k,\alpha;U}=\sup_{X,X'\in U }\;\sup_{i+j+2m= k}\left|\tfrac12 (t+t')\right|^{\frac{k+\alpha- 1}{2}}\frac{ |\partial_x^i \partial_s^j\partial_t^mf(X)- \partial_x^i \partial_s^j\partial_t^mf(X')|}{|d(X,X')|^\alpha}\, ,
\end{equation}
where for $X=(x,s,t)$ and $X'=(x',s',t')$ we work with the parabolic distance
\begin{equation}
d(X,X')=\sqrt{|x-x'|^2+|s-s'|^2+|t-t'|}.
\end{equation}
Then, we can define weighted $C^{k,\alpha}_W$ norms by
\begin{align}
 \|f\|_{C^{k,\alpha}_W	(U)}=\|f\|_{C^k_W(U)}+  [f]^W_{k,\alpha;U}\, ,\qquad \textrm{where} \qquad \|f\|_{C^k_W(U)}=\sum_{m=0}^k [f]^W_{m;U}\, .
\end{align}
Moreover, we consider the parabolic cube $Q_r$ given by
\begin{equation}
Q_r(x',s',t')=\{(x,s,t): |x-x'|\leq r,|s-s'| \leq r , t'-r^2\leq t\leq t' \}.
\end{equation}

\begin{lemma}[interior estimates in cylindrical region]\label{prop:interior_estimates_cylinder}
There exist $\kappa>0$ and ${\tau}_{\ast}>-\infty$, as well as a constant $C<\infty$,  such that whenever $M_1$ and $M_2$ are $\kappa$-quadratic at time $\tau_0 \leq {\tau}_{\ast}$, then
\begin{equation}
\sup_{Q_{\lambda/2}(x',0,t')}|V^D|\leq   \frac{C}{\lambda^{2}}\left( \int_{Q_{\lambda}(x',0,t')} (V^D)^2\, dx\, ds\, dt\right)^{\frac{1}{2}},
\end{equation}
and 
\begin{equation}
\|V^D\|_{C^{4,\alpha}_W(Q_{\lambda/2}(x',0,t'))}\leq C \|V^D\|_{C^{0}_W(Q_{\lambda}(x',0,t'))}\, ,
\end{equation}
hold if $\lambda^{-1} x'\in C_{-\log(-t')}$ and $t' \leq -e^{-\tau_0}$,   where $\lambda=(-t')^{\frac{1}{2}}$.
\end{lemma}

We note that to control the second time derivative $w_{\tau\tau}$ in \eqref{recall_c2_restricted} we need the parabolic $C^4$-norm.

\begin{proof}
We consider the rescaling
\begin{equation}
\hat V^D(\hat x_1,\hat x_2,\hat t)= \frac{1}{\lambda}V^D(x,s,t),\quad \textrm{where}\quad 
(\hat x_1,\hat x_2,\hat t)=\left(\frac{x-x'}{\lambda},\frac{s}{\lambda},\frac{t-t'}{\lambda^2}\right)\, .
\end{equation}
The evolution equation \eqref{eq_V^D_evolution} and  Lemma \ref{lemma_derivatives} (derivative estimates) imply
\begin{equation}
\tfrac{\partial}{\partial \hat t}  \hat V^D (\hat x,\hat t)=a_{ij}(\hat x,\hat t)\tfrac{\partial^2}{\partial \hat x_i \partial \hat x_j}\hat V^D (\hat x,\hat t)+b_{i}(\hat x,\hat t)\tfrac{\partial}{\partial \hat x_i}\hat V^D (\hat x,\hat t)+c(\hat x,\hat t)\hat V^D (\hat x,\hat t),
\end{equation}
where for sufficiently large $-\tau_*$ the smooth functions $a_{ij},b_i,c$ satisfy
\begin{equation}
\sum_{i,j}\|a_{ij}\|_{C^{2,\alpha}(Q_1(0))}+\sum_i\|b_i\|_{C^{2,\alpha}(Q_1(0))}+\|c\|_{C^{2,\alpha}(Q_1(0))}\leq C,\;\;\;\;\;\;\;a_{ij}\xi^i\xi^j \geq C^{-1}|\xi|^2\, .
\end{equation} 
Therefore, standard interior $L^\infty$-estimates yield
\begin{equation}
\sup_{Q_{1/2}(0)}|\hat V^D|\leq C \left( \int_{Q_{1}(0)}(\hat{V}^D)^2\, d\hat{x}\, d\hat{s}\, d\hat{t}\right)^{\frac{1}{2}}\, ,
\end{equation}
and standard interior Schauder estimates yield
\begin{equation}
\|\hat V^D\|_{C^{4,\alpha}(Q_{1/2}(0))}\leq C \|\hat V^D\|_{C^{0}(Q_{1}(0))} \, ,
\end{equation}
These imply the desired results.
\end{proof}

\bigskip

We can now establish the two estimates stated at the beginning of this subsection:

\begin{proof}[Proof of Proposition \ref{cor:C0_cylinder}]

For any point $(x',0,t')$ with $(-t')^{-1/2} x'  \in C_{-\log(-t')}$ denote the corresponding point in the renormalized flow by $(y',\tau')$. The $L^2$-norm of $V^D$ is related to norms of $w$ by 
\begin{align}
\int_{Q_{\sqrt{-t'}}(x',0,t')} (V^D )^2 \, dx\, ds\, dt&=\int_{-\sqrt{|t'|}}^{\sqrt{|t'|}}\int_{2t'}^{t'}\int_{x'-\sqrt{|t'|}}^{x'+\sqrt{|t'|}}(V^D)^2(x,s,t)\, dx\, dt\, ds\nonumber\\
&\leq  C|t'|^{1/2}\int_{2t'-\sqrt{|t'|}}^{t'+\sqrt{|t'|}}\int_{x'-\sqrt{|t'|}}^{x'+\sqrt{|t'|}}(V_1(x,r)-V_2(x,r))^2\, dx\, dr \nonumber\\
&\leq C|t'|^{1/2} \int_{\tau'-2}^{\tau'+1}\int_{y'-2}^{y'+2} e^{\tfrac{-5\tau}{2}}w^2(y,\tau)\, dy\, d\tau\\
& \leq C |t'|^{7/2} \int_{\tau'-2}^{\tau'+1}\int_{y'-2}^{y'+2} w^2(y,\tau)e^{\tfrac{\tau}{2}}\, dy\, d\tau, \nonumber
\end{align}
whenever $t' \leq -10$. Now, for $\kappa$ sufficiently small and ${\tau}_{\ast}$ sufficiently negative, assuming $\kappa$-quadraticity at time $\tau_0\leq\tau_\ast$, by Corollary \ref{cor_unique_asympt} (uniform sharp asymptotics) if $v_1(y',\tau') \geq\tfrac{8}{9}\theta$ and $\tau' = -\log(-t') \leq \tau_0-1$, then in the region under consideration we have
\begin{equation}
\varphi_{\cC}(v_i)=1 \qquad \textrm{and}\qquad y^2\leq -2\tau \, .
\end{equation}
Therefore, we obtain 
\begin{align*}
&\int_{Q_{\sqrt{-t'}}(x',0,t')} (V^D )^2\, dx\, ds\, dt \leq C |t'|^{7/2} \int_{\tau'-2}^{\tau'+1}\int w_{\cC}^2(y,\tau)e^{-\tfrac{y^2}{4}} \, dy\, d\tau \leq C|t'|^{7/2}\|w_{\cC}\|^2_{\fH,\infty}(\tau'+1).
\end{align*}
Thus, Lemma \ref{prop:interior_estimates_cylinder} (interior estimates in cylindrical region)  yields
\begin{equation}
\sqrt{|t'|}w(y',\tau') \leq  C|t'|^{3/4}\|w_{\cC}\|_{\fH,\infty}(\tau'+1),
\end{equation}
from which the result follows.
\end{proof}

\begin{proof}[Proof of Proposition \ref{cor:Schauder_cylinder}]
Using the definition of $ V^D$ we see that
\begin{equation}
|t|^{-\frac{1}{2}}  V^D(x,0,t)=w(|t|^{-\frac{1}{2}}x,-\log(-t)).
\end{equation}
As in the proof of Proposition \ref{prop_evol_profile} (evolution equation for profile function), this yields
\begin{align}
&  V^D_x=w_y , && |t|^{\frac{1}{2}}  V^D_{xx} =w_{yy}, && |t|  V^D_{xt}= w_{y\tau}+\frac{y}{2}w_{yy} ,
\end{align}
and
\begin{align}
&|t|^{\frac{1}{2}}V^D_t= w_\tau+\frac{y}{2}w_y-\frac{w}{2} , &&|t|^{\frac{3}{2}}V^D_{tt}= w_{\tau\tau}+yw_{\tau y}+\frac{y^2}{4}w_{yy}+\frac{y}{4}w_y-\frac{w}{4}.
\end{align}
Therefore,  Lemma \ref{prop:interior_estimates_cylinder}  (interior estimates in cylindrical region), taking also into account that
\begin{equation}
|y|^2 \leq 2(1+o(1))|\tau|
\end{equation}
by Corollary \ref{cor_unique_asympt} (uniform sharp asymptotics), as well as the elementary inequality
\begin{equation}
-\log(-t+|t|^{1/2})\leq -\log(-t)+\tfrac{1}{100}
\end{equation}
for $t\ll 0$, implies the desired result. 
\end{proof}

\bigskip

\subsection{Interior estimates in the tip region}
In this subsection, we establish interior $C^{2}$-estimates in the tip region starting from bounds for the Gaussian $L^2$-norm.  Specifically, we will first prove the following weighted $L^\infty$-estimate:
\begin{proposition}[{$L^\infty$-estimate in tip region}]\label{cor:C0_tip}
There exist $\kappa>0$ and ${\tau}_{\ast}>-\infty$  such that  if $M^1$ and $M^2$ are $\kappa$-quadratic at time $\tau_0\leq\tau_\ast$, then for any $\tau\leq \tau_0-1$ we have
\begin{equation}
\sup_{\tau' \leq \tau}e^{\frac{26}{100}\tau'}\sup\left\{| W(v,\tau')|: v\leq \tfrac{9}{10} \theta\right\} \leq \|W_\cT\|_{2,\infty}(\tau+1)\, .
\end{equation}
\end{proposition}

And then we prove the following $C^2$-estimate:

\begin{proposition}[{$C^2$-estimate in tip region}]\label{cor:Schauder_tip}
There exist $\kappa>0$ and ${\tau}_{\ast}>-\infty$  such that  if $M^1$ and $M^2$ are $\kappa$-quadratic at time $\tau_0\leq\tau_\ast$, then for any $\tau\leq \tau_0-1$ we have
\begin{align}
\|W\|_{C^2|T_\tau} \leq e^{-\frac{\tau}{100}}  \sup\left\{| W(v,\tau')|: \tau-1\leq \tau'\leq \tau+\tfrac{1}{100},v \leq 3\theta\right\}\, .
\end{align}
\end{proposition}

We recall that we use the notation
\begin{equation}
\|W \|_{C^2|T_\tau}:=\sup_{v\leq 2\theta } \big(|W|+|W_v|+|W_{\tau}|+ |W_{vv}|+|W_{v\tau}|+|W_{\tau\tau}|\big)\, .
\end{equation}
To put the translator equation into convenient form for establishing these interior estimates, for $h\gg 1$ we define positive functions $X^k$ by
\begin{equation}
(h,X^k(-h,{x}_3,{x}_4),{x}_3,{x}_4)\in M^k\, ,
\end{equation}
and then define $\tilde{X}^k$ by
\begin{equation}
\tilde{X}^k(x_1,x_2,x_3,t)=X^k(x_1+t,x_2,x_3).
\end{equation}
The functions $\tilde{X}^k$ satisfy the parabolic equation
\begin{equation}\label{eq_X^D_evolution}
\tilde{X}^k_t=\left( \delta_{ij} -\tfrac{\tilde{X}^k_i\tilde{X}^k_j}{1+|\nabla \tilde{X}^k|^2}\right)\tilde{X}^k_{ij}\, ,
\end{equation}
and the difference $X^D=\tilde{X}^1-\tilde{X}^2$ evolves by
\begin{equation}
X^D_t=\left(\delta_{ij}-\frac{\tilde{X}^1_{i}\tilde{X}^1_{j}}{1+|D\tilde{X}^1|^2}\right)X^D_{ij}+\frac{\tilde{X}^1_{i}\tilde{X}^2_{ij}X^D_j+\tilde{X}^2_{i}\tilde{X}^2_{ij}X^D_j}{1+|D\tilde{X}^1|^2}-\frac{\tilde{X}^2_{i}\tilde{X}^2_{j}\tilde{X}^2_{ij}(\tilde{X}^1_{k}+\tilde{X}^2_{k})X^D_k}{(1+|D\tilde{X}^1|^2)(1+|D\tilde{X}^2|^2)}\, .
\end{equation}
Similarly as before, writing $X=(x_1,x_2,x_3,t)=(x,t)$,  we set
\begin{equation}
[f]^W_{k;U}=\sup_{X\in U}\sup_{i+j+\ell+2m= k} |t|^{\frac{k-1}{2}} \left| \partial_{x_1}^i \partial_{x_2}^j  \partial_{x_3}^{\ell} \partial_t^mf(X)\right|\, ,
\end{equation}
and
\begin{equation}
[f]^W_{k,\alpha;U}=\sup_{X,X'\in U }\;\sup_{i+j+\ell+2m= k}\left|\tfrac12 (t+t')\right|^{\frac{k+\alpha- 1}{2}}\frac{ |\partial_{x_1}^i \partial_{x_2}^j  \partial_{x_3}^{\ell} \partial_t^mf(X)- \partial_{x_1}^i \partial_{x_2}^j  \partial_{x_3}^{\ell} \partial_t^mf(X')|}{|d(X,X')|^\alpha}\, ,
\end{equation}
where 
\begin{equation}
d(X,X')=\sqrt{|x-x'|^2+|t-t'|}\, .
\end{equation}
Then, we work with the weighted $C^{k,\alpha}_W$ norm given by
\begin{align}
 \|f\|_{C^{k,\alpha}_W	(U)}=\|f\|_{C^k_W(U)}+  [f]^W_{k,\alpha;U}\, \qquad \textrm{where} \qquad \|f\|_{C^k_W(U)}=\sum_{m=0}^k [f]^W_{m;U}\, .
\end{align}
Finally, we denote the parabolic cube $Q_r$ by
\begin{equation}
Q_r(x',t')=\left\{(x,t): t'-r^2\leq t\leq t',  |\langle x-x' , e_i\rangle|\leq r \;\; \text{for each} \;i=1,2,3  \right\}.
\end{equation}

\begin{lemma}[interior estimates in soliton region]\label{prop_interior_estimates_solition}
There exist constants  $\kappa>0$, ${\tau}_{\ast}>-\infty$, and $C<\infty$ with the following significance. If  $M^1$ and $M^2$ are $\kappa$-quadratic at time $\tau_0\leq {\tau}_{\ast}$, then
\begin{equation}
\sup_{Q_{\lambda/2}(x',t')}|X^D|\leq   \frac{C}{\lambda^{\frac{5}{2}}}\left( \int_{Q_{\lambda}(x',t')} (X^D)^2 \, dx\, dt\right)^{\frac{1}{2}}\, ,
\end{equation}
and
\begin{equation}
\|X^D\|_{C^{4,\alpha}_W(Q_{\lambda/2}(x',t'))}\leq |\log(-t')|^{10} \|X^D\|_{C^{0}_W(Q_{\lambda}(x',t'))}
\end{equation}
hold whenever $\langle x',e_1\rangle=0$, $|x'|\leq \lambda L $ and $-\log(-t')\leq \tau_0-1$, where $\lambda=|\log(-t')|^{-1/2}(-t')^{1/2}$.
\end{lemma}

\begin{proof}
We consider the rescaled function
\begin{equation}
\hat X^D(\hat x, \hat t)= \frac{1}{\lambda} X^D(x'+\lambda \hat x,t'+\lambda^2 \hat t). 
\end{equation}
Then, the evolution equation \eqref{eq_X^D_evolution} and Corollary \ref{cor_unique_asympt} (uniform sharp asymptotics) imply that there exist constants $C<\infty$ and $\tau_*>-\infty$ such that the smooth functions $a_{ij}$ and $b_i$ defined by 
\begin{equation}
 \hat X^D_t =a_{ij} \hat X^D_{ij}  +b_{i} \hat X^D_i
\end{equation}
satisfy
\begin{equation}
\sum_{i,j}\|a_{ij}\|_{C^{2,\alpha}(Q_1(0))}+\sum_i\|b_i\|_{C^{2,\alpha}(Q_1(0))}\leq C,\qquad   a_{ij}\xi_i\xi_j \geq C^{-1}|\xi|^2
\end{equation} 
at the points and times under consideration. Thus, standard interior $L^\infty$-estimates yield
\begin{equation}
\sup_{Q_{1/2}(0)}|\hat X^D|\leq C \|\hat X^D\|_{L^2(Q_1(0))}\, ,
\end{equation}
and standard interior Schauder estimates yield
\begin{equation}
\|\hat X^D\|_{C^{4,\alpha}(Q_{1/2}(0))}\leq C \|\hat X^D\|_{C^{0}(Q_{1}(0))}\, .
\end{equation}
These estimates imply the assertion.
\end{proof}

\begin{lemma}[interior estimates in collar region]\label{prop_interior_estimates_collar}
There exist $\kappa>0$ and ${\tau}_{\ast}>-\infty$, as well as positive integers $p,q$ and a constant $C<\infty$ with the following significance. If $M^1$ and $M^2$ are $\kappa$-quadratic at time $\tau_0 \leq {\tau}_{\ast}$, then 
\begin{equation}
\sup_{Q_{\lambda/2}(x',t')}|X^D|\leq   \frac{C}{\lambda^{\frac{5}{2}}}\left( \int_{Q_{\lambda}(x',t')} (X^D)^2 \, dx\, dt\right)^{\frac{1}{2}}\, ,
\end{equation}
and
\begin{equation}
\|X^D\|_{C^{4}_W(Q_{\lambda/2}(x',t'))}\leq |\log(-t')|^{p} \|X^D\|_{C^{0}_W(Q_{\lambda}(x',t'))} 
\end{equation}
hold whenever $\langle x',e_1\rangle=0$, $L|\log(-t')|^{-\frac{1}{2}} \leq  (-t)^{-1/2}|x'| \leq 4 \theta $, and $-\log(-t') \leq \tau_0-1$,  where $\lambda=|\log(-t')|^{-q}(-t)^{1/2}$.
\end{lemma}

\begin{proof}
To take care of the degenerating ellipticity, we set
\begin{equation}
\rho=|DX_1(x',t')|\, ,
\end{equation}
and consider the anisotropic rescaling
\begin{equation}
\hat X^D(\hat x, \hat t)= \frac{1}{\lambda} X^D(x_1'+(1+\rho)\lambda \hat x_1,x_2'+\lambda\hat x_2,x_3'+\lambda\hat x_3,t'+\lambda^2 \hat t)\, . 
\end{equation}
Using the evolution equation \eqref{eq_X^D_evolution} we see that
\begin{equation}
 \hat X^D_t =a_{ij} \hat X^D_{ij}  +b_{i} \hat X^D_i\, ,
\end{equation}
for some coefficients $a_{ij}, b_i$. Now, by definition of $X^k$ we have
\begin{equation}
X^k(t,V_k(x,t)\cos\theta,V_k(x,t)\sin\theta)=x\, ,
\end{equation}
and Corollary \ref{cor_unique_asympt} (uniform sharp asymptotics) shows that
\begin{equation}
(V_k)^{-1}\leq C(L)\sqrt{|t|^{-1}  \log |t|}
\end{equation}
 for some constant $C(L)<\infty$ depending only on $L$. Therefore, 
applying  Lemma \ref{lemma_derivatives} (derivative estimates), we infer that there are a positive integer $p'$ and a constant $C<\infty$ such that
\begin{equation}
\|\tilde{X}^k\|_{C^{4,\alpha}_W(Q_{\lambda}(x',t'))}\leq C |\log(-t')|^{p'}
\end{equation}
for sufficiently large $\log(-t')$ and sufficiently large $q$. For the coefficients of the rescaled difference, possibly after increasing $q$, this yields
\begin{equation}
\sum_{i,j}\|a_{ij}-\delta_{ij}\|_{C^{2,\alpha}(Q_1(0))}+\sum_i\|b_i\|_{C^{ 2,\alpha} (Q_1(0))}\leq C,\qquad \sum_{i,j}a_{ij}\xi_i \xi_j  \geq C^{-1}|\xi|^2\, .
\end{equation} 
Thus, similarly as in the proof of the previous lemma, interior $L^\infty$-estimates and interior Schauder estimates yield the desired result.
\end{proof}

As a final preparation, we need the following bound for the weight function:

\begin{lemma}[density bound]\label{prop:density.bound}
There exist $\kappa>0$ and ${\tau}_{\ast}>-\infty$ with the following significance. Assuming $\kappa$-quadraticity at time $\tau_0 \leq {\tau}_{\ast}$, for $\tau\leq \tau_0$ and $v\leq 5\theta$ we have
\begin{equation}
e^{\mu(v,\tau)} \geq v e^{\frac{51}{100}\tau}\, .
\end{equation} 
\end{lemma}

\begin{proof}
By definition of the weight function  we have
\begin{align}
\mu_v &= -\frac{1}{4}\zeta(v) (Y^2)_v +(1-\zeta(v))\frac{1+Y_v^2}{v}\nonumber\\
&=\frac{1-\zeta(v)}{v} + \left( \frac{(1-\zeta(v))Y_v}{2vY} - \frac{\zeta(v)}{4}   \right) (Y^2)_v\, .
\end{align}
Since $Y(0,\tau)=(\sqrt{2}+o(1))|\tau|^{\frac{1}{2}}$ by Corollary \ref{cor_unique_asympt} (uniform sharp asymptotics) and $|Y_v/v|\leq \sqrt{\tau}$ by Proposition \ref{lemma_tip_estimates} (tip estimates), possibly after decreasing $\theta$ and $\tau_\ast$, for $\tau\leq\tau_0$ and $v\leq 5\theta$ this yields
\begin{equation}
\mu_v \leq \frac{1}{v} -\frac38  (Y^2)_v \, .
\end{equation}
Therefore,  we get
\begin{align}\label{mu_bound_eq}
\mu(v,\tau)=-\frac14 Y^2(\theta,\tau)-\int^\theta_v \mu_{v'}(v',\tau)\, dv'\geq \log \left( \frac{v}{\theta}\right)+ \frac{1}{8} Y^2(\theta,\tau)-\frac{3}{8}Y^2(v,\tau)\, .
\nonumber
\end{align}
Using again Corollary \ref{cor_unique_asympt} (uniform sharp asymptotics), this implies the assertion.
\end{proof}

\bigskip

We can now prove the propositions stated at the beginning of this subsection.

\begin{proof}[Proof of Proposition \ref{cor:C0_tip}]By definition of $X^D$ we have
\begin{equation}
X^D(x_1,x_2,x_3,t)=e^{-\frac{\tau}{2}}W(v,\tau)\, ,
\end{equation}
where
\begin{equation}
 \tau=-\log(-x_1-t) \qquad\textrm{ and }\qquad v=e^{\frac{\tau}{2}}\left(x_2^2+x_3^2\right)^{\frac{1}{2}}\, .
\end{equation}
Notice also that
\begin{equation}
\int\int dx_2\, dx_3=2\pi e^{-\tau} \int v\, dv\, .
\end{equation}
Suppose that $x'=(0,x_2',x_3')$ and $t'=-e^{-\tau'}$ satisfy $|x'|e^{\frac{\tau'}{2}}\leq \tfrac{9}{10}\theta $ and $\tau'\leq \tau_0-1$. For any $R \leq |\tau'|^{-\frac{1}{2}}e^{-\frac{\tau'}{2}}$ we can compute
\begin{align}
\frac{1}{R}   \int_{Q_{R}(x',t')} (X^D )^2 \, dx\, dt = &\frac{1}{R} \int_{-R}^R \int_{t'-R^2}^{t'}  \int_{x_2'-R}^{x_2'+R}\int_{x_3'-R}^{x_3'+R} (X^D )^2 \, dx_2\, dx_3 \, dt\, dx_1\nonumber\\
   \leq&\,C e^{-3\tau'}\sup_{|x_1|\leq R} \int_{-\log(-t'-x_1)-1}^{-\log(-t'-x_1)} \int_{0}^{\theta} W^2(v,\tau)  v\,dv \,d\tau .
\end{align}
Using also Lemma \ref{prop:density.bound} (density bound) and the fact that $\phi_{\mathcal{T}}=1$ when $v\leq \theta$, this yields
\begin{equation}
\frac{1}{R}   \int_{Q_{R}(x',t')} (X^D )^2 \,dx\, dt\leq C|\tau'|^{\frac{1}{2}}e^{-\frac{351}{100}\tau'}\|W_\cT\|_{2,\infty}^2(\tau'+1)\, .
\end{equation}
Combining the above inequality with Lemma \ref{prop_interior_estimates_solition} (interior estimates in soliton region) and Lemma \ref{prop_interior_estimates_collar} (interior estimates in collar region) implies that there exists some positive integer $m$ such that
\begin{equation}
|W(v',\tau')|^2=e^{\tau'}|X^D(x',t')|^2\leq  |\tau'|^{m}e^{-\frac{51}{100}\tau'}\|W_\cT\|_{2,\infty}^2(\tau'+1)\, .
\end{equation}
Namely,
\begin{equation}
e^{\frac{26}{100}\tau'}|W(v',\tau')|\leq \|W_\cT\|_{2,\infty}(\tau'+1)
\end{equation}
holds for $\tau'\leq \tau_0-1$ and $v'\leq \tfrac{9}{10} \theta$. This proves the proposition.
\end{proof}

\begin{proof}[Proof of Proposition \ref{cor:Schauder_tip}]
By definition of $X^D$ we have
\begin{equation}
W(v,\tau)=e^{\frac{\tau}{2}}X^D(0,e^{-\frac{\tau}{2}}v\cos\theta,e^{-\frac{\tau}{2}}v\sin\theta,-e^{-\tau})\, .
\end{equation}
This implies
\begin{align}
&W_v=X^D_r\, , && W_{vv}=e^{-\frac{\tau}{2}}X^D_{rr}\, , && W_{v\tau}=-\tfrac{v}{2}e^{-\frac{\tau}{2}}X^D_{rr}+e^{-\tau}X^D_{rt}\, ,
\end{align}
where
\begin{equation}
X^D_{r}=\cos\theta\, X^D_2+\sin\theta\, X^D_3\, , \qquad X^D_{rr}=\cos^2\theta\, X^D_{22}+2 \cos\theta\sin\theta\, X^D_{23}+\sin^2\theta\, X^D_{33}\, .
\end{equation}
Similarly, we can compute
\begin{equation}
W_\tau =\tfrac{1}{2}e^{\frac{\tau}{2}}X^D-\tfrac{v}{2}X^D_r+ e^{-\frac{\tau}{2}}X^D_t\, ,
\end{equation}
and
\begin{equation}
W_{\tau\tau}= \tfrac{1}{4}e^{\frac{\tau}{2}}X^D -\tfrac{v}{4}X^D_r
+\tfrac{v^2}{4}e^{-\frac{\tau}{2}}X^D_{rr} - ve^{-\tau} X^D_{rt}+e^{-\frac{3\tau}{2}}X^D_{tt}\, .
\end{equation}
Hence, applying Lemma \ref{prop_interior_estimates_solition} (interior estimates in soliton region) and Lemma \ref{prop_interior_estimates_collar} (interior estimates in collar region), we obtain the desired result.
\end{proof}

\bigskip

\subsection{Conclusion of the proof}

In this subsection we conclude the proof of Theorem \ref{thm:uniqueness_eccentricity} (spectral uniqueness theorem).

Denoting the level sets by $\Sigma^i_h=M^i\cap \{x_1=h\}$, we consider their Hausdorff distance
\begin{equation}
D(h):=d_{\mathrm{Hausdorff}}\left(\Sigma^1_h\, ,\, \Sigma^2_h\right)\, .
\end{equation}

\begin{proposition}[Hausdorff-estimate]\label{lemma:C0_entire}
There exist $\kappa>0$ and ${\tau}_{\ast}>-\infty$ such that if $M^1$ and $M^2$ are $\kappa$-quadratic at time $\tau_0 \leq {\tau}_{\ast}$, then for every $\tau\leq \tau_0-1$ we have
\begin{equation}
\sup_{h\geq e^{-\tau}}h^{-\frac{76}{100}}D(h) \leq \|w_\cC\|_{\fH,\infty}(\tau+1)+\|W_\mathcal{T}\|_{2,\infty}(\tau+1)\, .
\end{equation}
\end{proposition}

\begin{proof}
Setting $\tau_h=-\log h$, by definition of the Hausdorff distance we always have
\begin{equation}
h^{-\frac{1}{2}}D(h)\leq \max\Big(\sup\left\{| w(y,\tau_h)|:    v_1(y,\tau_h)\geq \tfrac{8}{9}\theta \right\},\sup\left\{| W(v,\tau_h)|:	 v \leq \tfrac{9}{10} \theta\right\}\Big)\, .
\end{equation}
Now, by Proposition \ref{cor:C0_cylinder} ({$L^\infty$-estimate in cylindrical region}) and Proposition \ref{cor:C0_tip} ({$L^\infty$-estimate in tip region}) if the solution are $\kappa$-quadratic from time $\tau_0 \leq {\tau}_{\ast}$, and $\tau_h \leq \tau_0-1$, then we can estimate
\begin{equation}
e^{\frac{1}{4}\tau_h}\sup\left\{| w(y,\tau_h)|:    v_1(y,\tau_h)\geq \tfrac{8}{9}\theta \right\}\leq C \|w_\cC\|_{\fH,\infty}(\tau_h+1)\, ,
\end{equation}
and
\begin{equation}
e^{\frac{26}{100}\tau_h}\sup\left\{| w(y,\tau_h)|:    v_1(y,\tau_h)\geq \tfrac{8}{9}\theta \right\}\leq  \|W_\mathcal{T}\|_{2,\infty}(\tau_h+1)\, .
\end{equation}
This implies the assertion.
\end{proof}

\begin{proposition}[$C^2$-estimate]\label{lemma:Schauder_entire}
There exist $\kappa>0$ and ${\tau}_{\ast}>-\infty$ such that if $M^1$ and $M^2$ are $\kappa$-quadratic at time $\tau_0 \leq {\tau}_{\ast}$, then for every $\tau\leq \tau_0-1$ we have
\begin{equation}
\|w\|_{C^2|C_\tau} +\|W\|_{C^2|T_\tau} \leq e^{\frac{48}{100}\tau} \sup \left\{ D(h):\tau-1\leq -\log h \leq \tau+\tfrac{1}{100}\right\}\, .
\end{equation}
\end{proposition}

\begin{proof}
Set $\tau_h=-\log h$. Observe first that by Corollary \ref{cor_unique_asympt} (uniform sharp asymptotics), for $\kappa$ sufficiently small and ${\tau}_{\ast}$ sufficiently negative, and $\tau_h \leq  \tau_0-1$, we have
\begin{equation}\label{sup_small_w}
\sup\left\{|w(y,\tau_h)|:v_1(y,\tau_h) \geq \theta/2\right\}\leq 2 h^{-\frac{1}{2}}D(h)\, .
\end{equation}
Our next goal is to show that
\begin{equation}\label{eq_next_goal}
\sup\left\{|W(v,\tau_h)| : v\leq 3\theta\right\} \leq 2 (\log h)^{\frac{1}{2}} h^{-\frac{1}{2}} D(h)\, .
\end{equation}
To this end, considering $A_i\in \Sigma^i_h\cap \{x_3= h^{1/2}v,x_4=0\}$ we have
\begin{equation}
d(A_1,A_2)=\left|\langle A_1-A_2,e_2\rangle \right|=h^{1/2}\left|W(v,\tau_h)\right|.
\end{equation}
Let $B\in \Sigma^2_h$ be a point such that $d(B,A_1)=d(\Sigma^2_h\, ,\, A_1)$, and observe that $x_3(B)\geq 0$ and $x_4(B)=0$.
We may assume $A_1\neq A_2$. Then, $A_1\neq B$. So, applying the sine law to triangle spanned by $A_1$, $A_2$, $B$ yields
\begin{equation}\label{W_dh_1}
d(A_1,A_2)=\frac{\sin\angle A_1BA_2}{\sin\angle A_1A_2B}{d(A_1,B)} \leq \frac{1}{\sin\angle A_1A_2B}D(h). 
\end{equation}
On the other hand, denoting by $\vec{T}$ the tangent vector to $\mathrm{graph}(Y_2)$ at $v$, by convexity we have
\begin{equation}\label{W_dh_2}
\sin\angle A_1A_2B \geq \sin\angle(\vec{T},-e_2)=\frac{1}{\sqrt{1+ Y_{2,v}(v,\tau_h)^2}}.
\end{equation}
Moreover, applying Proposition \ref{lemma_tip_estimates} (tip estimates), with $\theta$  replaced by $\frac{3}{2}\theta$, we can estimate
\begin{equation}
Y_{2,v}(v,\tau_h)^2\leq |\log h|v^2 \, .
\end{equation}
Combining the above formulas proves the estimate \eqref{eq_next_goal}.

Having established the sup-bounds \eqref{sup_small_w} and \eqref{eq_next_goal} we can now apply Proposition \ref{cor:Schauder_cylinder} ($C^2$-estimate in cylindrical region) and Proposition \ref{cor:Schauder_tip} ($C^2$-estimate in  tip region) to conclude that
\begin{equation}
\|w\|_{C^2|C_\tau} +\|W\|_{C^2|T_\tau} \leq e^{-\tfrac{1}{100}\tau}\sup_{\tau-1\leq\tau_h\leq \tau+\frac{1}{100}} 2(\log h)^{1/2} h^{-1/2} D(h)\, .
\end{equation}
This implies the assertion.
\end{proof}

To proceed, we denote by $K^i\subset \mathbb{R}^4$ the convex hull of $M^i$, and set
\begin{equation}
K^i_h:= K^i\cap \{ x_1 = h\}\, .
\end{equation}
Then, by the comparison principle for translators we have the implication
\begin{equation}
K^1_{h'} \subseteq K^2_{h'}\quad \Rightarrow \quad K^1_{h} \subseteq K^2_{h}\, \textrm{ for all }\,  h \leq h'\ ,
\end{equation}
and similarly with $K^1$ and $K^2$ interchanged.

\begin{lemma}[almost congruent levels]\label{lemma:congruency_level_set}
There exist $\kappa>0$ and ${\tau}_{\ast}>-\infty$ such that if $M^1,M^2$ are $\kappa$-quadratic from time $\tau_0 \leq {\tau}_{\ast}$, and if $h'\geq e^{-\tau_0}$ satisfies $D(h') \leq \tfrac{1}{30}h'^{1/2}$, then we have
\begin{equation}
D(h) \leq 10\,(\log h')^{\frac{1}{2}}D(h')
\end{equation}
for all $h\in [h'/e^2,h']$.
\end{lemma}

\begin{proof}
Note first that by Corollary \ref{cor_unique_asympt} (uniform sharp asymptotics), provided $\kappa$ and ${\tau}_{\ast}$ are chosen  appropriately, the mean curvature $H=\langle \nu,e_1\rangle $ of $M^i$ satisfies
\begin{equation}\label{eq_gradient_translator}
 \frac{0.99}{\sqrt{2h}}\leq \langle \nu,e_1\rangle \leq \frac{1.01}{\sqrt{2}}\sqrt{\frac{\log h}{h}},
\end{equation}
for $h \geq e^{-\tau_0-3}$ (here we observed that the conclusion can be propagated a bit forward in time). Now, if $h' \geq e^{-\tau_0}$ satisfies $D(h') \leq \tfrac{1}{30}h'^{1/2}$, then
\begin{equation}
h'-2\sqrt{h'} D(h') \geq e^{-1} h' \geq e^{-\tau_0-1}\, .
\end{equation}
Hence, the lower bound in \eqref{eq_gradient_translator}  implies
\begin{align}
(K^1+2\sqrt{h'} D(h')e_1)_{h'}\subseteq K^2_{h'}\, .
\end{align}
Thus, by the comparison principle for all $h\leq h'$ we get the inclusion
\begin{align}\label{eq_inclus_ls}
 (K^1+2\sqrt{h'} D(h')e_1)_{h}\subseteq K^2_{h}\, ,
\end{align}
and, interchanging the role of $K^1$ and $K^2$, also the inclusion
\begin{align}\label{eq_inclus_ls}
 (K^2+2\sqrt{h'} D(h')e_1)_{h}\subseteq K^1_{h}\, .
\end{align}
On the other hand, if $h \geq h'/e^2$ then $h-2\sqrt{h'} D(h') \geq e^{-\tau_0-3}$, so it follows from the upper bound in \eqref{eq_gradient_translator} that
\begin{equation}\label{eq_motion_bound}
\mathrm{dist}_{\mathrm{Hausdorff}}\Big(\Sigma^i_{h-2\sqrt{h'} D(h')}\, ,\, \Sigma^i_h\Big)\leq 4\,(\log h')^{\frac{1}{2}}D(h'). 
\end{equation}
Hence, we conclude that
\begin{equation}
D(h) \leq 10\,(\log h)^{\frac{1}{2}}D(h')
\end{equation}
for all $h\in [h'/e^2,h']$. This proves the lemma.
\end{proof}

We can now conclude the proof of the spectral uniqueness theorem:

\begin{proof}[Proof of Theorem \ref{thm:uniqueness_eccentricity}]
Let $\kappa$ be small enough and ${\tau}_{\ast}$ negative enough such that all the preceding estimates hold for $M^1$ and $M^2$ that are  $\kappa$-quadratic from time $\tau_0 \leq {\tau}_{\ast}$. Applying Proposition \ref{main_decay} (decay estimate) and Proposition \ref{lemma:C0_entire} (Hausdorff-estimate) we see that
\begin{equation}\label{eq_first_of_proof}
\sup_{h\geq e^{-\tau_0+1}} h^{-\frac{76}{100}}D(h) \leq C\left( \|w\|_{C^2_{\exp}(\cC)}+\|W\|_{C^2_{\exp}(\cT)}\right)\, .
\end{equation}
On the other hand, by definition of our exponentially weighted norms there exists some $\tau'\in (-\infty,\tau_0]$ such that
\begin{equation}\label{eq_C2_maximum_time}
\|w\|_{C^2_{\exp}(\cC)}+\|W\|_{C^2_{\exp}(\cT)}\leq 2e^{\tau'}\left(|\tau'|\|w\|_{C^2|C_{\tau'}} +\|W\|_{C^2|T_{\tau'}} \right)\, ,
\end{equation}
and by Lemma \ref{lemma_derivatives} (derivative estimates) and Lemma \ref{prop_Y.collar.rough} (rough tip estimates), for all $\tau\leq \tau_0$ we have
\begin{equation}
\|w\|_{C^2|C_\tau} +\|W\|_{C^2|T_\tau} \leq |\tau|^{100}\, .
\end{equation}
In particular, for $h':=e^{-\tau'+1}$ this yields
\begin{equation}
D(h')\leq h'^{-\frac{23}{100}}\, .
\end{equation}
Hence, we can safely apply Lemma \ref{lemma:congruency_level_set} (almost congruent levels) to obtain
\begin{equation}
D(h) \leq 10 (\log h')^{1/2} D(h')
\end{equation}
for $ \tau'-1 \leq -\log h \leq \tau'+1$. Together with Proposition \ref{lemma:Schauder_entire} ($C^2$-estimates) it follows that
\begin{equation}
\|w\|_{C^2|C_{\tau'}} +\|W\|_{C^2|T_{\tau'}}\leq h'^{-{\frac{47}{100}}} D(h')\, . 
\end{equation}
In combination with \eqref{eq_first_of_proof} and \eqref{eq_C2_maximum_time} this yields
\begin{equation}
h'^{-\frac{76}{100}}D(h')\leq h'^{-\frac{146}{100}}D(h')\, .
\end{equation}
Since $h'\gg 1$, this implies $D(h')=0$, and consequently
\begin{equation}
\|w\|_{C^2_{\exp}(\cC)}+\|W\|_{C^2_{\exp}(\cT)}=0\, .
\end{equation}
Namely,
\begin{equation}
\Sigma_{h}^1=\Sigma_h^2
\end{equation}
 holds for $h \geq e^{-\tau_0}$. Finally, applying the comparison principle again we conclude that $M^1=M^2$. This finishes the proof of the spectral uniqueness theorem.
\end{proof}

\bigskip

\bibliography{translator}

\bibliographystyle{alpha}

\vspace{10mm}

{\sc Kyeongsu Choi, School of Mathematics, Korea Institute for Advanced Study, 85 Hoegiro, Dongdaemun-gu, Seoul, 02455, South Korea}\\

{\sc Robert Haslhofer, Department of Mathematics, University of Toronto,  40 St George Street, Toronto, ON M5S 2E4, Canada}\\

{\sc Or Hershkovits, Institute of Mathematics, Hebrew University, Givat Ram, Jerusalem, 91904, Israel}\\

\emph{E-mail:} choiks@kias.re.kr, roberth@math.toronto.edu, or.hershkovits@mail.huji.ac.il

\end{document}